\documentclass[11pt, reqno]{amsart}
\usepackage{textcmds}  
\usepackage{amsmath, amssymb, amsfonts, amstext, verbatim, amsthm, mathrsfs}
\usepackage{microtype}
\usepackage[all]{xy}
\usepackage[modulo]{lineno}
\usepackage[usenames]{color}
\usepackage{aliascnt}
\usepackage{enumitem}
\usepackage{xspace}
\usepackage{amsfonts}
\usepackage{amssymb}
\usepackage[centertags]{amsmath}
\usepackage{amsthm}
\usepackage{rotating}
\usepackage[margin=3.25cm]{geometry}
\usepackage{dsfont}
\usepackage{bm}
\usepackage{color}
\usepackage{subfigure}
\usepackage{amsmath}
\usepackage{array}
\usepackage[all]{xy}
\usepackage{euscript}
\usepackage[T1]{fontenc}
\usepackage{mathbbol}
\usepackage{svg}

\usepackage{stmaryrd}

\usepackage{tikz}
\usepackage{tikz-cd} 

\usepackage{graphics,graphicx}  %% comment for "draft" mode w/o pictures

\let\oldtocsection=\tocsection

\let\oldtocsubsection=\tocsubsection

\renewcommand{\tocsection}[2]{\hspace{0em}\oldtocsection{#1}{#2}}
\renewcommand{\tocsubsection}[2]{\hspace{1em}\oldtocsubsection{#1}{#2}}

\usepackage[backref,colorlinks=true,linkcolor=black,citecolor=blue,urlcolor=blue,citebordercolor={0 0 1},urlbordercolor={0 0 1},linkbordercolor={0 0 1}]{hyperref} %needs to be loaded after most things

\tikzset{node distance=3cm, auto}

\makeatletter
\def\@secnumfont{\bfseries}
\def\section{\@startsection{section}{1}%
  \z@{.7\linespacing\@plus\linespacing}{.5\linespacing}%
  {\normalfont\Large\bfseries}}
\def\subsection{\@startsection{subsection}{2}%
  \z@{.5\linespacing\@plus.7\linespacing}{-.5em}%
  {\normalfont\large\bfseries}}
  \def\subsubsection{\@startsection{subsubsection}{3}%
  \z@{.5\linespacing\@plus.7\linespacing}{-.5em}%
  {\normalfont\bfseries}}
\makeatother

\theoremstyle{plain}
\newtheorem{thm}{Theorem}[section]
\newtheorem{lemma}[thm]{Lemma}
\newtheorem{prop}[thm]{Proposition}
\newtheorem{cor}[thm]{Corollary}
\newtheorem{question}[thm]{Question}

\newtheorem{remark}[thm]{Remark}
\newtheorem{addendum}[thm]{Addendum}
\newtheorem{disclaimer}[thm]{Disclaimer}
\newtheorem{assumption}[thm]{Assumption}
\newtheorem{conjecture}[thm]{Conjecture}
\theoremstyle{definition}
\newtheorem{definition}[thm]{Definition}
\newtheorem{example}[thm]{Example}
\theoremstyle{remark}
\numberwithin{equation}{section}

\newcommand{\Z}{\mathbb{Z}}

\newcommand{\R}{\mathbb{R}}
\newcommand{\Q}{\mathbb{Q}}
\newcommand{\C}{\mathbb{C}}
\newcommand{\CP}{\mathbb{CP}}
\newcommand{\D}{\mathbb{D}}
\newcommand{\K}{\mathbb{K}}
\newcommand{\lam}{\lambda}
\newcommand{\la}{\lambda}
\renewcommand{\sc}{\op{SC}}
\newcommand{\eps}{\varepsilon}
\newcommand{\calF}{\mathcal{F}}
\newcommand{\calA}{\mathcal{A}}
\newcommand{\calL}{\mathcal{L}}

\newcommand{\calQ}{\mathcal{Q}}
\renewcommand{\bar}{\mathcal{B}}
\newcommand{\X}{\mathcal{X}}
\newcommand{\A}{\mathcal{A}}
\newcommand{\bdy}{\partial}

\newcommand{\Ai}{\mathcal{A}_\infty}
\newcommand{\Li}{\mathcal{L}_\infty}

\newcommand{\calM}{\mathcal{M}}
\newcommand{\dcalM}{\dashover{\mathcal{M}}}

\newcommand{\wh}{\widehat}
\newcommand{\wt}{\widetilde}

\newcommand{\ovl}{\overline}
\newcommand{\op}[1]{{\operatorname{#1}}}
\newcommand{\e}{\epsilon}
\newcommand{\Sh}{{\op{Sh}}}

\newcommand{\Eval}{{\op{Eval}}}

\newcommand{\sh}{\op{SH}}
\newcommand{\Lam}{\Lambda}
\newcommand{\Lamo}{\Lambda_{\geq 0}}
\newcommand{\chlin}{\op{CH}_{\op{lin}}}
\newcommand{\ch}{\op{CH}}
\newcommand{\lin}{{\op{lin}}}
\newcommand{\ip}{\, \lrcorner \,}
\newcommand{\cz}{{\op{CZ}}}
\newcommand{\ind}{\op{ind}}
\renewcommand{\ll}{\llbracket}
\newcommand{\rr}{\rrbracket}

\newcommand{\cha}{{\op{CHA}}}
\newcommand{\ev}{{\op{ev}}}
\newcommand{\Op}{\mathcal{O}p}
\newcommand{\nil}{\varnothing}

\newcommand{\sss}{\vspace{2.5 mm}}
\renewcommand{\leqq}{\preceq}

\newcommand{\orbset}{\mathfrak{S}}
\newcommand{\sign}{\diamondsuit}
\newcommand{\aug}{\e}
\newcommand{\auglin}{\e_\lin}
\newcommand{\fv}{\mathfrak{h}}
\newcommand{\capac}{\mathfrak{c}}
\newcommand{\dapac}{\mathfrak{d}}
\newcommand{\gapac}{\mathfrak{g}}
\newcommand{\rapac}{\mathfrak{r}}
\newcommand{\bb}{\frak{b}}

\newcommand{\fix}{{\op{fix}}}
\newcommand{\triv}{\Xi}
\newcommand{\m}{\mathfrak{m}}
\newcommand{\cl}{\mathfrak{cl}}
\newcommand{\cllin}{{\mathfrak{cl}_\lin}}
\newcommand{\gw}{\op{GW}}
\renewcommand{\lll}{\Langle}
\newcommand{\rrr}{\Rangle}

\newcommand{\smx}{\oset[-.5ex]{\frown}{\times}}

\newcommand{\sk}{{\op{sk}}}
\newcommand{\T}{\mathcal{T}}
\newcommand{\TT}{\T^{\bullet}} 
\newcommand{\fib}{g}

\newcommand{\hooksymp}{\overset{s}\hookrightarrow}
\newcommand{\NI}{{\noindent}}
\newcommand{\ra}{\rightarrow}
\renewcommand{\tt}{\frak{t}}
\newcommand{\ga}{\gamma}
\newcommand{\stab}{\op{stab}}

\newcommand{\wind}{\op{wind}}
\makeatletter

\newcommand{\oset}[3][0ex]{%
  \mathrel{\mathop{#3}\limits^{
    \vbox to#1{\kern-2\ex@
    \hbox{$\scriptstyle#2$}\vss}}}}

\newcommand{\dashover}[2][\mathop]{#1{\mathpalette\df@over{{\dashfill}{#2}}}}
\newcommand{\fillover}[2][\mathop]{#1{\mathpalette\df@over{{\solidfill}{#2}}}}
\newcommand{\df@over}[2]{\df@@over#1#2}
\newcommand\df@@over[3]{%
  \vbox{
    \offinterlineskip
    \ialign{##\cr
      #2{#1}\cr
      \noalign{\kern1pt}
      $\m@th#1#3$\cr
    }
  }%
}
\newcommand{\dashfill}[1]{%
  \kern-.5pt
  \xleaders\hbox{\kern.5pt\vrule height.4pt width \dash@width{#1}\kern.5pt}\hfill
  \kern-.5pt
}
\newcommand{\dash@width}[1]{%
  \ifx#1\displaystyle
    2pt
  \else
    \ifx#1\textstyle
      1.5pt
    \else
      \ifx#1\scriptstyle
        1.25pt
      \else
        \ifx#1\scriptscriptstyle
          1pt
        \fi
      \fi
    \fi
  \fi
}
\newcommand{\solidfill}[1]{\leaders\hrule\hfill}
\makeatother

\begin{document}

\title{Higher symplectic capacities}

\begin{abstract}
We construct new families of symplectic capacities indexed by certain symmetric polynomials, defined using rational symplectic field theory.
In particular, we introduce a sequence of capacities based on an $\Li$ structure on linearized contact homology and rational curve counts with local tangency constraints.
We prove various structural properties of these capacities and give some preliminary computations which show that they give state of the art symplectic embedding obstructions in basic examples.
\end{abstract}

\author{Kyler Siegel}

\maketitle

\tableofcontents

\section{Introduction}\label{sec:intro}

\subsection{Context}
\subsubsection{Symplectic embeddings}
Given a symplectic manifold $(M^{2n},\omega)$ and a function $H: M \rightarrow \R$,
flowing along the symplectic gradient of $H$ defines a family of symplectomorphisms from $M$ to itself.
Such Hamiltonian dynamical systems provide the geometric basis of classical mechanics and appear in various guises throughout mathematics and physics. 
It has been known since the 19th century that symplectomorphisms preserve the canonical volume form
$\tfrac{1}{n!}\omega^{\wedge n}$, a result known as Liouville's theorem in statistical mechanics.
However, in his seminal 1985 paper \cite{gromov1985pseudo} on pseudoholomorphic curves, Gromov showed that symplectomorphisms are much more subtle and mysterious than volume-preserving maps.
The first proof of concept for this underlying richness is Gromov's celebrated ``nonsqueezing theorem'',
which states that a large ball cannot be symplectically squeezed into a narrow infinite cylinder.
More precisely, for $0 < a' < a$ there is no symplectic embedding $$B^{2n}(a) \hookrightarrow B^2(a') \times \C^{n-1}.$$
Here both sides are equipped with their canonical symplectic structures coming from the restriction of the standard symplectic form $\sum_{i=1}^ndx_i \wedge dy_i$ on $\C^n$, and $B^{2n}(a)$ denotes the ball of area $a$ (i.e. radius $\sqrt{a/\pi}$).
Such an embedding is perfectly allowed by volume considerations alone, so this shows that symplectic geometry imposes fundamental and previously unanticipated constraints. 

Understanding the precise nature of symplectic transformations is still a largely open field of research.
After the nonsqueezing theorem, a natural next toy question is when can one symplectic ellipsoid be embedded into another.
We define the symplectic ellipsoid $E(a_1,...,a_n)$ by 
$$ E(a_1,...,a_n) := \left\{(z_1,...,z_n) \in \C^n \;:\; \sum_{k=1}^n \frac{\pi |z_k|^2}{a_k} \leq 1\right\}.$$
For convenience, we will typically assume $a_1 \leq ... \leq a_n$.
Note that we have the special cases $E(a,...,a) = B^{2n}(a)$ and $E(a,\infty,...,\infty) = B^2(a) \times \C^{n-1}$.
\begin{question}\label{question:ell_emb}
For which $a_1,...,a_n$ and $a_1',...,a_n'$ is there a symplectic embedding of $E(a_1,...,a_n)$ into $E(a_1',...,a_n')$?
\end{question}

\noindent

Since the volume of $E(a_1,...,a_n)$ is given by $\tfrac{1}{n!}a_1...a_n$, the first basic obstruction is $$a_1...a_n \leq a_1'...a_n'.$$
Also, Gromov's nonsqueezing theorem can be rephrased as saying that there is no symplectic embedding
$$E(a,...,a) \hookrightarrow E(a',\infty,...,\infty)$$
unless $a \leq a'$,
and this implies that in general we must have
$$a_1 \leq a_1'.$$

\subsubsection{Symplectic capacities}
For $n = 1$, an embedding $B^2(a_1) \hooksymp B^2(a_1')$ evidently exists if and only if $a_1 \leq a_1'$.
For $n = 2$, the answer to Question \ref{question:ell_emb} is now completely understood, or at least reduced to pure combinatorics.
Before recalling the result, it is useful to formalize symplectic embedding obstructions using the notion of a symplectic capacity.
Roughly, a symplectic capacity $\capac$ is any invariant of symplectic manifolds which takes values in positive real numbers and is monotone with respect to symplectic embeddings.
More precisely, the symplectic capacities considered in this paper will be defined for symplectic manifolds with nonempty convex boundary
(although one could formally extend them to all symplectic manifolds if desired).
 We will also distinguish between ``exact symplectic capacities'', which are monotone only for exact symplectic embeddings, and ``non-exact symplectic capacities'', which are monotone for all symplectic embeddings.\footnote{See Definition~\ref{def:exact_symp_emb} for the precise definition of exact symplectic embedding.} 
 By default, ``capacity'' will mean exact, but we note that this distinction is immaterial e.g. for Liouville domains with $H^1(X;\R) = 0$ (see Remark~\ref{rmk:auto_exact}).
One often also requires a symplectic capacity $\capac$ to scale like area, i.e. $\capac(M,a\omega) = a\capac(M,\omega)$, which rules out symplectic volume in higher dimensions. 
This property will indeed hold for all the capacities discussed in this paper.
Finally, some authors include the normalization condition $\capac(B^{2n}) = \capac(B^2 \times \C^{n-1}) = 1$, but this property will {\em not} hold for most of the capacities considered in this paper.\footnote{In fact, the ``strong Viterbo conjecture'' states that there should be a {\em unique} capacity satisfying this normalization condition, when restricted to convex domains in $\R^{2n}$.}

An important sequence of symplectic capacities $\capac_1^\op{EH},\capac_2^\op{EH},\capac_3^\op{EH},...$ was defined by Ekeland--Hofer in \cite{EH1,EH2} using variational methods. The first Ekeland--Hofer capacity satisfies $$\capac_1^\op{EH}(B^{2n}(r)) = \capac_1^\op{EH}(B^2(r) \times \C^{n-1}) = r,$$
and so in particular recovers Gromov's nonsqueezing theorem. 
More generally, for $k \in \Z_{> 0}$, the $k$th Ekeland--Hofer capacity of the ellipsoid $E(a_1,...,a_n)$ has been computed to be the $k$th smallest element in the infinite matrix $(ia_j\;:\; i \in \Z_{> 0}, 1 \leq j \leq n)$.
For example, the first few values for $E(1,2)$ are $1,2,2,3,4,4,5,...$ and the first few values for $E(1.5,1.5)$ are $1.5,1.5,3,3,4.5,4.5,6,...$,
and hence $c_2^\op{EH}$ rules out the existence of a symplectic embedding $E(1,2) \hookrightarrow E(1.5,1.5)$.

Although the Ekeland--Hofer capacities give various obstructions beyond Gromov's nonsqueezing theorem, the four dimensional ellipsoid embedding problem is controlled by another sequence of symplectic capacities $\capac_1^\op{ECH},\capac_2^\op{ECH},\capac_3^\op{ECH},...$ defined by Hutchings in \cite{Hutchings_quantitative_ECH} using embedded contact homology (ECH).
Unlike the Ekeland--Hofer capacities, these capacities are gauge-theoretic in nature and are only defined in dimension four, but in that dimension they tend to give very strong obstructions.
The $k$th ECH capacity of the ellipsoid $E(a,b)$ has been computed to be the  $(k+1)$st smallest element in the infinite matrix $(ia + jb \; : \; i,j \in \Z_{\geq 0})$.
For example, the first few values for $E(1,2)$ are $1,2,2,3,3,4,4,4,...$
and the first few values for $E(1.5,1.5)$ are $1.5,1.5,3,3,3,4.5,4.5,4.5,4.5,...$, etc.
McDuff proved the following result in \cite{McDuff_Hofer_conjecture} (see the references therein for the history):
\begin{thm}
There is a symplectic embedding of $E(a,b)$ into $E(c,d)$ if and only if $\capac_k^\op{ECH}(E(a,b)) \leq \capac_k^\op{ECH}(E(c,d))$ for all $k \in \Z_{> 0}$.
\end{thm}
Note that up to a scaling factor we can write $E(a,b)$ as $E(1,x)$ for some $x \geq 1$.
In the special case $c = d$ (i.e. the target is a ball), the answer can be summarized in terms of a function $f: [1,\infty) \rightarrow [1,\infty)$,
with $f(x)$ defined to be the infimal $c$ such that there exists a symplectic embedding $E(1,x) \hookrightarrow E(c,c)$.
McDuff--Schlenk \cite{McDuff-Schlenk_embedding} gave an explicit characterization of the function $f(x)$, discovering the following rather subtle features:
\begin{itemize}
\item for $x < \tau^4$, $f(x)$ is piecewise linear, with values determined by ratios of Fibonacci numbers (here $\tau = \frac{1+\sqrt{5}}{2}$ the golden ratio)
\item for $x > (\tfrac{17}{6})^2$, $f(x) = \sqrt{x}$ (i.e. embeddings in this regime are obstructed only by volume)
\item for $\tau^4 < x < (\frac{17}{6})^2$, $f(x)$ alternates between the above two behaviors.
\end{itemize}

\subsubsection{Stabilized symplectic embeddings}\label{subsubsec:stab}

In higher dimensions, Question \ref{question:ell_emb} is still largely open.
One surprising development was Guth's construction \cite{Guth_polydisks} of symplectic embeddings $$E(1,\underbrace{S,...,S}_{n-1}) \hookrightarrow E(R,R,\underbrace{\infty,...,\infty}_{n-2})$$
for some finite $R$ and $S$ arbitrarily large. 
This shows that an infinite amount of squeezing is possible in the second factor, in significant contrast to the first factor.
Hind and Kerman managed to refine Guth's embedding and obtained the following optimal result:
\begin{thm}\cite{Hind-Kerman_new_obstructions, Pelayo-Ngoc_hofer_question}\footnote{
More precisely, it was shown in \cite{Pelayo-Ngoc_hofer_question} that the existence of a symplectic embedding 
$E(1,\underbrace{S,...,S}_{n-1}) \hookrightarrow E(R,R,\underbrace{\infty,...,\infty}_{n-2})$ for all sufficiently large $S$ implies the existence of a symplectic embedding
$E(1,\underbrace{\infty,...,\infty}_{n-1}) \hookrightarrow E(R,R,\underbrace{\infty,...,\infty}_{n-2})$.
In the sequel we will often not distinguish between these two.
}
\label{thm:HK}
There exists a symplectic embedding $$E(1,\underbrace{\infty,...,\infty}_{n-1}) \hookrightarrow E(R,R,\underbrace{\infty,...,\infty}_{n-2})$$ if and only if $R \geq 3$.
\end{thm}
More recent subsequent progress has focused on the following special case of Question \ref{question:ell_emb}:
\begin{question}[``stabilized ellipsoid embedding problem'']\label{question:stab_ell}
For which $a,b,c,d$ is there a symplectic embedding of $E(a,b,\underbrace{\infty,...,\infty}_{n-2})$ into $E(c,d,\underbrace{\infty,...,\infty}_{n-2})$?
\end{question}
Note that any symplectic embedding $E(a,b) \hookrightarrow E(c,d)$ can be stabilized to an embedding
$E(a,b,\infty,...,\infty) \hookrightarrow E(c,d,\infty,...,\infty)$, but the converse need not hold.
Indeed, the volume obstruction disappears after stabilizing, as do many of the ECH obstructions.
For the special case $c = d$, i.e. embeddings of the form
$E(1,x,\infty,...,\infty) \hookrightarrow E(c,c,\infty,...,\infty)$,
notable progress has been made, for instance we have:
\begin{thm}\cite{Hind-Kerman_new_obstructions, CG-Hind_products, Ghost_stairs_stabilize, Mcduff_remark}\label{thm:stab_ell_known_results}
For $1 \leq x \leq \tau^4$, the optimal $c$ is given by $f(x)$, i.e. the optimal embedding comes from the four dimensional case. 
For $x = b_n = \tfrac{\fib_{4n+6}}{\fib_{4n+2}}$ with $\fib_i$ the $i$th Fibonacci number (e.g. $b_0 = 8$, $b_1 = 55/8$, etc),
or for $x = 3d-1$ with $d \in \Z_{> 0}$, the optimal $c$ is given by $\tfrac{3x}{x+1}$.
\end{thm}
Although for $x > \tau^4$ the optimal embedding is not known, Hind \cite{hind2015some} has produced embeddings
\begin{align}\label{eq:hind_emb}
E(1,x) \times \C^{n-2} \hookrightarrow E(\tfrac{3x}{x+1},\tfrac{3x}{x+1}) \times \C^{n-2}  
\end{align}
which are conjectured in \cite{Mcduff_remark} to be optimal.
We mention that in all known cases the answer to Question \ref{question:stab_ell} is independent of $n$ provided $n \geq 3$.

\sss

In this paper we primarily focus on the problem of obstructing high dimensional symplectic embeddings.
The results comprising Theorem \ref{thm:stab_ell_known_results}
have relied on the following basic strategy.
First, given a four-dimensional symplectic embedding $E(1,x) \hookrightarrow E(c,c)$, one finds a rigid punctured pseudoholomorphic curve $C$ in the symplectic completion of the complementary cobordism $E(c,c) \setminus E(1,x)$ (see \S\ref{subsec:geometric_setup} for more on the geometric setup).
The existence of such a curve is proved using embedded contact homology in \cite{CG-Hind_products}, using obstruction bundle gluing in \cite{Mcduff_remark}, and by neck stretching closed curves in $\CP^2$ in \cite{Ghost_stairs_stabilize}.
One then argues that this curve meaningfully extends to the stabilization as a product curve $C \times \{pt\}$.
Assuming this, a deformation argument shows that a similar curve $\wt{C}$ must also exist in the complementary cobordism of any hypothetical symplectic embedding $E(1,x,\infty,...,\infty) \hookrightarrow E(c,c,\infty,...,\infty)$.
Finally, positivity of energy for $\wt{C}$ together with an application of Stokes' theorem gives a lower bound for $c$ in terms of $x$.

The process of stabilizing the curve $C$ cannot always be achieved.
Indeed, for an SFT-type curve $C$ with positive ends asymptotic to Reeb orbits $\gamma_1^+,..,.,\gamma_k^+$ and negative ends asymptotic to Reeb orbits $\gamma^-_1,...,\gamma^-_l$, the Fredholm index is given by
\begin{align}\label{eq:fred_ind_intro}
\ind(C) = (n-3)(2 - 2g - k - l) + \sum_{i=1}^k \cz(\gamma_i^+) - \sum_{j=1}^l\cz(\gamma_j^-) + 2c_1 \cdot [C]
\end{align}
(see \S\ref{subsec:punc_pseudo}).
Since stabilizing by $\C$ has the effect of increasing $n$ by $1$ and also increasing each Conley--Zehnder index $\cz(\gamma)$ by $1$,
the index after stabilizing is
\begin{align}\label{eq:stab_ind_intro}
\ind(C \times \{pt\}) =  2 - 2g - 2l + \ind(C). 
 \end{align}

This means that $C \times \{pt\}$ will typically not be rigid, and might potentially even have negative index, in which case it ought to disappear after a small perturbation.
However, a basic observation from \cite{Hind-Kerman_new_obstructions} is that for genus zero curves with exactly one negative end, the index is unchanged by stabilization, and in fact such curves do generally persist to give higher dimensional obstructions.
One of the main motivations of this paper is to understand whether these obstructions can be understood more systematically in terms of symplectic capacities which are defined in any dimension.

\subsection{Floer theory perspective}\label{subsec:capac_from_HF}

Before discussing our main results, we first motivate and illustrate the concept of ``higher symplectic capacities'' from the perspective of Floer theory, showing how to construct new capacities based on algebraic structures on symplectic cohomology.
The eager reader may skip ahead to the next subsection, where we turn our attention to rational symplectic field theory, which will be our main setting for the bulk of this paper.

Let $X^{2n}$ be a connected Liouville domain, i.e. a compact symplectic manifold with boundary, whose symplectic form $\omega$ admits a global primitive $\lam$ restricting to a positive contact form on $\bdy X$.
We associate to $X$ its symplectic cohomology $\sh(X)$, defined over say $\K = \Q$.
We can also consider the underlying cochain complex $\sc(X)$,
with its natural increasing action filtration
$$ \calF_{\leq a}\sc(X) \subset \calF_{\leq a'}\sc(X),\;\;\;\;\; 0 \leq a < a',$$
with $\bigcup\limits_{a \in \R_{\geq 0}}\calF_{\leq a}\sc(X) = \sc(X)$.
Roughly, as a $\K$-module $\sc(X)$ has one generator for each critical point of a chosen Morse function on $X$, and two generators  
$\hat{\gamma},\check{\gamma}$ for each Reeb orbit $\gamma$ of $\bdy X$,
with differential $\bdy$ counting Floer cylinders.
The Morse cochain complex of $X$ (defined over $\K$) sits inside $\sc(X)$ as the subcomplex of action zero elements: $$C(M) = \calF_{\leq 0}\sc(X) \subset \sc(X).$$
In particular, there is a canonical action zero element $e \in \calF_{\leq 0}\sc(X)$ represented by the minimum of the Morse function, which we will refer to as the unit (even when not discussing product structures).
Given any exact symplectic embedding $X \hookrightarrow X'$, we have a filtration-preserving chain map $\Phi: \sc(X') \rightarrow \sc(X)$ sending unit to unit.

We define a symplectic capacity from this data as follows:
$$ \capac(X) := \inf \{ a\;: \bdy(x) = e \text{ for some } x \in \calF_{\leq a}\sc(X)\}.$$
Moreover, given an exact symplectic embedding $X \hookrightarrow X'$, we have the monotonicity property
$\capac(X) \leq \capac(X')$. 
Indeed, this follows from a simple diagram chase, since
$x \in \calF_{\leq a}\sc(X')$ with $\bdy x = e$ implies $$\bdy(\Phi x) = \Phi(\bdy x) = \Phi(e) = e,$$
with $\Phi(x) \in \calF_{\leq a}\sc(X)$.

One can compute $\capac(B^{2n}(r)) = r$ and $\capac(B^2(R) \times \C^{n-1}) = R$,\footnote{Technically since $B^2(R) \times \C^n$ is not compact, this statement actually holds for a smoothing of $B^2(R) \times B^{2n-2}(S)$ with $S$ sufficiently large.}
and hence the capacity $\capac$ recovers Gromov's nonsqueezing theorem.
However, in order to capture more refined obstructions we need more capacities.
Recently, Gutt--Hutchings \cite{Gutt-Hutchings_capacities} showed that by replacing symplectic cohomology with its $S^1$-equivariant analogue, we can extend $\capac(X)$ to a sequence of capacities $\capac_1^{\op{GH}}(X),\capac_2^{\op{GH}}(X),\capac_3^{\op{GH}},...$ which conjecturally agree with the Ekeland--Hofer capacities.
To explain the construction, recall that $\sc(X)$ naturally has the structure of an $S^1$-complex, which means 
we have a sequence of operations $\Delta_i: \sc(X) \rightarrow \sc(X)$ of degree $1-2i$,\footnote{In the sequel we will discuss $\Z$ gradings assuming they are available, i.e. $2c_1(X) = 0$, or else interpret them in the appropriate quotient of $\Z$.
See \S\ref{subsec:rsft_formalism} for more on our grading conventions.
}
with $\Delta_0 = \bdy$, satisfying the relations $$\sum_{i + j = k} \Delta_i\circ \Delta_j = 0$$
for all $k \in \Z_{\geq 0}$ (see e.g. \cite{Seidel_biased_view, Bourgeois-Oancea_equivariant,ganatracyclic}).
Here $\Delta_1$ descends to the BV operator on $\sh(X)$, while $\Delta_{k \geq 2}$ is a ``higher BV operator'' compensating for the failure of $\Delta_1$ to satisfy $\Delta_1^2 = 0$ on the nose.
Let $\K[u^{-1}]$ be a shorthand for the $\K[u]$-module $\K[u,u^{-1}]/\left( u\K[u]\right)$.
We define $S^1$-equivariant symplectic cochains $\sc_{S^1}$ by $$\sc_{S^1}(X) := \sc(X) \otimes_{\K[u]} \K[u^{-1}],$$
with differential $\bdy_{S^1}$ given by
$$\bdy_{S^1}(x) = \sum_{i=0}^{\infty} \Delta_i(x)u^i$$
for $x \in \sc_{S^1}(X)$.\footnote{There are two common variations on this definition of the $S^1$-equivariant complex: the {\em negative} version, in which one works over $\K\ll u \rr$, and the {\em periodic} version, in which one works over $\K\ll u \rr[u^{-1}]$. The version we use above is akin to taking homotopy orbits of an $S^1$-action.}

There is a natural increasing action filtration on $\sc_{S^1}(X)$, 
and we have canonical elements 
$u^{-k}e \in \sc_{S^1}(X)$
 for $k \in \Z_{\geq 0}$.
We define $\capac_k^{\op{GH}}(X) \in \R_{> 0}$ by
$$ \capac_k^{\op{GH}}(X) := \inf \{ a\;:\; \bdy_{S^1}(x) = u^{-k+1}e\text{ for some } x \in \calF_{\leq a}\sc_{S^1}(X)\}.$$
Similar to before, exact symplectic embeddings $X \hookrightarrow X'$ induce transfer maps $\sc_{S^1}(X') \rightarrow \sc_{S^1}(X)$,
from which it follows that $\capac_k^{\op{GH}}(X) \leq \capac_k^{\op{GH}}(X')$.

To understand why the equivariant theory offers more quantitative data, we can rephrase the construction as follows.
By quotienting out the Morse complex of $X$, we obtain a homology-level long exact sequence
\begin{equation} \label{eq:SH_LES}
\begin{tikzcd}
\cdots & H(X) & \sh(X) & \sh_+(X) & H(X) & \cdots
\arrow[from=1-1,to=1-2]
\arrow[from=1-2,to=1-3]
\arrow[from=1-3,to=1-4]
\arrow[from=1-4,to=1-5,"\delta"]
\arrow[from=1-5,to=1-6]
\end{tikzcd}
\end{equation}
where the map $\delta$ is the connecting homomorphism.
We have the following equivalent homology level description: $$\capac(X) = \inf \{a \;:\; \delta(x) = [e]\text{ for some } x \in \calF_{\leq a}SH_+(X)\}.$$
In fact, we can define more general ``spectral capacities'' by replacing $[e] = 1 \in H^0(X)$ by any class in $H^*(X)$. However, if $X$ is say a star-shaped domain\footnote{By {\em star-shaped domain in $\C^n$} we mean any subdomain $M^{2n}\subset \C^n$ whose boundary is smooth and transverse to the radial vector field. The restriction of the standard Liouville form $\sum_{i=1}^n\tfrac{1}{2}(x_idy_i -y_idx_i)$ makes $M$ into a Liouville domain. Note that all star-shaped domains are diffeomorphic and in fact Liouville deformation equivalent, but typically not symplectomorphic.} in $\C^n$, we have $\sh_+(X) \cong H(X) \cong \K$, so there is essentially only one spectral invariant.
In constrast, in the $S^1$-equivariant version we replace $H(X)$ by $H_{S^1}(X) \cong H(X) \otimes H(\mathbb{CP}^{\infty})$, giving rise to countably many independent classes even when $X$ is contractible.

\sss

Let us now observe that symplectic cohomology has various additional algebraic structures which we can try to exploit to define further capacities.
Indeed, there is a Lie bracket on $S^1$-equivariant symplectic cohomology $\sh_{S^1}(X)$, which corresponds to the string bracket when $X$ is a cotangent bundle.
On the chain level this is expected to come from an underlying $\Li$ structure, i.e. we have a sequence of operations $\ell_{S^1}^1,\ell_{S^1}^2,\ell_{S^1}^3,\dots$, where
$\ell_{S^1}^k$ is a $k$-to-$1$ operation of degree $4-3k$ (see \S\ref{sec:filtered_L-inf_algebras} for more details on $\calL_\infty$ structures and our conventions).
Here $\ell_{S^1}^1$ coincides with the equivariant differential $\bdy_{S^1}$, and $\ell_{S^1}^2$ descends to the Lie bracket on homology, while $\ell_{S^1}^{k \geq 3}$ does not descend to homology.

It is useful to observe that the data of this $\Li$ algebra can be encoded in terms of the bar complex $\bar \sc_{S^1}(X)$, which is a single chain complex with bar differential $\wh{\ell}_{S^1}: \bar \sc_{S^1}(X) \rightarrow \bar \sc_{S^1}(X)$ (see Definition~\ref{def:bar}).
Here as a $\K$-module $\bar \sc_{S^1}(X)$ is the (reduced, graded) symmetric tensor algebra $\ovl{S}\sc_{S^1}(X)$ on $\sc_{S^1}(X)$.
In particular, we have a distinguished subcomplex with trivial differential, which is spanned by words in the elements $u^{-k}e$, $k \in \Z_{\geq 0}$, and which we identify with $\bar \K[u^{-1}] \cong \ovl{S}\K[u^{-1}]$.

We now cook up a family of symplectic capacities as follows.
Given an element $\bb \in \ovl{S}\K[u^{-1}]$, define $\dapac_{\bb}$ by 
$$ \dapac_{\bb}(X) := \inf\{ a\;:\; \bb = \wh{\ell}_{S^1}(x)  \text{ for some } x \in \calF_{\leq a}\bar\sc_{S^1}(X)\}.$$ As with $\capac_k^{\op{GH}}(X)$ above, it follows from the invariance properties of symplectic cochains that $\dapac_{\bb}(X)$ is independent of all choices involved in its construction and is monotone with respect to exact symplectic embeddings.

  Similar to \eqref{eq:SH_LES}, quotienting out the subcomplex $\bar C(X) \subset \bar \sc(X)$ gives a long exact sequence on homology 
  \begin{equation} \label{eq:bar_LES}
\begin{tikzcd}
\cdots & H(\bar C(X)) & H(\bar \sc_{S^1}(X)) & H(\bar_+ \sc_{S^1}(X)) & H(\bar C(X)) & \cdots
\arrow[from=1-1,to=1-2]
\arrow[from=1-2,to=1-3]
\arrow[from=1-3,to=1-4]
\arrow[from=1-4,to=1-5,"\delta"]
\arrow[from=1-5,to=1-6]
\end{tikzcd}
\end{equation} 
  If $\sh_{S^1}(X) = 0$ (or equivalently $\sh(X) = 0$), then the homology of $\bar \sc_{S^1}(X)$ also necessarily vanishes.
  This occurs e.g. if $X$ is a star-shaped domain in $\C^n$, in which case the connecting map in \eqref{eq:bar_LES}
  gives an isomorphism
  \begin{align}
    H(\bar_+ \sc_{S^1}(X)) \cong \ovl{S} \K[u^{-1}],
  \end{align}
and we see that $\dapac_{\bb}(X)$ is precisely the spectral invariant of the homology class mapping to $\bb \in \ovl{S} \K[u^{-1}]$.

\begin{remark}\label{rmk:intro_S1_gen_of_GH}
We also note that the bar complex $\bar\sc_{S^1}(X)$ has another increasing filtration by word length, i.e. $\bar_{\leq l}\sc_{S^1}(X) \subset \bar_{\leq l'}\sc_{S^1}(X)$
for $1 \leq l < l'$, 
and we can refine the construction and define capacities $\dapac_{\bb}^{\leq l}$ by
$$ \dapac_{\bb}^{\leq l}(X) := \inf\{a \;:\; \bb = \wh{\ell}_{S^1}(x)\text{ for some } x \in \calF_{\leq a}\bar_{\leq l}\sc_{S^1}(X)\}.$$
In the special case with $l=1$ and $\bb = u^{-k+1} \in \ovl{S} \K[u^{-1}]$ for some $k \in \Z_{\geq 1}$ (i.e. $\bb$ is a word of length one),
$\dapac_{u^{-k+1}}^{\leq 1}$ reduces to $\capac_k^{\op{GH}}$,
so the family of capacities $\{\dapac_\bb^{\leq l}\} $ naturally extends the sequence of Gutt--Hutchings capacities.
\end{remark}

\begin{remark}[other structures on symplectic cohomology]
  
In characteristic zero, symplectic cohomology $\sh$ is known to have the structure of a BV-algebra (see e.g. \cite{abouzaid2013symplectic}), i.e. the BV operator is compatible with a product and a bracket, both given by counting pairs of pants but with asymptotic markers treated differently.
On the chain level the product has an underlying $\Ai$ structure $\mu^1,\mu^2,\mu^3,\dots$, where $\mu^k$ has degree $2-k$, and the bracket has an underlying $\Li$ structure $\ell^1,\ell^2,\ell^3,\dots$, where $\ell^k$ has degree $3-2k$.
It is tempting to also use these non-equivariant structures to define symplectic capacities, but unfortunately
the subcomplex of $\bar \sc$ spanned by words in $e$ (i.e. the analogue of $\bar\K[u^{-1}]$) is acyclic in the case of the $\Ai$ structure (due to a nontrivial product), and it is just one-dimensional in the case of the $\Li$ structure (for grading reasons).

\end{remark}

\subsection{Main results}

Finally, we arrive at the main thrust of this paper.

\subsubsection{Stable symplectic capacities}

As before, let $X^{2n}$ be a connected Liouville domain.
In \S\ref{sec:RSFT} we discuss rational symplectic field theory as a filtered $\Li$ algebra whose underlying chain complex is the contact homology algebra $\cha(\bdy X)$, which is generated as a $\K$-algebra by (good) Reeb orbits in $\bdy X$.
This linearizes to a filtered $\Li$ structure whose underlying chain complex is linearized contact homology $\chlin(X)$,\footnote{In this paper by default all SFT structures are taken at chain level, and we will be explicit whenever passing to homology (similar conventions are used e.g. in \cite{bourgeois2012effect}).}
which is generated as a $\K$-module by (good) Reeb orbits in $\bdy X$.

In \S\ref{sec:construction_of_capacities} we construct a family of dimensionally stable symplectic capacities as follows.
Let $\K\langle \tt_1,\tt_2,\tt_3,\dots\rangle$ denote the graded $\K$-module spanned by formal variables $\tt_1,\tt_2,\tt_3,\dots$ with $|\tt_i| = -2-2i$, and let $\ovl{S}\K\langle \tt_1,\tt_2,\tt_3,\dots\rangle$ denote its (reduced) symmetric tensor algebra, i.e. the $\K$-module spanned by nonempty words in the generators $\tt_1,\tt_2,\tt_3,\dots$ (we can identify this with the polynomial algebra $\K[\tt_1,\tt_2,\tt_3,\dots]$ modulo the constants).
We view $\K\langle \tt_1,\tt_2,\tt_3,\dots\rangle$ as an $\Li$ algebra over $\K$ with all $\Li$ operations trivial, and we view $\ovl{S}\K\langle \tt_1,\tt_2,\tt_3,\dots\rangle$ as its bar complex (which also has trivial differential).

\begin{thm}\label{thm:g_b_intro}
 For each $\bb \in \ovl{S}\K\langle\tt_1,\tt_2,\tt_3,\dots\rangle$, there is a symplectic capacity $\gapac_\bb(X)$ defined for Liouville domains $X$ which satisfies the following properties:
\begin{itemize}
  \item (scaling) If $X'$ is given by scaling the symplectic form of $X$ by $c \in \R_{>0}$, then we have $\gapac_\bb(X') = c\gapac_\bb(X)$.
  \item (monotonicity) Given a symplectic embedding $X \hooksymp X'$ between Liouville domains of the same dimension we have $\gapac_\bb(X) \leq \gapac_\bb(X')$.
  \item (stabilization) We have $\gapac_\bb(X \times B^2(S)) = \gapac_\bb(X)$ whenever $S > \gapac_\bb(X)$, provided that $X$ satisfies the index positivity Assumption~\ref{assump:index_pos} (this holds e.g. if $X \subset \R^{2n}$ is a convex domain).\footnote{Strictly speaking $X \times B^2(S)$ is a Liouville domain with corners, but this holds for any sufficiently close smoothing of the corners. In fact, in \S\ref{subsubsec:behavior_under_stab} we formulate a more general setting in which stabilization holds.} 

\end{itemize}
\end{thm} 
\NI Note that it follows from monotonicity that $\gapac_\bb$ is a symplectomorphism invariant.
The stabilization property can be viewed as a quantitative refinement of the property $\gapac_k(X \times \C) = \gapac_k(X)$.

The definition of $\gapac_\bb$ is parallel to that of $\dapac_\bb$ discussed in \S\ref{subsec:capac_from_HF}, but with using the $\Li$ structure underlying linearized contact homology $\chlin(X)$ in place of $\sc_{S^1,+}(X)$, and replacing the connecting homomorphism $\delta$ in \eqref{eq:bar_LES} with a geometrically defined $\Li$ homomorphism $\auglin \lll \TT p \rrr: \chlin(X) \ra \K\langle \tt_1,\tt_2,\tt_3,\dots\rangle$ which counts punctured curves satisfying a local tangency constraint in the symplectic completion $\wh{X}$ of $X$ (local tangency constraints and more general geometric constraints are discussed in \S\ref{sec:geometric_constraints}).
The fact that $\gapac_\bb$ is monotone under general (i.e. not necessarily exact) symplectic embeddings is nonobvious and is based on the discussion of non-exact functoriality and deformations by Maurer--Cartan elements which appears in \S\ref{sec:CL}.

\begin{disclaimer}\label{disclaimer_from_intro}
The invariants defined in this paper are based on rational symplectic field theory, and we take as a black box its existence and expected properties as outlined e.g. in \cite{EGH2000} -- see \S\ref{sec:RSFT} for more details.
It is well-known that the general construction requires working with a virtual perturbation framework, for which several proposals have been put forward in varying degrees of completeness and generality. 
Apart from the general structural properties outlined in \S\ref{sec:RSFT} we will not otherwise discuss particularities of any virtual framework. See Remark~\ref{rmk:transversality} below for a more detailed discussion.
 \end{disclaimer}

The indexing set for the capacities $\{\gapac_\bb\}$ is rather large, but a specialization of this family yields a more user-friendly sequence of capacities $\gapac_1,\gapac_2,\gapac_3,\dots$ as follows.\footnote{These are also denoted by $\gapac\lll \T^{k-1}p\rrr$ in the body of the paper.}
\begin{thm}\label{thm:g_k_intro}
  For each $k \in \Z_{\geq 1}$ there is a symplectic capacity $\gapac_k(X)$ defined for Liouville domains $X$ which satisfies the scaling, monotonicity, and stabilization properties from Theorem~\ref{thm:g_b_intro}, as well as:
\begin{itemize}
  \item (upper bounds from closed curves) If $(M,\omega_M)$ is a closed symplectic manifold of the same dimension as $X$ and there is a symplectic embedding $X \hooksymp M$, then we have 
  $\gapac_k(X) \leq [\omega_M] \cdot A$ provided that $\gw_{M,A}\lll \T^{k-1}p\rrr \neq 0$ for some $A \in H_2(M)$.
  \item (Reeb orbits) If $\bdy X$ has nondegenerate Reeb dynamics, there is a collection of Reeb orbits $\ga_1,\dots,\ga_l$ in $\bdy X$ satisfying $\sum_{i=1}^l [\ga_i] = 0 \in H_1(X)$ such that $\gapac_k(X) = \sum_{i=1}^l \calA(\ga_i)$.
\end{itemize}
\end{thm}
\NI In the first bullet above, $\gw_{M,A}\lll \T^{k-1}p\rrr$ denotes the Gromov--Witten-type count of rational curves satisfying the local tangency constraint $\lll \T^{k-1}p\rrr$ (see \S\ref{subsec:local_tan_con}).
For the second bullet, if $X \subset \R^{2n}$ is a $2n$-dimensional star-shaped domain, or more generally if $c_1(X) = 0$, then we moreover have $\sum_{i=1}^l \cz(\ga_i) = l(n-3)+2k+2$, where $\cz$ denotes the Conley--Zehnder index with respect to any global symplectic trivialization of $TX$.

Intuitively, $\gapac(X)$ can be interpreted as the minimal energy\footnote{Note that the notion of energy used for SFT curves is a modification of the naive symplectic area, which is typically infinite --  
see \S\ref{subsec:geometric_setup}.}
of a rigid genus zero punctured pseudoholomorphic curve in $\wh{X}$ which passes through a specified point $p$ and is tangent to order $k-1$ to a specified smooth local divisor $D$ near $p$.
Of course, this interpretation comes with various caveats
\begin{itemize}
\item the asymptotic orbits must form a cycle with respect to the differential on the bar complex (this is automatic for an irrational ellipsoid for degree parity reasons)
\item the cycle may in fact be a $\K$-linear combination of (good) Reeb orbits
\item curves are counted algebraically rather than geometrically, and in general these counts could be virtual.
\end{itemize}

\sss

Analogous to Remark~\ref{rmk:intro_S1_gen_of_GH}, using the word length filtration on bar complexes (i.e. restricting the number of positive ends our curves can have), we can define extended exact symplectic capacities $\gapac_\bb^{\leq l}$ for $l \in \Z_{\geq 1}$ and $\bb \in \K[\tt_1,\tt_2,\tt_3,\dots]$, and similarly we have $\gapac_k^{\leq l}$ for $k,l \in \Z_{\geq 1}$.
The special case $l = 1$ corresponds to curves with one positive end, and at least for ``nice'' domains $X$ we expect 
\begin{align}\label{eq:gapac_equals_GH}
\gapac_k^{\leq 1}(X) = c_k^{\op{GH}}(X),
\end{align}
i.e. the family $\{\gapac_k^{\leq l}\}$ naturally extends the sequence of Gutt--Hutchings capacities.
\begin{addendum}
Indeed, \eqref{eq:gapac_equals_GH} is established in \cite{pereira2022equivariant} for any Liouville domain $X^{2n}$ satisfying $\pi_1(X) = 2c_1(X) = 0$.
\end{addendum}

\begin{remark}In the body of the paper we work over the Novikov ring in lieu of using action filtrations.
This allows us to define capacities for larger class of domains which are not necessarily exact, and it also allows us to discuss Maurer--Cartan theory (see \S\ref{sec:CL}).
\end{remark}
\begin{remark}
In contrast to the capacities $\capac_k^{\op{GH}}$, the capacities $\gapac_k$ and more generally $\gapac_{\bb}$ are actually monotone even for non-exact symplectic embeddings.
This follows from the Cieliebak--Latschev formalism (see \S\ref{sec:CL}), which describes cobordism maps for non-exact symplectic embeddings.
We note, however, that the word length refinements $\gapac_k^{\leq l}$ and $\gapac_{\bb}^{\leq l}$ are only monotone for exact symplectic embeddings.
\end{remark}

\subsubsection{First computations and applications}

In \S\ref{subsec:first_computations}, we give the following first computations:
\begin{thm}\label{thm:intro_computations}
We have:
\begin{enumerate}
  \item  $\gapac_k(B^4) = \lceil \tfrac{k+1}{3}\rceil$ for $k \geq 1 \text{ congruent to } 2 \text{ mod } 3$
  \item  $\gapac_k(E(1,x)) = k$ for $1 \leq k \leq x$
  \item $\gapac_k(P(1,x)) = \min(k,x+\lceil \tfrac{k-1}{2}\rceil)$ for $k \geq 1\text{ odd}$.
\end{enumerate}
 \end{thm} 
\NI Here $P(a_1,..,a_n)$ denotes the polydisk $B^2(a_1) \times ... \times B^2(a_n)$.
Items (1),(2),(3) correspond to Propositions \ref{prop:gapac_for_ball},\ref{prop:ellipsoid_computation},\ref{prop:gapac_poly_comp} respectively.

Comparing the obstructions from the first two items in Theorem~\ref{thm:intro_computations} with the embedding \eqref{eq:hind_emb}, it follows that the capacities $\{\gapac_k\}$ give optimal obstructions for the problem 
$$E(1,3d-1) \times \C^{n-2} \hookrightarrow E(c,c) \times \C^{n-2}$$
for any $d \geq 1$, $n \geq 3$, 
recovering a recent result of McDuff \cite{Mcduff_remark} (c.f. Theorem \ref{thm:stab_ell_known_results}).
Similarly, they give optimal obstructions for the problem
$$E(1,2d-1) \times \C^{n-2} \hookrightarrow P(c,c) \times \C^{n-2}$$
for any $d \geq 1$, $n \geq 3$.

The computations in Theorem~\ref{thm:intro_computations} also imply various new obstructions for symplectic embeddings, which we illustrate with some examples.
\begin{example}
Consider the symplectic embedding problem $$P(1,2) \times \C^{n-2} \hooksymp P(c,c)\times \C^{n-2}.$$
The capacities $\{\gapac_k\}$ give the lower bound $c \geq \tfrac{3}{2}$ (for convenience see Table~\ref{caps_for_P_1_2_and_ball}). We note that the Ekeland--Hofer capacities $\capac_k^{\op{EH}}$ give only $c \geq 1$, and the ECH capacities are not applicable here.
\end{example}
\begin{table}
\begin{equation*}
 \begin{array}{|c|c|c|c|c|c|}
 \hline
k & 1  & 3 & 5  & 7 & 9\\ \hline 
\gapac_k(P(1,2)) & 1  & 3  & 4  & 5  & 6 \\  \hline
\gapac_k(P(c,c)) & c  & 2c  & 3c  & 4c  & 5c \\  \hline
c_k^\op{EH}(P(1,2)) & 1  & 3  & 5  & 7  & 9\\  \hline
c_k^\op{EH}(P(c,c)) & c  & 3c  & 5c  & 7c  & 9c\\  \hline
 \end{array}.
\end{equation*}
\caption{Sample computations using Theorem~\ref{thm:intro_computations}.}
\label{caps_for_P_1_2_and_ball}
\end{table}
\begin{addendum}
The recent preprint \cite{irvine} gives a nearly complete description of the stabilized four-dimensional polydisk into polydisk problem; see also \cite[\S4.4]{chscI}.
\end{addendum}
\begin{example}
Consider the symplectic embedding problem $$P(1,3) \times \C^{n-2} \hookrightarrow B^4(c) \times \C^{n-2}.$$
The capacities $\{\gapac_k\}$ give the lower bound $c \geq \tfrac{5}{2}$, while 
the Ekeland--Hofer capacities give only $c \geq 2$ 
(for convenience see Table~\ref{caps_for_P_1_3_and_ball}).
Note that, according to Schlenk \cite[Prop 4.3.9]{Schlenk_embedding_problems}, symplectic folding to produces a symplectic embedding of $P(1,a)$ into $B^4(c)$ for any $a > 2$ and $c > 2 + \tfrac{a}{2}$.

Actually, using the stabilization part of Theorem~\ref{thm:g_k_intro}, we find e.g. that for any $S \geq 5$ there is no symplectic embedding $P(1,3) \times B^2(S) \hooksymp B^4(c) \times \C^{n-2}$ unless $c \geq 5/2$.
\end{example}
\begin{table}
\begin{equation*}
 \begin{array}{|c|c|c|c|c|c|}
 \hline
k & 1  & 3  & 5  & 7  & 9\\ \hline 
\gapac_k(P(1,3)) & 1  & 3  & 5  & 6  & 7 \\  \hline
\gapac_k(B^4(c)) & c  & 2c  & 2c  & 3c  & 4c \\  \hline
c_k^\op{EH}(P(1,3)) & 1  & 3  & 5  & 7 & 9\\  \hline
c_k^\op{EH}(B^4(c)) & c  & 2c  & 3c  & 4c  & 5c\\  \hline
 \end{array}.
\end{equation*}
\caption{More sample computations using Theorem~\ref{thm:intro_computations}.}
\label{caps_for_P_1_3_and_ball}
\end{table}

\begin{example}
In the followup work \cite{ell_into_poly}, we use the capacities $\{\gapac_k\}$ and the computations in Theorem~\ref{thm:intro_computations} to give optimal obstructions for certain stabilized embeddings of four-dimensional ellipsoids into polydisks, in particular proving Conjecture 1.4 from \cite{cristofaro2017symplectic}.
\end{example}

\sss

For the problem 
$$ E(1,x) \times \C^{n-2} \hooksymp B^4(c) \times \C^{n-2},$$
it turns out that the simplified capacities $\{\gapac_k\}$ do {\em not} always give optimal obstructions. 
Indeed, we give a general formula for $\gapac_k(E(a,b))$ (in fact for all four-dimensional convex toric domains) in the follow up \cite{enumerating2}, and we observe that these capacities do not always recover the known optimal results in Theorem~\ref{thm:stab_ell_known_results}. 
In hindsight this is not too surprising, since the simplification from $\{\gapac_\bb\}$ to $\{\gapac_k\}$ involves a rather coarse minimization procedure. 

On the other hand, for abstract reasons which we explain at the end of \S\ref{subsec:first_computations}, the set of capacities $\gapac_\bb$ necessarily give optimal obstructions for all known cases covered in Theorem \ref{thm:stab_ell_known_results}, suggesting:
\begin{conjecture}\label{conj:complete_obstructions}
The capacities $\{\gapac_\bb\}$ provide a complete set of obstructions for Question \ref{question:stab_ell}.
\end{conjecture}

\begin{addendum}\label{addendum:chscI_comps}
In the followup work \cite{chscI} we give an explicit algorithmic computation of $\gapac_\bb(X)$ for $X$ a convex toric domain. 
In particular, this gives many new obstructions for stabilized ellipsoid embeddings.
However, since the combinatorics are rather involved, the full case of Conjecture~\ref{conj:complete_obstructions} is at present still open.
\end{addendum}

  \begin{remark}
The indexing set $\ovl{S}\K[t]$ for the above capacities is quite large, and is also somewhat redundant, since at the very least the capacities only depend on $\bb \in \ovl{S}\K[t]$ up to multiplication by an invertible element of $\K$. 
Note also that we are not using in any way the algebra structure of $\K[t]$.
One can use persistent homology theory to find, for two given Liouville domains $X,X'$, a countable set of $\{\gapac_\bb\}$ which carries all of the obstructions of the full family -- see \cite[\S4.2]{chscI}.
\end{remark}

\subsubsection{Further capacities from rational symplectic field theory}

The local tangency constraints appearing in the definition of $\gapac_\bb$ are key to the dimensional stability property, but we can also consider various other geometric constraints and define corresponding (typically unstable) capacities. 
Indeed, the family $\{\gapac_\bb\}$ is part of larger family given by replacing local tangency constraints with more general multibranched local tangency constraints
 of the form $\lll \T^{m_1}p,\dots, \T^{m_b}p\rrr$. Roughly, this requires curves to pass through the specified point $p$ in $b$ branches, with the $i$th branch tangent to order $m_i$ to the local divisor $D$ (see \S\ref{subsec:multibr_loc_tan}).
We denote this family of capacities by $\{\rapac_\bb\}$, where now $\bb \in \ovl{S}\K[\tt_1,\tt_2,\tt_3,\dots]$.
We note that in lieu of the stabilization property we have a disjoint union property (see \S\ref{subsubsec:disjoint_union_lb}), the ECH analogue of which has many important ramifications (see e.g. \cite{cristofaro2015asymptotics,cristofaro2016one,iriedense}).

 Whereas the capacities $\gapac_\bb$ are defined using the $\Li$ structure on linearized contact homology $\chlin(X)$, the capacities $\{\rapac_\bb\}$ require the $\Li$ on the full contact homology algebra $\cha(X)$. 
This algebraic structure, which we describe in \S\ref{subsubsec:Li_str_on_CHA}, gives a perspective on rational symplectic field theory which is particularly useful for defining symplectic capacities.\footnote{The recent preprint \cite{latschev2022remarks} explains in detail the relationship between our constructions and the original RSFT formalism of \cite{EGH2000}.}

At first glance still more generally, we can also impose multibranched local tangency constraints at several distinct points. 
Namely, we specify distinct points $p_1,\dots,p_r \in \wh{X}$ and corresponding smooth local divisors at those points, and consider constraints of the form
$\lll \underbrace{\TT p_1,\dots \TT p_1}_{b_1},\dots,\underbrace{\TT p_r,\dots,\TT p_r}\rrr$ for $b_1,\dots,b_r \in \Z_{\geq 1}$.
Notably, in the special case without any tangency conditions this specializes to a family of ``blowup capacities'', denoted by $\rapac\lll p_1^{\times b_1},\dots,p_r^{\times b_r}\rrr$, which can be also be interpreted in terms of blowing up $X$ at $r$ distinct points and then restricting to certain homology classes of curves.
It follows from known results and basic properties of these capacities that the full family of blowup capacities gives (together with the volume constraint) complete obstructions to the four dimensional ellipsoid embedding problem (as well as the more general embedding problems considered in \cite{cristofaro2019symplectic}).
The further special case $\rapac_k := \rapac\lll p_1,\dots,p_r\rrr$, which roughly corresponds to the minimum energy of a (not necessarily connected) rational curve passing through $r$ distinct points, was first sketched by Hutchings in \cite{hutchings_2013} (this was an important inspiration for this paper).

As we explain in \S\ref{subsec:ppt}, in dimension four the above capacities involving constraints at multiple points can actually be reduced to the capacities $\rapac_\bb$. Indeed, we present an explicit formula describing how to combine constraints at different $p_1,\dots,p_r$ into constraints at a single point $p$ (a special case involving two points appears in Gathmann's work \cite{Gathmann_GW_of_blowups}).
This can be used to give recursive formulas for curve counts with tangency constraints in terms of curve counts without tangency constraints, e.g. in \cite{enumerating} we apply this technique to count closed curves in $\CP^2$ with multibranched tangency constraints.
In particular, the capacities $\{\rapac_\bb\}$ recover the obstructions coming from both the blowup capacities $\rapac\lll p_1^{\times b_1},\dots,p_r^{\times b_r}\rrr$ and the stable capacities $\{\gapac_\bb\}$, and we expect the full family to play an important role in the general case of Question~\ref{question:ell_emb}.

\begin{remark}
It is also possible to define higher genus analogues of $\rapac_\bb$ using the full symplectic field theory \cite{EGH2000}, although we do not pursue this here. It is an interesting question whether higher genus curves play any role in higher dimensional symplectic embedding obstructions.
\end{remark}

\subsection{Outline}

In \S\ref{sec:filtered_L-inf_algebras} we review the necessary background on filtered $\Li$ algebras, which forms the algebraic backbone of the paper. In \S\ref{sec:RSFT} we discuss rational symplectic field theory, formulated using the language of filtered $\Li$ algebras. 
In \S\ref{sec:CL} we describe the Cieliebak--Latschev formalism in this framework, which among other things is useful for studying the role of compactifications.
In \S\ref{sec:geometric_constraints} we construct the $\Li$ augmentations, which are interpreted as counting curves with geometric constraints and defined formally using RSFT cobordism maps.
Finally, in \S\ref{sec:construction_of_capacities} we give the general construction of the capacities, prove some key properties, and present the aforementioned applications.

\section*{Acknowledgements}
{
I wish to wish Dusa McDuff and Dan Cristofaro-Gardiner for many interesting discussions and for explaining their work on the stabilized ellipsoid embedding problem. I also greatly benefited from discussions with Michael Hutchings, Chiu-Chu Melissa Liu, John Pardon, and Jingyu Zhao. Lastly, I thank the referee for many helpful suggestions which improved the exposition.
}

\section{Filtered $\Li$ algebras}\label{sec:filtered_L-inf_algebras}

In this section we go over the necessary background on filtered $\Li$ algebras, meanwhile setting the notation and conventions for the rest of the paper. Roughly speaking, an $\Li$ algebra is like a graded Lie algebra which is equipped with a differential, and where the graded Jacobi identity holds only up to a coherent sequence of higher homotopies. Thus we have a graded module $V$ and multilinear $k$-to-$1$  operations $\ell^k$ for $k = 1,2,3,...$, such that $\ell^1$ is a differential, $\ell^2$ descends to a Lie bracket on the $\ell^1$-cohomology, $\ell^3$ gives a chain nullhomotopy of the Jacobi relation for $\ell^2$, and so on.
The homology inherits the structure of a graded Lie algebra,
and the special case where $\ell^k \equiv 0$ for all $k \geq 3$ is equivalent to the notion of a differential graded Lie algebra (DGLA). Thus $\Li$ algebras and DGLAs are Lie analogues of the slightly more familiar notions of $\Ai$ algebras and differential graded algebras (DGAs).

As it turns out, DGLAs and $\Li$ naturally occur all over geometry and topology,
typically modeling ``chain level objects'' for which the underlying homology is a graded Lie algebra. Even in cases when the natural model is a DGLA (and in fact one can always abstractly replace an $\Li$ algebra by an equivalent DGLA), one is more or less forced to face $\Li$ notions when considering the homotopy category. For example, a key property of $\Li$ homomorphisms is that quasi-isomorphisms are automatically homotopy equivalences. This statement fails for the naive homotopy category of DGLAs, for which the naive notion of homomorphism turns out to be too ``strict''. 

In symplectic geometry, $\Li$ algebras arise naturally when considering rational pseudoholomorphic curves with multiple positive (i.e. input) punctures and one negative (output) puncture.
For the quantitative applications we have in mind, it will be essential for us to consider $\Li$ algebras which are additionally equipped with filtrations. 
In terms of geometry these filtrations will be 
related to the symplectic action functional.
One way to encode this extra data is by working over the Novikov ring and recording energies in pseudoholomorphic curve counts.
The resulting category of filtered $\Li$ algebras is very rich and is central to our constructions of quantitative symplectic invariants.

In what follows, we adopt some of the notation and conventions for $\Li$ algebras from \cite{Fukaya-deformation,Getzler_lie,Yalin_MC} and other sources.
Although the basic definitions are essentially standard, there is some small room for discrepancy arising from sign conventions, grading conventions, etc.

\subsection{$\Li$ algebra basics}
\subsubsection{Review of $\Ai$ algebras}

To set the stage for $\Li$ notions, we first discuss the slightly more familiar case of $\Ai$ algebras.
Let $\K$ be a fixed commutative ring containing $\Q$ (e.g. $\Q$ itself).
Let $V$ be a graded module over $\K$.
 In general we will only assume $\Z/2$ gradings, although the following discussion applies equally if $V$ has a $\Z$ grading (which occurs in favorable cases).
Let $s\K$ denote the graded $\K$-module consisting of a single copy of $\K$ in degree $-1$,
and let $sV := s\K \otimes V$ be the ``suspension'' of $V$. Thus $(sV)^k$ is identified with $V^{k+1}$, and there is a natural degree $-1$ map $s: V \rightarrow sV$ sending $v$ to $s(v) := 1 \otimes v$, where $1 \in \K$ is the unit. 
For $x \in V^k$, we put $|x| = k$ and $|sx| = k-1$.
In what follows, the ubiquitous use of suspensions gives a convenient way of describing the gradings and signs in $\Li$ equations; they can be ignored on a first pass if one is not too concerned with these.

A succinct way to define an $\Ai$ algebra structure on $V$ is as a degree $+1$ coderivation $\wh{\mu}$ on the reduced tensor coalgebra $\ovl{T}sV$ which satisfies $\wh{\mu}^2 = 0$. In more detail, let $\otimes^kV$ denote the tensor product (over $\K$) of $k$ copies of $V$, and set $\ovl{T}V = \bigoplus\limits_{k \geq 1}\otimes^kV$.  
Although $\ovl{T}V$ is an algebra, at present it is more relevant that is a coalgebra. Here the line over $T$ refers to the fact that we start at $k=1$ rather than $k=0$,
so $\ovl{T}V$ is the {\em reduced} tensor coalgebra on $V$. 
In general, a graded coassociative coalgebra is the dual notion of a graded associative algebra, consisting of a graded $\K$-module $C$ and a degree $0$ comultiplication map $\Delta: C \rightarrow C \otimes C$
satisfying the coassociativity condition $(1 \otimes \Delta) \circ \Delta = (\Delta \otimes 1) \circ \Delta$.
In the case of $\ovl{T}V$, the comultiplication $\Delta$ is given by $$\Delta(v_1 \otimes ... \otimes v_k) = \sum_{i=1}^{k-1}(v_1 \otimes ... \otimes v_i) \otimes (v_{i+1} \otimes ... \otimes v_k),$$
which is indeed coassociative.

Given a coassociative coalgebra $(C,\Delta)$, a coderivation is a map $\delta: C \rightarrow C$ satisfying the coLeibniz condition $\Delta \circ \delta = (1 \otimes \delta) \circ \Delta + (\delta \otimes 1)\circ \Delta$.
Here and throughout we implement the Koszul--Quillen sign convention that $f \otimes g$ pairs with $x \otimes y$ to give $(-1)^{|x| |g|}f(x) \otimes g(y)$, for $f,g: C \rightarrow C$ and $x,y \in C$.
For example, for $\delta$ of degree $1$ we have $(1 \otimes \delta)(x\otimes y) = (-1)^{|x|}x \otimes \delta(y)$.
\begin{definition}
An {\em $\Ai$ algebra} is a graded $\K$-module $V$ and a degree $+1$ 
coderivation\footnote{The degree of $+1$ for $\wh{\mu}$ reflects the fact that we choose cohomological conventions throughout. Some references instead adapt homological conventions, in which case $\wh{\mu}$ has degree $-1$.}
 $\wh{\mu}: \ovl{T}(sV) \rightarrow \ovl{T}(sV)$ satisfying $\wh{\mu}^2 = 0$.
\end{definition}
A basic fact is that any coderivation $\wh{\mu}: \ovl{T}V \rightarrow \ovl{T}V$ is uniquely determined by the sequence of maps $\mu^k: \otimes^kV \rightarrow V$ given by the compositions
$$\otimes^kV \hookrightarrow \ovl{T}V \overset{\wh{\mu}}\longrightarrow \ovl{T}V \rightarrow V,$$
i.e. inclusion followed by $\wh{\mu}$ followed by projection.
Indeed, any sequence of $\K$-linear maps $\mu^k: \otimes^kV \rightarrow V$ uniquely extends to a coderivation $\wh{\mu}: \ovl{T}V \rightarrow \ovl{T}V$ such the induced maps $\otimes^k V \rightarrow V$ recover the $\mu^k$. Explicitly, the extension is given by 
\begin{align}
\wh{\mu}(v_1 \otimes ... \otimes v_n) &:= \sum_{\substack{1 \leq i \leq n-k+1\\ 1 \leq k \leq n}}\left( 1^{\otimes (i-1)}\otimes\mu^k \otimes 1^{\otimes(n-i-k+1)} \right)(v_1 \otimes ... \otimes v_n). \label{mu_extension}
\end{align}
We can therefore equivalently define an $\Ai$ algebra structure as a sequence of degree $1$ maps $\mu^k: \otimes^k(sV) \rightarrow sV$ satisfying the quadratic relations
\begin{align*}\label{Ainf_relations}
\sum_{\substack{1 \leq i \leq n-k+1\\ 1 \leq k \leq n}} (-1)^{|sv_1| + ... + |sv_{i-1}|}\mu^{n-k+1}\left(sv_1 \otimes ... \otimes sv_{i-1} \otimes \mu^k(sv_i \otimes ... \otimes sv_{i+k-1}) \otimes sv_{i+k} \otimes ... \otimes sv_n\right) = 0.
\end{align*}
Or, we could work instead with the desuspended operations $s^{-1} \circ \mu^k \circ s^{\otimes k}: \otimes^k V \rightarrow V$, which have degree $2-k$ and satisfy the same quadratic relations except for some additional Koszul signs.
In any case, we will often refer to an $\Ai$ algebra by its underlying $\K$-module $V$ when the rest of the structure is implicit.

\begin{remark}
Recall that a DGA is an associative $\K$-algebra equipped with a degree $+1$ differential $\bdy$ and a degree $0$ associative product satisfying the graded Leibniz rule $\bdy(a\cdot b) = \bdy(a)\cdot b + (-1)^{|a|}a\cdot \bdy(b)$. 
In the special case of an $\Ai$ algebra with $\mu^k = 0$ for all $k \geq 3$, we recover the notion of a DGA on $V$ by setting 
$x\cdot y := s^{-1}\mu^2\circ (s\otimes s)(x \otimes y) = (-1)^{|x|}s^{-1}\mu^2(sx,sy)$.
\end{remark}

\subsubsection{Symmetric tensor coalgebras and $\Li$ algebras}

Let $\Sigma_k$ denote the symmetric group on $k$ elements. 
For $V$ a graded $\K$-module, we consider the signed action of $\Sigma_k$ on $\otimes^kV$, given by
$$ \sigma(v_1 \otimes ... \otimes v_k) = \sign(V,\sigma;v_1,...,v_n)v_{\sigma(1)}\otimes ... \otimes v_{\sigma(k)},$$
where $\sign(V,\sigma;v_1,...,v_n) \in \{1,-1\}$ is the associated Koszul sign
$$\sign(V,\sigma;v_1,...,v_n) = (-1)^{\{|v_i| |v_j|\;:\; 1 \leq i < j \leq n,\; \sigma(i) > \sigma(j)\}}$$
(i.e. swapping $v_i$ and $v_{i+1}$ ``costs'' $(-1)^{|v_i| |v_{i+1}|}$).
Let $\odot^kV := (\otimes^kV)/\Sigma_k$ denote the quotient space.
Then $\ovl{S}V := \bigoplus\limits_{k \geq 1}\odot^k V$ is the reduced symmetric tensor coalgebra on $V$.
For $v_1 \otimes ... \otimes v_k \in \otimes^kV$, we denote its image in $\odot^kV$ by $v_1 \odot ... \odot v_k$.
Let $\Sh(i,k-i)$ denote the subset of permutations $\sigma \in \Sigma_k$ satisfying 
$\sigma(1) < ... < \sigma(i)$ and $\sigma(i+1) < ... < \sigma(k)$.
The comultiplication $\Delta: \ovl{S}V \rightarrow \ovl{S}V \otimes \ovl{S}V$ is given by 
\begin{align*}
\Delta(v_1 \odot ... \odot v_k) := \sum_{i=1}^{k-1}\sum_{\sigma \in \Sh(i,k-i)}\sign(\sigma,V;v_1,...,v_k)(v_{\sigma(1)} \odot ... \odot v_{\sigma(i)}) \otimes (v_{\sigma(i+1)} \odot ... \odot v_{\sigma(k)}).
\end{align*}
We note that this comultiplication is
 cocommutative in the sense that $\mathfrak{R} \circ \Delta = \Delta$, where $\mathfrak{R}: \ovl{S}V \otimes \ovl{S}V \rightarrow \ovl{S}V \otimes \ovl{S}V$ is given by $\mathfrak{R}(x \otimes y) = (-1)^{|x| |y|} y \otimes x$.

\begin{definition}
An {\em $\Li$ algebra} is a graded $\K$-module $V$ and a degree $+1$ coderivation $\wh{\ell}: \ovl{S}(sV) \rightarrow \ovl{S}(sV)$ satisfying $\wh{\ell}^2 = 0$.
\end{definition}

As in the $\Ai$ case, any coderivation $\wh{\ell}: \ovl{S}V \rightarrow \ovl{S}V$ is uniquely determined by the maps $\ell^k: \odot^k V \rightarrow V $ given by the composition $\odot^kV \hookrightarrow \ovl{S}V \overset{\wh{\ell}}\longrightarrow \ovl{S}V \rightarrow V$ of inclusion followed by $\wh{\ell}$ followed by projection.
Indeed, any sequence of $\K$-linear maps $\ell^k: \odot^kV \rightarrow V$ extends uniquely to a coderivation
$\wh{\ell}: \ovl{S}V \rightarrow \ovl{S}V$ via the formula
\begin{align*}
 \wh{\ell}(v_1 \odot ... \odot v_n) &:= \sum_{k=1}^n\sum_{\sigma \in \Sh(k,n-k)}\sign(\sigma,V;v_1,...,v_n)\ell^k(v_{\sigma(1)}\odot ... \odot v_{\sigma(k)}) \odot v_{\sigma(k+1)} \odot ... \odot v_{\sigma(n)}\\
 &= \sum_{k=1}^n\sum_{\sigma \in \Sigma_k}\frac{\sign(\sigma,V;v_1,...,v_n)}{k!(n-k)!} \ell^k(v_{\sigma(1)}\odot ... \odot v_{\sigma(k)}) \odot v_{\sigma(k+1)} \odot ... \odot v_{\sigma(n)}  .
\end{align*}
Note that an $\K$-linear map $\ell^k: \odot^k(sV) \rightarrow sV$ is the same as a symmetric $\K$-linear map $\otimes^k (sV) \rightarrow sV$, i.e. one for which interchanging inputs $sv_i,sv_{i+1}$ introduces a sign $(-1)^{|sv_i| |sv_{i+1}|}$.
This is in turn equivalent to the desuspension $s^{-1}\circ \ell^k \circ s^{\otimes k}: \otimes^k V \rightarrow V$ being skew-symmetric,
i.e. interchanging inputs $v_i,v_{i+1}$ introduces a sign $-(-1)^{|v_i||v_{i+1}|}$. 
We can thus equivalently define an $\Li$ algebra as a sequence of maps $\ell^k: \odot^ksV \rightarrow sV$ of degree $1$ satisfying the quadratic relations 
\begin{align}
\sum_{k=1}^n\sum_{\sigma \in \Sh(k,n-k)}\sign(\sigma,V;sv_1,...,sv_n)\ell^{n-k+1}\left(\ell^k(sv_{\sigma(1)}\odot ... \odot sv_{\sigma(k)}) \odot sv_{\sigma(k+1)} \odot ... \odot sv_{\sigma(n)}\right) = 0,
\end{align}
or in terms of the degree $2-k$ desuspended maps $s^{-1} \circ \ell^k \circ s^{\otimes k}: \otimes^k V \rightarrow V$,
which are skew-symmetric and satisfy the same relations with a few extra signs.

\begin{definition}\label{def:bar}
Given an $\Li$ algebra as above, the chain complex $(\ovl{S}(sV),\wh{\ell})$ is called its {\em bar complex}, denoted by $\bar V$. 
\end{definition}
\begin{remark}
Note that the bar complex is generally much larger than the chain complex $(V,\ell^1)$.
There is also an $\Ai$ analogue.
Technically it is the {\em reduced} bar complex, whereas the unreduced version has underlying $\K$-module 
$S(sV)$ instead of $\ovl{S}(sV)$.
There is also a related but somewhat different complex, the Chevalley--Eilenberg complex of $V$, which plays an important role in Lie algebra cohomology. 
It is defined by dualizing the $k$-to-$1$ operations $\ell^k$ to produce $1$-to-$k$ operations on the dual space 
$(sV)^\vee$, and then assembling these into a differential on the symmetric tensor algebra of $(sV)^\vee$ via the cobar construction.
\end{remark}

\noindent 

\begin{remark}
Recall that a DGLA is a graded $\K$-module with a degree $1$ differential $\bdy$
and a degree $0$ bracket $[-,-]$ which is graded skew-symmetric, $[x,y] = -(-1)^{|x| |y|}[y,x]$,
satisfying the Jacobi identity 
\begin{align*}
[[x,y],z] + (-1)^{|z|(|x|+|y|)}[[z,x],y] + (-1)^{|x|(|y| + |z|)}[[y,z],x] = 0
\end{align*}
and the Leibniz rule
\begin{align*}
\bdy[x,y] = [\bdy x,y] + (-1)^{|x|}[x,\bdy y].
\end{align*}
This corresponds to the special case of an $\Li$ algebra with $\ell^k \equiv 0$ for all $k \geq 3$.
More precisely, given such an $\Li$ algebra structure on $V$, we recover the structure of a DGLA by setting 
$[x,y] := s^{-1} [-,-] \circ (s \otimes s)(x\otimes y) =  (-1)^{|x|}s^{-1}\ell^2(sx,sy)$.
\end{remark}

\begin{remark}
Instead of working with the quotient spaces $\otimes^kV/\Sigma_k$,
we could alternatively work with the fixed point spaces $(\otimes^kV)^{\Sigma_k}$
and define the reduced symmetric tensor coalgebra by
$\ovl{S}_\fix V := \bigoplus\limits_{k \geq 1} (\otimes^k V)^{\Sigma_k}$.
Here the coproduct on $\ovl{S}_\fix V$ is given by the restriction of the coproduct on $\ovl{T}V$.
Indeed, the obvious projection map from $\ovl{S}_\fix V$ to $\ovl{S}V$ is an isomorphism of coalgebras,
with inverse $N: \ovl{S}V \rightarrow \ovl{S}_\fix V$ given by 
\begin{align*}
N(v_1 \odot ... \odot v_k) = \frac{1}{k!}\sum_{\sigma \in \Sigma_k} \sigma \cdot (v_1 \otimes ... \otimes v_k).
\end{align*}
\end{remark}

\subsubsection{$\Li$ homomorphisms, quasi-isomorphisms, homotopies, and homotopy equivalences} 

We now define more key players in the homotopy category of $\Li$ algebras. 
For coalgebras $(C,\Delta)$ and $(C',\Delta')$, an $\K$-module map $F: C \rightarrow C'$ is a coalgebra homomorphism
if $$\Delta' \circ F = (F \otimes F) \circ \Delta.$$
Consider two $\Li$ algebras $V$ and $V'$
with corresponding coderivations 
$\wh{\ell}: \ovl{S}(sV) \rightarrow \ovl{S}(sV)$ and $\wh{\ell}': \ovl{S}(sV') \rightarrow \ovl{S}(sV')$.
\begin{definition}
An {\em $\Li$ homomorphism} from $V$ to $V'$ is a degree $0$ coalgebra homomorphism $\wh{\Phi}: \ovl{S}(sV) \rightarrow \ovl{S}(sV')$
satisfying $\wh{\Phi} \circ \wh{\ell} = \wh{\ell}' \circ \wh{\Phi}.$
\end{definition}
A coalgebra homomorphism $\wh{\Phi}: \ovl{S}(sV) \rightarrow \ovl{S}(sV')$ is uniquely determined by the maps $\Phi^k: \odot^k (sV) \rightarrow sV'$ given by the composition $\odot^k (sV) \hookrightarrow \ovl{S}(sV) \overset{\wh{\Phi}}\longrightarrow \ovl{S}(sV') \rightarrow sV'$,
and any sequence of maps $\Phi^k: \odot^k (sV) \rightarrow sV'$ can be uniquely extended to a coalgebra homomorphism $\wh{\Phi}: \ovl{S}(sV) \rightarrow \ovl{S}(sV')$ by the formula 
$$\wh{\Phi}(v_1\odot ... \odot v_n) := \sum_{\substack{k \geq 1\\ i_1 + ... + i_k = n}}\sum_{\sigma \in \Sigma_n} \frac{\sign(\sigma,V;v_1,...,v_n)}{k!i_1!...i_k!}(\Phi^{i_1} \odot ... \odot \Phi^{i_k})(v_{\sigma(1)}\odot ... \odot v_{\sigma(n)}).$$
We will typically write an $\Li$ homomorphism as $\Phi: V \rightarrow V'$ and think of it either as a sequence of maps $s^{-1}\circ \Phi^k \circ s^{\odot k}: \odot^k V \rightarrow V$, $k = 1,2,3,...$, where $\Phi^k$ has degree $1-k$, or equivalently as a single map of bar complexes $\wh{\Phi}:\ovl{S}(sV) \rightarrow \ovl{S}(sV')$ of degree $0$.

Now consider an $\Li$ algebra $V$ with operations $\ell_V^k: \odot^k sV \rightarrow sV$, and suppose that $C$ is a commutative differential graded algebra (CDGA) with differential $\bdy: C \rightarrow C$. In this case one can define a new $\Li$ algebra with underlying underlying $\K$-module $V\otimes_\K C$. 
The $\Li$ operations  
$\ell^k_{V \otimes C}: \odot^k(sV \otimes C) \rightarrow sV \otimes C$ are given by
\begin{align*}
&\ell^1_{V \otimes C}(sv \otimes c) = \ell^1_{V}(sv) \otimes c + (-1)^{|sv|}sv \otimes \bdy(c)
\end{align*}
for $k=1$, and 
\begin{align*}
\ell^k_{V \otimes C}(sv_1 \otimes c_1,...,sv_k \otimes c_k) = (-1)^{\sum\limits_{j < i}|sv_i| |c_j|}\ell^k_{V}(sv_1,...,sv_k) \otimes c_1...c_k
\end{align*}
for $k \geq 2$.

Let $\K[t,dt]$ denote the CDGA which is freely generated as a graded algebra by symbols $t$ and $dt$ with degrees $|t| = 0$ and $|dt| = 1$, with differential $\bdy(t) = dt$ and $\bdy(dt) = 0$.
We view $\K[t,dt]$ as the algebraic differential forms on the interval $[0,1]$. 
As a special case of the above construction, we can consider $V \otimes \K[t,dt]$.
Following \cite{Fukaya-deformation}, we can write a typical element of $sV \otimes \K[t,dt]$ as $P(t) + Q(t)dt$, where $P(t)$ and $Q(t)$ are polynomials with coefficients in $sV$. 
The $\Li$ operations can then be written as 
\begin{align*}
\ell^1(P(t) + Q(t)dt) = \ell^1(P(t)) + (-1)^{|P(t)|}\frac{dP}{dt}dt + \ell^1(Q(t))dt
\end{align*}
and
\begin{multline*}
\ell^k(P_1(t) + Q_1(t)dt,...,P_k(t) + Q_k(t)dt) = \\ \ell^k(P_1(t),...,P_k(t)) + \sum_{i=1}^k (-1)^{|P_{i+1}| + ... + |P_k|}\ell^k(P_1(t),...,Q_i(t),...,P_k(t))dt
\end{multline*}
for $k \geq 2$.

For $\Li$ homomorphisms $\Phi: V \rightarrow V'$ and $\Psi: V' \rightarrow V''$, the composite $\Li$ homomorphism $V \rightarrow V''$ is defined on the level of bar complexes by simply 
$\wh{\Psi} \circ \wh{\Phi}: \ovl{S}(sV) \rightarrow \ovl{S}(sV'')$.
For any $t_0 \in [0,1]$, there is an $\Li$ homomorphism $\Eval_{t = t_0}: V \otimes \K[t,dt] \rightarrow V$,
whose component with one input and one output sends $P(t) + Q(t)dt$ to $P(t_0)$, and which has no higher components.
\begin{definition}
Two $\Li$ homomorphisms $\Phi,\Psi: V \rightarrow V'$ are {\em $\Li$ homotopic} if there exists an $\Li$ homomorphism 
$H: V \rightarrow V' \otimes \K[t,dt]$ such that $\Eval_{t=0}\circ H = \Phi$ and $\Eval_{t=1}\circ H = \Psi$.
\end{definition}
\noindent A basic fact is that $\Li$ homotopy is an equivalence relation
(see \cite[\S 2]{Fukaya-deformation}).
\begin{definition}
Two $\Li$ algebras $V$ and $V'$ are {\em homotopy equivalent} if there exist $\Li$ homomorphisms $\Phi: V \rightarrow V'$ and $\Psi: V' \rightarrow V$ such that the compositions $\Psi \circ \Phi$ and $\Phi \circ \Psi$ are homotopic to the identity
(i.e. the $\Li$ endomorphism whose component with one input and one output is the identity and whose higher components are trivial).
\end{definition}

\subsection{Filtered $\Li$ algebras}\label{subsec:filtered_Linf_algebras}
 
An $\Li$ homomorphism $\Phi$ from $V$ to $V'$ is a {\em quasi-isomorphism} if the homology level map $H(V,\ell^1_V) \rightarrow H(V',\ell^1_{V'})$ induced by $\Phi^1$ is an isomorphism.
The following is a standard consequence of the homological perturbation lemma:
\begin{thm}\label{quasi-iso implies htpy eq}
If $\K$ is a field, an $\Li$ quasi-isomorphism of $\Li$ algebras over $\K$ is automatically an $\Li$ homotopy equivalence.
\end{thm}
\noindent In particular, any $\Li$ algebra which is acyclic (with respect to $\ell^1$) is homotopy equivalent to the trivial $\Li$ algebra.
This means that $\Li$ algebras up to homotopy equivalence, at least over a field, are not powerful enough to model the geometric situations we have in mind.
For example, for any star-shaped domain in $\C^n$, the $\Li$ structures on both symplectic cochains and $S^1$-equivariant symplectic cochains are acyclic, even though there are plenty of non-symplectomorphic pairs.
To remedy this and incorporate more quantitative symplectic geometric input, we now consider $\Li$ algebras with filtrations.

\sss

Let $R$ be a commutative ring. 
We say that $R$ has a {\em decreasing filtration} if there are $R$-linear subspaces $\calF_{\geq a}R \subset R$ for $a \in \R$, with
\begin{itemize}
\item
$\calF_{\geq a}R \subset \calF_{\geq a'}R$ for $a > a'$
\item
$(\calF_{\geq a}R) \cdot (\calF_{\geq a'}R) \subset \calF_{\geq a + a'}R$
\item
$\underset{a\in\R}\cup \calF_{\geq a}R = R$.
\end{itemize}
The filtration makes $R$ into a topological ring, wherein a basis of neighborhoods at $x \in R$ is given by $x + \calF_{\geq a}R$, $a \in \R$. A sequence $x_i \in R$ is said to be {\em Cauchy} if for any $a$ there exists $N$ such that $x_i - x_j \in \calF_{\geq a}\Lam$ for all $i,j \geq N$, and we say that $R$ is {\em complete} if every Cauchy sequence converges.

To give our main example, suppose that $\K$ is a commutative ring containing $\Q$ as before.
The (universal) Novikov ring $\Lamo$ over $\K$ is the ring of formal series
\begin{align*}
\Lamo := \left\{ \sum_{i = 1}^{\infty} c_iT^{a_i}\;:\; c_i \in \K,\; a_i \in \R_{\geq 0},\; \lim_{i \rightarrow \infty}a_i = +\infty\right\},
\end{align*}
where $T$ is a formal variable.
The filtration is given by 
\begin{align*}
\calF_{\geq a}\Lamo := T^a\cdot\Lamo = \left\{ \sum_{i = 1}^{\infty} c_iT^{a_i}\;:\; a_i \geq a\right\}
\end{align*}
for $a \geq 0$ and $\calF_{\geq a}V := V$ for $a < 0$.
There is also the (universal) Novikov field\footnote{Strictly speaking $\Lam$ is not a field unless $\K$ is a field, in which case $\Lam$ is the field of fractions of $\Lamo$.}
$\Lam$ over $\K$, defined by 
\begin{align*}
\Lam :=\left\{ \sum_{i = 1}^{\infty} c_iT^{a_i}\;:\; c_i \in \K,\; a_i \in \R,\; \lim_{i \rightarrow \infty}a_i = +\infty\right\},
\end{align*}
which is also complete and contains $\Lamo$ as a subring.

If $V$ is a $R$-module, a decreasing filtration consists of $R$-linear subspaces
$\calF_{\geq a}V \subset V$, $a \in \R$, satisfying
\begin{itemize}
\item
 $\calF_{\geq a}V \subset \calF_{\geq a'}V$ for $a > a'$
\item
$(\calF_{\geq a}R)\cdot (\calF_{\geq a'}V) \subset \calF_{\geq a+a'}V$
\item
$\underset{a\in\R}\cup \calF_{\geq a}V = V$.
 \end{itemize}
The filtration similarly induces a topology on $V$, and we define completeness for $V$ in essentially the same way as for $R$. 
Note that a filtration on $V$ induces one on $\ovl{S}(sV)$, namely
$$\calF_{\geq a}(\odot^ksV) := \underset{a_1 + ... + a_k \geq a}\bigcup \calF_{\geq a_1}sV \odot ... \odot \calF_{\geq a_k}sV.$$
\begin{definition}\label{filtered Linf algebra}
A {\em filtered $\Li$ algebra} over $R$ is an $\Li$ algebra over $R$ whose underlying $R$-module $V$ is equipped with a decreasing filtration $\{\calF_{\geq a}V\}$ such the coderivation $\wh{\ell}: \ovl{S}(sV) \rightarrow \ovl{S}(sV)$
preserves the induced filtration on $\ovl{S}(sV)$.
\end{definition}
\noindent In terms of the operations $\ell^k: \odot^ksV \rightarrow sV$, this corresponds to having
$$\ell^k(\calF_{\geq a_1}sV \odot ... \odot \calF_{\geq a_k}sV) \subset \calF_{\geq (a_1 + ... + a_k)}sV.$$
We similarly define filtered $\Li$ homomorphisms, homotopies, and homotopy equivalences by requiring all structure maps to respect filtrations; see the aforementioned references for more details.

\sss

Note that any filtered $R$-module can be canonically completed, and we will always implicitly assume this has been done.
A filtered $\Li$ algebra $V$ is {\em complete} if the underlying $R$-module is complete. 
Filtered $\Li$ algebras have better functoriality properties, for example Maurer--Cartan elements can be pushed forward (see \S\ref{subsec:MC_theory} below).
Any filtered $\Li$ algebra can be canonically completed by completing the $R$-module $V$ and extending the $\Li$ operations, and we will also assume this has been done.
Moreover, from now on in filtered settings we take all tensor products and direct sums to be completed,
although we suppress this from the notation.
For example, for $V$ a filtered $\Li$ algebra, $\ovl{S}(sV)$ will denote the naive reduced symmetric tensor algebra of $sV$, but completed with respect to the natural filtration-induced topology.

In this paper, all of the filtered $\Li$ algebras we encounter from symplectic geometry will be defined over $\Lamo$, such that the underlying $\Lamo$-module $V$ is torsion-free.
In this case, the filtration on $V$ is actually redundant, as it is uniquely determined by $\calF_{\geq a}V = T^a\cdot V$ for $a \geq 0$ and $\calF_{\geq a}V = V$ for $a < 0$.
From now on we will always assume the filtration is of this form and will therefore typically omit the qualifier ``filtered'', unless we wish to emphasis the presence of filtrations.

\begin{remark}\label{rmk:filtrations}
One defines $\Li$ algebras with {\em increasing} filtrations in an analogous way. In the context of exact symplectic manifolds, the invariants we consider can also be defined as $\Li$ algebras over $\K$ with increasing filtrations, where $\calF_{\leq a}V$ roughly corresponds to the span of (words of) Reeb orbits with (total) action at most $a$ (c.f. \S\ref{sec:intro}). 
By contrast, we will see in \S\ref{sec:RSFT} that working with decreasing filtrations over the Novikov ring as above remembers only {\em differences} between actions.
Despite this apparent loss of information, the latter will suffice to define symplectic capacities, and is more effective when discussing e.g. nonexact symplectic manifolds or Maurer--Cartan theory.
\end{remark}

\subsection{Maurer--Cartan theory}\label{subsec:MC_theory}
Given a filtered $\Li$ algebra $V$ over $\Lamo$, an element $\m \in \calF_{>0} V := \bigcup\limits_{\eps > 0}\calF_{\geq \eps}V$ with $|s\m| = 0$ is {\em Maurer--Cartan} if it satisfies the equation
\begin{align*}
\sum_{k=1}^\infty\frac{1}{k!}\ell^k (\underbrace{s\m,...,s\m}_k) = 0.
\end{align*}
Note that all of the summands land in the degree $1$ part of $sV$; equivalently, we have $|\m| = 1$ and all of the summands land in the degree $2$ part of $V$.
In the unfiltered setting this equation involves an infinite sum and hence is actually ill-defined.
However, since $V$ is filtered we can make sense of the Maurer--Cartan equation by asking the infinite sum to converge to zero in the filtration-induced topology.
Equivalently, if we put $$\exp(s\m) := \sum_{k = 1}^\infty \frac{1}{k!}(\underbrace{s\m\odot ... \odot s\m}_k) \in \bar V,$$ 
then $\m$ satisfies the Maurer--Cartan equation if and only if $\wh{\ell}(\exp(s\m)) = 0$.

Two Maurer--Cartan elements $\m,\m' \in V$ are said to be {\em gauge-equivalent} if there exists a Maurer--Cartan element $g \in V \otimes \Lamo[t,dt]$ with $\Eval_{t = 0}(g) = \m$ and $\Eval_{t = 1}(g) = \m'$.
Also, given an $\Li$ homomorphism $\Phi: V \rightarrow V'$,
we can push forward a Maurer--Cartan element $\m \in V$ to one in $V'$ via the formula
$$\Phi_*(\m) := \sum_{k =1}^{\infty} \frac{1}{k!}s^{-1}\Phi^k(\underbrace{s\m,...,s\m}_k).$$
A basic fact (see for example \cite[\S 2]{Fukaya-deformation}) is that the set of Maurer--Cartan elements modulo gauge equivalence is a filtered $\Li$ homotopy invariant of $V$.

Given a Maurer--Cartan element $m \in V$, we can deform $\ell^1$
to a new differential $\ell^1_\m$, defined by
\begin{align*}
\ell^1_{\m}(sx) := \sum_{i=1}^{\infty} \frac{1}{i!} \ell^{1+i}(sx,\underbrace{s\m,...,s\m}_{i}).
\end{align*}
In fact, the entire $\Li$ structure gets deformed via
\begin{align*}
\ell^k_\m(sx_1,...,sx_k) := \sum_{i=1}^{\infty} \frac{1}{i!} \ell^{k+i}(sx_1,...,sx_k,\underbrace{s\m,...,s\m}_{i}).
\end{align*}
Note that since $\m$ lives in the positive part $\calF_{>0}V$ of the filtration, these deformed operations preserve same the decreasing filtration,\footnote{We point out that if we instead try to work with {\em increasing} filtrations as in Remark~\ref{rmk:filtrations} then the analogous deformed $\Li$ operations do not a priori preserve the filtration.} and one can check that they still satisfy the $\Li$ structure equations. 
Also, since $|s\m| = 0$, all of these terms have the same degree, although it is sometimes useful to relax this condition by asking only that $|s\m|$ be even, in which case case we still get a deformed $\Li$ structure but with only a $\Z/2$ grading.
We will denote the $\Li$ algebra deformed by the Maurer--Cartan element $\m$ by $V_\m$.

\begin{remark}
Maurer--Cartan elements play an important role in deformation theory, and their most standard usage is in {\em formal} deformation theory.
For $V$ an unfiltered $\Li$ algebra over $\K$, we can tensor $V$ with say the formal power series ring $R\ll q \rr$ and then use the $q$-adic filtration to make sense of the Maurer--Cartan equation in $V \otimes R\ll q \rr$.
We call such Maurer--Cartan elements {\em formal} to distinguish them from the {\em filtered} Maurer--Cartan elements considered above.
In fact, the $\Li$ algebra $V$ defines a {\em formal deformation functor} which inputs an Artinian local $\K$-algebra $A$ and outputs the set of Maurer--Cartan elements in $V \otimes A$ modulo gauge equivalence.
Deformation problems in geometry typically give rise to such formal deformation functors,
and the Deligne philosophy states that, at least in characteristic zero, every nice deformation functor should be represented by some natural $\Li$ algebra $V$ (see \cite{Lurie_deformation} for a modern treatment).
\end{remark}

\section{Rational symplectic field theory}\label{sec:RSFT}

In this section we describe the basic geometry and algebra of rational symplectic field theory, which will serve as the backbone for the subsequent sections of this paper. After explaining the general geometric setup, defining punctured pseudoholomorphic curves, and recalling their compactness properties, we finally present the relevant RSFT formalism. Our approach is somewhat nonstandard in that we phrase everything in terms of filtered $\Li$ algebras, giving a language which is particularly useful for quantitative applications, and also for identifying parallels in Floer theory. We will subsequently build on this formalism in \S\ref{sec:geometric_constraints} by introducing additional $\Li$ augmentations, which we use in \S\ref{sec:construction_of_capacities} to construct symplectic capacities.

Symplectic field theory is an elaborate algebraic formalism which packages together counts of punctured pseudoholomorphic curves in symplectic manifolds with certain types of cylindrical ends.
After Gromov's seminal compactness theorem and the subsequent development of Gromov--Witten theory, Floer homology, and so on,
it was discovered that punctured pseudoholomorphic curves in symplectic cobordisms also admit a good compactness theory \cite{Hofer_pseudoholomorphic_curves_in_symplectizations, compactness_results_in_SFT}. 
In \cite{EGH2000}, Eliashberg--Givental--Hofer outlined the general formalism of symplectic field theory, which incorporates pseudoholomorphic curves of all genera with arbitrary numbers of positive and negative punctures. 
In a spirit akin to topological field theories, one associates data to both contact manifolds and to symplectic cobordisms between contact manifolds.
A key tool in SFT is the ``stretching the neck'' procedure, which allows one to cut symplectic manifolds along contact type hypersurfaces and thereby decompose symplectic manifolds into potentially simpler pieces. For example, this procedure makes it possible to decompose the Gromov--Witten invariants of a closed symplectic manifold
into the SFT invariants of a symplectic filling and a symplectic cap.
Eliashberg--Givental--Hofer also described simplified invariants such as rational symplectic field theory, which involves only genus zero curves, and the contact homology algebra, which further restricts to rational curves with only one positive end.
One can also consider relative analogues, replacing contact manifolds by Legendrian submanifolds and symplectic cobordisms by Lagrangian cobordisms between Legendrians.

Since its introduction, SFT has produced a large number of impressive geometric applications, with many more still expected. 
Most but not all applications so far have centered around the contact homology algebra  and its linearized version, and especially their relative cousins.
In general, computing the full SFT in examples and making fruitful use of its elaborate structure, especially in the presence of higher genus curves, rotating asymptotic markers, chain level structures, etc is still a highly active area of research.

\begin{remark}[on transversality and virtual techniques in SFT]\label{rmk:transversality}
As already mentioned in Disclaimer \ref{disclaimer_from_intro}, 
it is well-known that the compactified moduli spaces of curves involved in SFT are not 
generally transversely cut out within a classical perturbation framework.
The typical situation is that, for a generic choice of almost complex structure $J$, all relevant somewhere injective curves will be regular, but their multiple covers might appear with higher than expected dimension.
For example, one often encounters negative index branched covers of trivial cylinders in symplectizations which necessarily persist for any choice of almost complex structure.
Even in the seemingly innocuous case of cobordisms between four-dimensional ellipsoids, some of the counts one wishes to perform might be represented by multilevel buildings. Such configurations have expected codimension at least $1$ and hence ought not to exist for a generic 
$J$, if transversality were to hold.

This means that the structures outlined in \cite{EGH2000} can only be achieved by interpreting curve counts in a suitable virtual sense.
The situation is similar in Gromov--Witten theory or Floer theory, although in those cases one can often get away with classical perturbation techniques under additional topological assumptions such as semipositivity.
There are several different virtual frameworks being developed to achieve the goals of SFT.
See for example \cite{Pardon_algebraic_approach,polyfold_and_fredholm_theory,ishikawa2018construction,fish2018lectures} and the references therein for more details.
We note that arguments in this paper are based on the general structural properties of rational symplectic field theory 
 and not on the particulars of any virtual perturbation scheme.
For some of the computations given in \S\ref{subsec:first_computations}, we implicitly invoke the standard assumption that, for generic $J$, any curve which is already regular will persist in the presence of virtual perturbations.\footnote{This is sometimes referred to as the ``Obamacare axiom'', the slogan being ``if you like your curves you can keep them''.}
\end{remark}
\begin{addendum}
In the followup paper \cite{enumerating2}, we construct another sequence of capacities, denoted by $\{\widetilde{\gapac}_k\}$, which circumvents SFT transversality and agrees with the sequence $\{\gapac_k\}$ in many situations. We expect a similar approach to extend to many of the other capacities introduced in this paper.
\end{addendum}

\subsection{Geometric setup}\label{subsec:geometric_setup}
In order to make sense of symplectic invariants defined in terms of pseudoholomorphic curves, we will need to restrict the class of symplectic manifolds under consideration and choose various pieces of auxiliary data (e.g. an almost complex structure $J$). 
Adapting the general setup of \cite{EGH2000}, we focus our attention on certain symplectic cobordisms between contact manifolds.
The formalism is simplest when these cobordisms are exact, although we also explain in \S\ref{sec:CL} how this can be extended to the case of non-exact cobordisms.

Recall that a {\em contact form} $\alpha$ on an odd-dimensional oriented manifold $Y^{2n-1}$ is a $1$-form such that $\alpha \wedge (d\alpha)^{\wedge(n-1)} > 0$. A (co-oriented) contact structure on $Y$ is a hyperplane distribution $\xi \subset TY$ of the form $\xi = \ker\alpha$ for some contact form $\alpha$. By ``strict contact manifold'' we will mean a pair $(Y,\alpha)$ consisting of a manifold $Y$ and a contact one-form $\alpha$.
The contact form $\alpha$ is preserved by a canonical vector field $R_\alpha$ called the Reeb vector field, 
defined by the properties $d\alpha(R_\alpha,-) = 0$ and
$\alpha(R_\alpha) = 1$. For $T > 0$, a $T$-periodic orbit is by definition a $T$-periodic trajectory of $R_\alpha$, i.e. a path
$\gamma: [0,T] \rightarrow Y$ such that $\dot{\gamma}(t) = R_\alpha(\gamma(t))$ and $\gamma(T) = \gamma(0)$.
The period of $\gamma$ is equivalently given by $\int \gamma^*\alpha$, and 
we also say that $\gamma$ has {\em action} $\calA_\alpha(\gamma) = T$.
{\em Note: in the sequel we will frequently omit decorative subscripts and superscripts unless needed to clarify ambiguities, for example denoting $\calA_{\alpha}$ by simply $\calA$. We will also often suppress symplectic forms, contact forms, etc from the notation when no confusion should arise.}

Every Reeb orbit $\gamma$ has an underlying {\em simple} (i.e. embedded) Reeb orbit $\ovl{\gamma}$.
We denote by $\kappa_\gamma$ the covering multiplicity of $\gamma$, and sometimes write $\gamma = \ovl{\gamma}^{\kappa_{\gamma}}$.
We will assume throughout that every simple Reeb orbit $\ovl{\gamma}$ has a chosen basepoint $m_{\ovl{\gamma}}$ in its image.
This also induces a basepoint on the image of any multiple cover $\ovl{\gamma}^{\kappa_{\gamma}}$.
Note that parametrized Reeb orbits come in $S^1$-families, but from now we will consider all orbits in the same $S^1$-family to be equivalent and work only with the representative having $\gamma(0)$ as its basepoint.
\begin{remark}
Since we do not distinguish between Reeb orbits in the same $S^1$-family, we are effectively working with unparametrized Reeb orbits everywhere, and indeed SFT is ``by default'' an $S^1$-equivariant theory. In fact it is also possible to use SFT to construct non-equivariant invariants by a more refined use of asymptotic markers, c.f. Bourgeois--Oancea's model for symplectic cohomology \cite{bourgeois2009exact}.
\end{remark}

Observe that $d\alpha$ makes $\ker \alpha$ into a $(2n-2)$-dimensional symplectic vector bundle over $Y$.
Given a trivialization $\triv$ of this symplectic vector bundle over a Reeb orbit $\gamma$, the linearized Reeb flow along $\gamma$ can be identified with a path of $(2n-2)\times(2n-2)$ symplectic matrices starting at the identity. 
Unless otherwise stated, we will always assume the contact form $\alpha$ is nondegenerate, which means this path ends on a matrix without $1$ as an eigenvalue.
Associated to any such path of symplectic matrices is its {\em Conley--Zehnder index} $\cz_\triv(\gamma) \in \Z$,
which depends on the homotopy class of the trivialization $\triv$.
In fact, the parity of $\cz_\triv(\gamma)$ is determined just by the sign of $\mathbb{I} - P_\gamma$, where $P_\gamma$ is the linearized Poincar\'e return map of $\gamma$, and hence it does not depend on any choice of trivialization.

To any contact manifold $(Y,\alpha)$ we can associate the symplectization, i.e. the symplectic manifold $\R_s \times Y$ with exact symplectic form 
$d(e^s\alpha)$. 
We will also consider the half-symplectizations given by $(\R_\pm \times Y,d(e^s\alpha))$, where 
$\R_+ = [0,\infty)$ and $\R_- = (-\infty,0]$ (with $s$ still used to denote the coordinate on the first factor).
Due to the presence of a negative end, symplectizations do not satisfy the original ``geometrically bounded'' criterion of Gromov's compactness theorem. Nevertheless, it turns out there is a good compactness theory for pseudoholomorphic curves in symplectizations, and more generally in symplectic manifolds with ends modeled on positive or negative half-symplectizations.
\begin{definition}
A {\em compact symplectic cobordism} is a compact symplectic manifold $(X,\omega)$ equipped with a contact form $\alpha$ on $\bdy X$ such that $\alpha = \lam|_{\bdy X}$ for some primitive one-form $\la$ for $\omega$ defined near $\bdy X$.
A boundary component of $X$ is called {\em positive} if $\alpha \wedge (d\alpha)^{\wedge(n-1)}$ is positive with respect to the boundary orientation, otherwise it is called {\em negative}. We denote by $\bdy^+X$ the union of all positive boundary components of $X$ and $\alpha^+ := \alpha|_{\bdy^+X}$ 
the induced contact form, and similarly for $\bdy^-X$ and $\alpha^-$.
\end{definition}
\begin{remark}\label{rmk:caps_are_symp_invt}
We will often suppress the contact form $\alpha$ from the notation, or we will write $(X,\omega,\alpha)$ when we wish to make its role explicit.
Observe that the value of a symplectic capacity $\mathfrak{c}$ on a symplectic filling $(X,\omega,\alpha)$ is manifestly independent of the choice of $\alpha$ (due to monotonicity under symplectic embeddings), and the same is true for an exact symplectic capacity provided that $H^1(\bdy X;\R) = 0$ (c.f. Remark~\ref{rmk:auto_exact} below).
\end{remark}
\begin{definition}
A compact symplectic cobordism with $\bdy^-X = \nil$ is called a {\em symplectic filling} of $(\bdy^+X,\alpha^+)$,
 and a compact symplectic cobordism with $\bdy^+X = \nil$ is called a {\em symplectic cap} of $(\bdy^-X,\alpha^-)$. 
\end{definition}

Given a compact symplectic cobordism $(X,\omega)$ with $\omega|_{\Op(\bdy X)} = d\lam$,
we can form a {\em completed symplectic cobordism} (also called a {\em symplectic manifold with cylindrical ends} in \cite{EGH2000}) by attaching a positive half-symplectization $(\R_+\times \bdy^+X,d(e^s\alpha^+))$ to the positive boundary of $X$ and attaching a negative half-symplectization $(\R_-\times \bdy^-X,d(e^s\alpha^-))$ to the negative boundary of $X$.
Indeed, a standard observation is that the {\em Liouville vector field} $Z_\lam$, defined near $\bdy X$ by the condition $Z_\lam \ip d\lam = \lam$, gives a natural collar coordinate $s$ near $\bdy X$, with respect to which $\lam$ looks like $e^s\alpha$, and hence can be smoothly extended to the cylindrical ends. Note that $Z$ is outwardly transverse to each positive boundary component of $X$ and inwardly transverse to each negative boundary component of $X$. 
We denote the completion by $(\wh{X},\wh{\omega})$.

A compact symplectic cobordism is called {\em exact} if the one-form $\lam$ extends to a primitive for $\omega$ on all of $X$. 
This is also called a {\em Liouville cobordism} (with $\lam$ typically viewed as part of the data), and the case with no negative end is known as a {\em Liouville domain}.
\begin{definition}\label{def:exact_symp_emb}
Given symplectic fillings $(X,\omega,\alpha)$ and $(X',\omega',\alpha')$, an {\em exact symplectic embedding} is a symplectic embedding $\iota: X \hookrightarrow X'$ 
such that $Q := X' \setminus \iota(\op{Int}\, X)$ is an exact symplectic cobordism with respect to the primitives
$\alpha'$ on $\bdy X'$ and $\iota_*\alpha$ on $\iota(\bdy X)$.  
\end{definition}
\NI This is equivalent to the existence of a one-form $\theta \in \Omega^1(Q)$ such that $d\theta = \omega'$, $\theta|_{\bdy X'} = \alpha'$, and $\theta|_{\iota(\bdy X)} = \iota_*\alpha$.

\begin{lemma}
 Let $(X,\omega)$ be a symplectic filling of $(\bdy X,\alpha)$ and let $(X',\la')$ be a Liouville domain. Given a symplectic embedding $\iota: X \hooksymp X'$, put $Q := X' \setminus \iota(\op{Int}\, X)$, $\bdy^+Q := \bdy X'$, and $\bdy^-Q := \iota(\bdy X)$.
Then $\iota$ is exact if and only if $[\la'|_{\bdy^-Q} - \iota_*\alpha] \in H^1(\bdy^- Q;\R)$ lies in the image of restriction map
 $H^1(Q,\bdy^+Q;\R) \ra H^1(\bdy^-Q;\R)$. 
\end{lemma}
\begin{proof}
 If $\iota$ is exact, then by definition there exists $\theta \in \Omega^1(Q)$ such that $d\theta = \omega'$, $\theta|_{\bdy^+Q} = \alpha' := \la'|_{\bdy^+Q}$, and $\theta|_{\bdy^-Q} = \iota_*\alpha$.
Then $\eta := \la' - \theta \in \Omega^1(Q)$ is closed and satisfies $\eta|_{\bdy^+Q} \equiv 0$ and $\eta|_{\bdy^-Q} = \la'|_{\bdy^-Q} - \iota_*\alpha$.

Conversely, let $\eta \in \Omega^1(Q)$ be a closed one-form such that 
$\eta|_{\bdy^+Q} \equiv 0$ and $[\eta|_{\bdy^-Q}] = [\la'|_{\bdy^-Q} - \iota_*\alpha] \in H^1(\bdy^-Q;\R)$.
After a small modification (by adding an exact one-form supported near $\bdy^-Q$) we can assume that $\eta|_{\bdy^- Q} = \iota_*\alpha$, whence
$\theta := \la' - \eta$ satisfies $d\theta = \omega'$, $\theta|_{\bdy^+Q} = \alpha'$, and $\theta|_{\bdy^-Q} = \iota_*\alpha$.
\end{proof}
\begin{cor}\label{cor:auto_exact}
A symplectic embedding of a symplectic filling $(X,\omega,\alpha)$ into a Liouville domain $(X,\la')$ is automatically exact provided that $H^1(\bdy X;\R) = 0$. 
\end{cor}
\begin{remark}\label{rmk:auto_exact}
Given Liouville domains $(X,\la),(X',\la')$, a symplectic embedding $\iota: X \hooksymp X'$ is automatically exact provided that $H^1(X;\R) = 0$, or more generally that the restriction map $H^1(X;\R) \ra H^1(\bdy X;\R)$ vanishes.
Indeed, in this case $[\la'|_{\bdy^-Q} - \iota_*\alpha] \in H^1(\bdy^-Q;\R)$ vanishes if and only if $[\iota^*\la'|_{\bdy X} - \alpha] = [(\iota^*\la'| - \la)|_{\bdy X}] \in H_1(\bdy X;\R)$ vanishes, and the latter clearly lies in the image of $H^1(X;\R) \ra H^1(\bdy X;\R)$.
Note also that any generalized Liouville embedding in the sense of \cite[Def. 1.23]{Gutt-Hutchings_capacities} (i.e. $[(\iota^*\la' - \la)|_{\bdy X} \in H^1(\bdy X;\R)$ vanishes) is exact.
\end{remark}

\begin{remark}
Our definition of compact symplectic cobordism is sometimes called a {\em strong symplectic cobordism}. There are several other important types of symplectic cobordisms commonly appearing in the literature, for example {\em weak symplectic cobordisms}, {\em Weinstein cobordisms}, and {\em Stein cobordisms}. 
See \cite{Massot-Nieder-Wendl_fillings} for a more comprehensive discussion.

\end{remark}
\begin{remark}[on gradings of orbits in symplectic fillings]\label{rmk:on_gradings_in_symplectic_fillings}
Suppose that $(X,\omega)$ is a symplectic filling of $(Y,\alpha)$.
As mentioned above, each Reeb orbit of $(Y,\alpha)$ has a canonical $\Z/2$ grading.
In some cases, this can naturally be upgraded.
For example, if we assume that $2c_1(X,\omega) = 0$ and $H_1(X;\Z) = 0$, 
then each Reeb orbit $\gamma$ of $(Y,\alpha)$ bounds a spanning surface in $X$, and there is a unique (up to homotopy) trivialization $\triv$
of the symplectic vector bundle $(TX,\omega)$ over $\gamma$ which extends over this spanning surface.
Two different choices of spanning surfaces define an element $A \in H_2(X;\Z)$, and the two resulting Conley--Zehnder indices of $\gamma$ differ by $2c_1(X,\omega)\cdot A$. Therefore, in this case each Reeb orbit has a canonical Conley--Zehnder index $\cz(\gamma) \in \Z$.

More generally, if $2c_1(X,\omega) = 0$ and $H_1(X,\omega)$ is torsion-free, we can still assign an integer Conley--Zehnder index to each Reeb orbit, although it is no longer canonical. Rather, the space of choices is a torsor over $H_1(X;\Z)$, corresponding to a choice of trivialization of the symplectic vector bundle $(TX,\omega)$ over a set of basis elements for $H_1(X;\Z)$.
\end{remark}

Now suppose we have a symplectic cobordism $(X_1,\omega_1)$ with positive end $(\bdy^+X_1,\alpha_1^+)$ and another symplectic cobordism $(X_2,\omega_2)$ with negative end $(\bdy^-X_2,\alpha_2^-)$.
Assuming there is a strict contactomorphism $(\bdy^+X_1,\alpha_1^+) \cong (\bdy^-X_2,\alpha_2^-)$ (i.e. a diffeomorphism $\bdy^+X_1 \cong \bdy^-X_2$ pulling back $\alpha_2^-$ to $\alpha_1^+$),
 we can use the collar coordinates to glue the positive boundary of $X_1$ with the negative boundary of $X_2$, forming a new 
symplectic cobordism $(X_1,\omega_1) \circledcirc (X_2,\omega_2)$. In particular, if $(X_1,\omega_1)$ is a filling and $(X_2,\omega_2)$ is a cap, then $(X_1,\omega_1) \circledcirc (X_2,\omega_2)$ is a closed symplectic manifold.
One of the key features of SFT is that curves in such a glued cobordism can be understood in terms of curves in either of the two pieces.

Suppose that $(Y,\alpha)$ is a strict contact manifold, 
and let $\Gamma^+ = (\gamma_1^+,...,\gamma_{s^+}^+)$ and $\Gamma^- = (\gamma_1^-,...,\gamma_{s^-}^-)$
be tuples of Reeb orbits.
We denote by $H_2(Y,\Gamma^+\cup\Gamma^-;\Z)$ the homology group of $2$-chains $\Sigma$ in $Y$ with 
$\bdy \Sigma = \sum_{i=1}^{s^+}\gamma_i^+ - \sum_{j=1}^{s^-}\gamma_j^-,$
modulo boundaries of $3$-chains in $Y$.
Note that this group is a torsor over $H_2(Y;\Z)$.
Integrating over the closed two-form $d\alpha$ defines a symplectic energy homomorphism
\begin{align*}
[d\alpha]\cdot -:H_2(Y,\Gamma^+\cup \Gamma^-;\Z) \rightarrow \R.
\end{align*}
In fact, by Stokes' theorem we have simply 
$$[d\alpha]\cdot A = \sum_{i=1}^{s^+}\calA_\alpha(\gamma_i^+) - \sum_{j=1}^{s^-}\calA_\alpha(\gamma_j^-).$$
\begin{remark}
The notion of symplectic energy which we consider is purely homological and will mostly only be considered in the context of $J$-holomorphic curves or buildings. Note that it is not the same as symplectic area, which
would be given by integrating $d(e^s\alpha)$ and is typically infinite for the curves we consider. 
\end{remark}
\noindent Also, after choosing a trivialization $\triv$ of the symplectic vector bundle $(\ker \alpha,d\alpha)$ 
over the Reeb orbits of $\Gamma^+$ and $\Gamma^-$,
we have a relative first Chern class homomorphism
$$c_1^\triv(Y,\alpha) \cdot -: H_2(Y,\Gamma^+\cup\Gamma^-;\Z) \rightarrow \Z.$$
For $[\Sigma] \in H_2(Y,\Gamma^+\cup\Gamma^-;\Z)$, 
$c_1^\triv(Y,\alpha)\cdot [\Sigma]$ is the obstruction to finding a nonvanishing section of $(\ker \alpha,d\alpha)$ over $\Sigma$ which is constant along the ends with respect to the trivialization $\triv$.

Similarly, suppose that $(X,\omega)$ is a symplectic cobordism between $(Y^+,\alpha^+)$ and $(Y^-,\alpha^-)$,
and let $\Gamma^\pm = (\gamma_1^\pm,...,\gamma_{s^\pm}^\pm)$ be a tuple of Reeb orbits of $(Y^\pm,\alpha^\pm)$.
Let $H_2(X,\Gamma^+ \cup \Gamma^-;\Z)$ denote the homology group of $2$-chains $\Sigma$ in $X$ with $$\bdy \Sigma = \sum_{i=1}^{s^+}\gamma_i^+ - \sum_{j=1}^{s^-}\gamma_j^-,$$
modulo boundaries of $3$-chains in $X$.
Integrating over the symplectic form $\omega$ defines a symplectic energy homomorphism
$$ [\omega] \cdot -: H_2(X,\Gamma^+\cup\Gamma^-;\Z) \rightarrow \R.$$
If $(X,\omega)$ is an {\em exact} symplectic cobordism, by Stokes' theorem we have
$$[\omega]\cdot A =  \sum_{i=1}^{s^+}\calA_{\alpha^+}(\gamma_i^+) - \sum_{j=1}^{s^-}\calA_{\alpha^-}(\gamma_j^-).$$
Given a trivialization $\triv$ of the symplectic vector bundle $(TX,\omega)$,
we also have a relative first Chern class homomorphism
$$c_1^\triv(X,\omega) \cdot-: H_2(X,\Gamma^+ \cup \Gamma^-;\Z) \rightarrow \Z.$$

In the situation of a glued symplectic cobordism $(X_1,\omega_1)\circledcirc(X_2,\omega_2)$
with $\Gamma_1^+ = \Gamma_2^-$,
there is a natural concatenation map 
\begin{align*}
H_2(X_1,\Gamma_1^+\cup\Gamma_1^-;\Z) \otimes H_2(X_2,\Gamma_2^+\cup\Gamma_2^-;\Z) &\rightarrow H_2(X_1 \circledcirc X_2,\Gamma_2^+\cup \Gamma_1^-;\Z)\\
A_1 \otimes A_2 &\mapsto A_1\circledcirc A_2.
\end{align*}
There are also similar concatenation maps for gluing a symplectization to a symplectic cobordism or gluing two symplectizations, of the form
\begin{align*}
H_2(X,\Gamma_1^+\cup\Gamma_1^-;\Z) \otimes H_2(\bdy^+X,\Gamma_2^+\cup\Gamma_2^-;\Z) \rightarrow H_2(X,\Gamma_2^+\cup\Gamma_1^-;\Z)\\
H_2(\bdy^- X,\Gamma_1^+\cup\Gamma_1^-;\Z) \otimes H_2(X,\Gamma_2^+\cup\Gamma_2^-;\Z) \rightarrow H_2(X,\Gamma_2^+\cup\Gamma_1^-;\Z)\\
H_2(Y,\Gamma_1^+\cup\Gamma_1^-;\Z) \otimes H_2(Y,\Gamma_2^+\cup\Gamma_2^-;\Z) \rightarrow H_2(Y,\Gamma_2^+\cup\Gamma_1^-;\Z).
\end{align*}
The symplectic energy homomorphisms are additive with respect to these concatenation maps, as are the relative first Chern class maps (provided we pick appropriately compatible trivializations).

When discussing pseudoholomorphic curves, it will ocassionally be convenient
to let $\ovl{\wh{X}}$ denote the compactification of $\wh{X}$ given by replacing $\R_+ = [0,\infty)$ with $\ovl{R}_+ := [0,+\infty]$
and $\R_- = (-\infty,0]$ with $\ovl{\R}_- := [-\infty,0]$.
We of course have a natural identification $H_2(\ovl{\wh{X}},\Gamma^+\cup\Gamma^-;\Z) \cong H_2(X,\Gamma^+\cup\Gamma^-;\Z)$.
We will be considering punctured curves in $\wh{X}$ which extend by continuity to curves $u: \Sigma \rightarrow \ovl{\wh{X}}$, and hence naturally define classes in $A_u \in H_2(X,\Gamma^+\cup\Gamma^-;\Z)$.

\sss

\begin{remark}[on symplectizations as cobordisms]
For $(Y,\alpha)$ a strict contact manifold, the symplectization $(\R_s \times Y,d(e^s\alpha))$ looks roughly like the completion of the compact symplectic cobordism
$(X_M,\omega_M) := ([-M,M]\times Y,d(e^s\alpha))$ as $M \rightarrow 0$. 
However, the existence of a global $\R$-action by $s$-translations means that pseudoholomorphic curves in symplectizations behave somewhat differently,
and we also define symplectic energy differently in symplectizations. 
\end{remark}

\subsection{Punctured pseudoholomorphic curves}\label{subsec:punc_pseudo}

Given a strict contact manifold $(Y,\alpha)$,
an almost complex structure $J$ on the symplectization ${(\R_s \times Y,d(e^s\alpha))}$ is {\em admissible} if 
\begin{itemize}
\item
$J$ is invariant under $s$-translations and restricts to a $d\alpha$-compatible almost complex structure on $\ker\alpha$
\item $J$ sends $\bdy_s$ to $R_\alpha$.
\end{itemize}

For a Riemann surface $\Sigma$ and a point $p \in \Sigma$, an {\em asymptotic marker at $p$} is a real half-line in $T_p\Sigma$. 
Note that there is a circle's worth of choices of asymptotic markers at $p$.
Given an asymptotic marker at $p$, there is a contractible space of holomorphic identifications of $\Op(p)$ with the open unit disk $\D^2 \subset \C$ in such a way that the half-line at $p$ is identified with the positive real direction at the origin.
Endowing $\D^2$ with standard polar coordinates $(r,\theta)$, any such identification induces local polar coordinates on $\Op(p)$. 

Now fix an admissible $J$, and consider a collection of Reeb orbits $\gamma_1^+,...,\gamma_{s^+}^+$ 
and $\gamma_1^-,...,\gamma^-_{s^-}$ (possibly with repeats) with respect to $\alpha$.
A {\em pseudoholomorphic curve in the symplectization $\R \times Y$ with positive ends $\gamma_1^+,...,\gamma_{s^+}^+$ and negative ends $\gamma_1^-,...,\gamma^-_{s^-}$} 
consists of:
\begin{itemize}
\item
a closed Riemann surface $\Sigma$, with almost complex structure denoted by $j$ 
\item 
a collection of pairwise distinct points $z^+_1,...,z^+_{s^+},z_1^-,...,z_{s^-}^- \in \Sigma$, each equipped with an asymptotic marker
\item
a map $u: \dot{\Sigma} \rightarrow \R \times Y$
satisfying $du \circ j = J \circ du$, where $\dot{\Sigma}$ denotes the punctured Riemann surface $\Sigma \setminus \{z_1^+,...,z_{s^+}^+,z_1^-,...,z_{s^-}^-\}$
\item 
for each $z_i^+$, with corresponding coordinates $(r,\theta)$ compatible with the asymptotic marker, we have
$\lim\limits_{r \rightarrow 0}(\pi_\R\circ u)(re^{i\theta}) = +\infty$ and $\lim\limits_{r \rightarrow 0}(\pi_Y \circ u)(re^{i\theta}) = \gamma_i(\tfrac{1}{2\pi}T_i^+\theta)$
\item 
for each $z_j^-$, with corresponding coordinates $(r,\theta)$ compatible with the asymptotic marker, we have $\lim\limits_{r \rightarrow 0}(\pi_\R\circ u)(re^{i\theta}) = -\infty$ and $\lim\limits_{r \rightarrow 0}(\pi_Y \circ u)(re^{-i\theta}) = \gamma_j(\tfrac{1}{2\pi}T_j^-\theta)$.
\end{itemize}
Here $\pi_\R: \R \times Y \rightarrow \R$ and $\pi_Y: \R \times Y  \rightarrow Y$ denote the two projection maps defined on the ends, and $T_k^\pm$ denotes the period of $\gamma_k^{\pm}$.
We will freely identity the marked Riemann surface $\Sigma$ with the punctured Riemann surface $\dot{\Sigma}$, which carry essentially the same data, and correspondingly we also refer to the marked points $z_k^\pm$ as punctures.
For brevity, we sometimes refer to the above data as simply a ``curve'' and denote it by slight abuse of notation by $u: \Sigma \rightarrow \R \times Y$.

For shorthand, put $\Gamma^{\pm} = (\gamma_1^\pm,...,\gamma_{s^{\pm}}^\pm)$.
Let $\calM_{Y,g}^J(\Gamma^+;\Gamma^-)$ denote the moduli space of $J$-holomorphic curves in $\R \times Y$ with positive ends $\gamma_1^+,...,\gamma_{s^+}^+$ and negative ends $\gamma_1^-,...,\gamma_{s^-}^-$, with domain varying over all connected Riemann surfaces $\Sigma$ of genus $g$ with varying punctures $z_1^+,...,z_{s^+}^+,z_1^-,...,z_{s^-}^-\in \Sigma$ and varying asymptotic markers.
In this moduli space two such curves are equivalent if the corresponding holomorphic maps differ by a biholomorphic reparametrization of their domains 
sending one ordered set of punctures with asymptotic markers to the other. 
A curve $u: \Sigma \rightarrow \R \times Y$ is called {\em simple} if it 
does not factor as $u' \circ f$ for a holomorphic map $f: \Sigma \rightarrow \Sigma'$ and another curve $u': \Sigma' \rightarrow \R \times Y$.
A curve $u$ which does admit such a covering is called a {\em multiple cover} of $u'$.

For generic admissible $J$, the expected dimension of the moduli space 
$\Gamma^{\pm} = (\gamma_1^\pm,...,\gamma_{s^{\pm}}^\pm)$
is given by the index formula
$$\ind(u) = (n-3)(2-2g - s^+ - s^-) + \sum_{i=1}^{s^+}\cz(\gamma_i^+) -  \sum_{j=1}^{s^-}\cz(\gamma_j^-) + 2c_1(Y,\alpha)\cdot A_u.$$
Here $A_u \in H_2(Y,\Gamma^+ \cup \Gamma^-;\Z)$ is the homology class naturally associated to the curve $u$.
Note that although the Conley--Zehnder index and relative first Chern class terms both depend on choices of trivializations, the full expression for $\ind(u)$ does not.

An important feature of the moduli space $\calM_{Y,g}^J(\Gamma^+;\Gamma^-)$ is that it admits a free $\R$-action given by translations in the $s$ coordinate, since $J$ itself is globally $s$-translation invariant.
By a Sard--Smale type argument, for generic admissible $J$ the quotient space $\calM_{Y,g}^J(\Gamma^+;\Gamma^-)/\R$ near a simple curve is a smooth orbifold of dimension
$$(n-3)(2-2g - s^+ - s^-) + \sum_{i=1}^{s^+}\cz(\gamma_i^+) -  \sum_{j=1}^{s^-}\cz(\gamma_j^-) + 2c_1(Y,\alpha)\cdot A_u - 1.$$
Note that for $u$ regular, the existence of a free $\R$-action implies that $\ind(u) \geq 1$ unless $u$ is a union of trivial cylinders. Here by {\em trivial cylinder} we mean a 
pseudoholomorphic curve $u:\R \times S^1 \rightarrow \R \times Y$ such that $\pi_Y\circ u$ is constant in the $\R$ direction of the domain.

\sss

Similarly, let $(X,\omega)$ be a compact symplectic cobordism, with $\lam$ a primitive for $\omega$ on $\Op(\bdy X)$ and 
$\alpha_{\pm} := \lam|_{\bdy_{\pm}X}$ the induced contact forms. Let $(\wh{X},\wh{\omega})$ denote the completion.
An almost complex structure $J$ on $\wh{X}$ is {\em admissible} if 
\begin{itemize}
\item $J$ is compatible with $\wh{\omega}$
\item on $\wh{X} \setminus X$, $J$ is invariant under $s$-translations and restricts to a $d\alpha^{\pm}$-compatible almost complex structure on $\ker\alpha^{\pm}$
\item on the end $\R_\pm \times \bdy_{\pm} X$, $J$ sends $\bdy_s$ to $R_{\alpha_{\pm}}$.
\end{itemize}
Given Reeb orbits $\Gamma^+ = (\gamma_1^+,...,\gamma_{s^+}^+)$ with respect to $\alpha^+$ and $\Gamma^- = (\gamma^-_1,...,\gamma^-_{s^-})$ with respect to $\alpha^-$, we define the moduli space $\calM_{X,g}^J(\Gamma^+;\Gamma^-)$ in the same way as above, except that now the target space is the completed symplectic cobordism $\wh{X}$ rather than the symplectization $\R \times Y$.
For generic admissible $J$, near a simple curve $u$ the moduli space $\calM_{X,g}^J(\Gamma^+;\Gamma^-)$ is a smooth orbifold of dimension
$$\ind(u) = (n-3)(2-2g - s^+ - s^-) + \sum_{i=1}^{s^+}\cz(\gamma_i^+) -  \sum_{j=1}^{s^-}\cz(\gamma_j^-) + 2c_1(X,\omega) \cdot A_u,$$
and there is no longer any $\R$ action due to the lack of $s$-translation symmetry.
As before, $A_u \in H_2(X,\Gamma^-\cup\Gamma^-;\Z)$ is the homology class associated to $u$.

We define $\calM_{Y,g,A}^J(\Gamma^+;\Gamma^-) \subset \calM_{Y,g}^J(\Gamma^+;\Gamma^-)$ as the subspace of curves with associated homology class $A_u = A \in H_2(Y,\Gamma^+\cup\Gamma^-;\Z)$.
For $r \geq 1$ we can also define $\calM_{Y,g,A,r}^J(\Gamma^+;\Gamma^-)$ in the same way except that the domain of any curve $u$ is equipped with $r$ additional pairwise distinct marked points $z_1,...,z_r \in \dot{\Sigma}$, and biholomorphic reparametrizations must respect these marked points (with their ordering).
Note that the expected dimension of $\calM_{Y,g,A,r}^J(\Gamma^+;\Gamma^-)/\R$ is $$(n-3)(2-2g - s^+ - s^-) + \sum_{i=1}^{s^+}\cz(\gamma_i^+) -  \sum_{j=1}^{s^-}\cz(\gamma_j^-) + c_1(X,\omega)\cdot A_u + 2r - 1.$$
The moduli spaces $\calM^J_{X,g,A}(\Gamma^+;\Gamma^-)$ and $\calM^J_{X,g,A,r}(\Gamma^+;\Gamma^-)$ are defined similarly.

\sss

 Given a contact manifold $(Y,\alpha)$, an admissible almost complex structure $J$, and
 a curve $u: \Sigma \rightarrow \R \times Y$ in $\calM^J_{Y,g}(\Gamma^+;\Gamma^-)$,
we have $$0 \leq \int u^*(d\alpha) = [d\alpha]\cdot A_u = \sum_{i=1}^{s^+}\calA(\gamma_i^+) - \sum_{j=1}^{s^-}\calA(\gamma_j^-),$$
where the first inequality follows from compatibility of $J$ with $d\alpha$ along $\ker \alpha$.
In particular, this implies
\begin{align}\label{eq:action_eq_symp}
\sum_{i=1}^{s^+}\calA(\gamma_i^+) = \sum_{j=1}^{s^-}\calA(\gamma_j^-) + [d\alpha]\cdot A_u \geq \sum_{j=1}^{s^-}\calA(\gamma_j^-),
\end{align}

i.e. the action is nondecreasing from the top to the bottom of a curve.
Moreover, this inequality is strict unless the curve is a union of trivial cylinders or a multiple cover thereof.

Similarly, for $(X,\omega)$ a symplectic cobordism between $(Y^+,\alpha^+)$ and $(Y^-,\alpha^-)$, $J$ an admissible almost complex structure,
and $u: \Sigma \rightarrow \wh{X}$ a curve in $\calM^J_{X,g}(\Gamma^+;\Gamma^-)$,
the symplectic energy of $u$ is nonnegative and is given by the integral
of the piecewise smooth two-form $$(d\alpha^+)|_{\R_+ \times \bdy^+X} +\omega|_{X} + (d\alpha^-)|_{\R_- \times \bdy^-X}.$$
This agrees with $[\omega] \cdot A_u$, where
$A_u \in H_2(X,\Gamma^+\cup\Gamma^-;\Z) \cong H_2(\ovl{\wh{X}},\Gamma^+\cup\Gamma^-;\Z)$ is the homology class associated to $u$. 
In the case that the symplectic cobordism $(X,\omega)$ is {\em exact}, Stokes' theorem applied thrice gives
\begin{align}\label{eq:action_eq_cob}
\sum_{i=1}^{s^+}\calA(\gamma_i^+) = \sum_{j=1}^{s^-}\calA(\gamma_j^-) + [\omega] \cdot A_u\geq \sum_{j=1}^{s^-}\calA(\gamma_j^-).
\end{align}

\subsection{SFT compactness and neck stretching}\label{subsec:SFT_comp_and_neck}

The SFT compactness theorem extends Gromov's compactness theorem to punctured curves in symplectic cobordisms. As in the case of closed curves in closed symplectic manifolds, we must add nodal configurations \`a la Kontsevich's stable map compactification. However, we must additionally allow multi-level pseudoholomorphic ``buildings'', consisting of various curves in different symplectic cobordisms whose asymptotic Reeb orbits agree end-to-end to form a chain.

More precisely, let $(Y,\alpha)$ be a strict contact manifold, and let $J$ be an admissible almost complex structure.
Let $\dcalM_{Y,g,A,r}^J(\Gamma^+;\Gamma^-)/\R$ denote 
the partially compactified version of $\calM_{Y,g,A,r}^J(\Gamma^+;\Gamma^-)/\R$ in which 
curves are no longer required to be connected and they
are allowed to develop nodes, provided that no irreducible component is a constant sphere with less than three special points, 
a constant torus with no special points, or a disjoint union of trivial cylinders without marked points.
The SFT compactness theorem \cite{compactness_results_in_SFT} in this situation states that any sequence of curves $u_1,u_2,u_3,... \in \calM_{Y,g,A,r}^J(\Gamma^+;\Gamma^-)/ \R$ converges in a certain suitable Gromov--Hofer topology to a {\em pseudoholomorphic building} consisting of curves $v_i \in \dcalM_{Y,g_i,A_i,r_i}^J(\Gamma^+_i;\Gamma^-_i)/\R$, $i = 1,...,a$, for some $a \geq 1$, such that 
\begin{enumerate}
\item $\Gamma^+_i = \Gamma^-_{i+1}$ for $i = 1,...,a-1$
\item the total arithmetic genus is $g$
\item $A_1 \circledcirc ... \circledcirc A_a = A$
\item $r_1 + ... + r_a = r$.
\end{enumerate}
Each $v_i$ corresponds to a ``level'' of the building, and we have listed them from bottom to top, i.e. $v_1$ is the ``ground level'', etc.
Note that symplectization curves are considered only up to the $\R$-action translating in the $\R$ direction in the target.
There is actually another ambiguity in the pseudoholomorphic building which is related to the asymptotic markers. 
Namely, suppose $\gamma$ is a Reeb orbit appearing as a positive end of $v_i$ and as a negative end of $v_{i+1}$, and as usual let $\kappa_\gamma$ denote its covering multiplicity. 
The underlying simple Reeb orbit has a basepoint $m_{\gamma}$ and there are $\kappa_\gamma$ possibilities for the location of the corresponding asymptotic marker in the domain of $v_i$, and similarly for $v_{i+1}$. 
This gives $\kappa_\gamma^2$ possibilities, but in fact we do not distinguish between orbits of the $\Z_{\kappa_{\gamma}}$-action which increments both choices according to the Reeb orientation.

\sss

Next, suppose that $(X,\omega)$ is a symplectic cobordism between $(Y^+,\alpha^+)$ and $(Y^-,\alpha^-)$. 
Let $J$ be any admissible almost complex structure on $\wh{X}$, and let $J^\pm$ denote the induced translation invariant almost complex structures on the symplectizations of $(Y^\pm,\alpha^\pm)$.
Consider a sequence of curves $u_1,u_2,u_3,... \in \calM_{Y,g,A,r}^J(\Gamma^+;\Gamma^-)$.
In this case the SFT compactness theorem takes a similar form, with the limiting pseudoholomorphic building consisting of curves
$v_1,...,v_{a+b+1}$ with $a,b \geq 0$ and 
\begin{itemize}
\item $v_i \in \dcalM_{Y^-,g_i,A_i,r_i}^{J^-}(\Gamma_i^+;\Gamma_i^-)/\R$ for $i=1,...,a$
\item $v_{i} \in \dcalM_{X,g_{i},A_{i},r_{i}}^J(\Gamma_{i}^+;\Gamma_{i}^-)$ for $i = a+1$
\item $v_i \in \dcalM_{Y^+,g_i,A_i,r_i}^{J^+}(\Gamma_i^+;\Gamma_i^-)/\R$ for $i=a+2,...,a+b+1$, 
\end{itemize}
again satisfying the conditions
\begin{itemize}
\item $\Gamma_i^+ = \Gamma^-_{i+1}$ for $i = 1,...,a+b$
\item the total arithmetic genus is $g$
\item $A_1 \circledcirc ... \circledcirc A_{a+b+1} = A$
\item $r_1 + ... + r_{a+b+1} = r$
\end{itemize}
and with the same ambiguities regarding symplectization curves and asymptotic markers.

\sss

Next, we recall the neck stretching procedure and the corresponding compactness result.
Consider a symplectic cobordism $(X,\omega) = (X_1,\omega_1)\circledcirc (X_2,\omega_2)$ obtained by gluing together two symplectic cobordisms along a common contact boundary $(Y,\alpha)$. 
We assume that $\omega|_{\Op(Y)} = d\lam$, with $\lam|_Y = \alpha$, and hence the flow of the Liouville vector field $Z_\lam$ defines a collar coordinate $s$ on $\Op(Y)$ with respect to which 
$(\Op(Y),\omega|_{\Op(Y)})$ looks like $((-\e,\e) \times Y, d(e^s\alpha))$.
We will describe a family $J_t$ of admissible almost complex structures on $(X,\omega)$ which limits to a singular one.
We find it convenient to describe $J_t$ on a varying family of symplectic manifolds $(X_t,\omega_t)$ which are all diffeomorphic to $(X,\omega)$.\footnote{We note that each $(X_t,\omega_t)$ is a fortiori symplectomorphic to $(X,\omega)$ in the case that $(X_1,\omega_1)$ is a Liouville domain.}
Namely, for $t >0$ set $X_t = X_1 \cup \left( [-t,0] \times Y\right) \cup X_2$, with the symplectic form $\omega_t$ given by
\begin{itemize}
\item
$e^{-t}\omega_1$ on $X_1$
\item
$e^t\alpha$ on $[-t,0] \times Y$
\item
$\omega_2$ on $X_2$.
\end{itemize}
Let $J$ be an admissible almost complex structure on $X$ which on $(-\e,\e) \times Y$ is the restriction of an admissible almost complex structure $J_Y$ on the symplectization of $Y$.
Let $J_i$ denote the restriction of $J$ to $X_i$ for $i =1,2$.
We endow $X_t$ with the almost complex structure $J_t$ given by
\begin{itemize}
\item
$J_1$ on $X_1$
\item
the restriction of $J_Y$ on $[-t,0] \times Y$
\item
$J_2$ on $X_2$.
\end{itemize}
In this situation the SFT compactness theorem states that any sequence of curves $u_1,u_2,u_3,... \in \calM_{X,g,r,A}^{J_{t_i}}(\Gamma^+;\Gamma^-)$ with 
$t_i \rightarrow +\infty$ converges to a pseudoholomorphic building consisting of curves $v_1,...,v_{a+b+c+2}$ with $a,b,c \geq 0$ and
\begin{itemize}
\item $v_i \in \dcalM_{\bdy^-X_1,g_i,A_i,r_i}^{J_1^-}(\Gamma^+_i,\Gamma_i^-)/\R$ for $i = 1,...,a$
\item $v_i \in \dcalM_{X_1,g_i,A_i,r_i}^{J_1}(\Gamma^+_i,\Gamma_i^-)$ for $i = a+1$
\item $v_i \in \dcalM_{Y,g_i,A_i,r_i}^{J_Y}(\Gamma^+_i,\Gamma_i^-)/\R$ for $i = a+2,...,a+b+1$
\item $v_i \in \dcalM_{X_2,g_i,A_i,r_i}^{J_2}(\Gamma^+_i,\Gamma_i^-)$ for $i = a+b+2$
\item $v_i \in \dcalM_{\bdy^+X_2,g_i,A_i,r_i}^{J_2^+}(\Gamma^+_i,\Gamma_i^-)/\R$ for $i = a+b+3,...,a+b+c+2$,
\end{itemize}
satisfying the same conditions as before regarding matching ends and so on.

\sss

Given a moduli space $\calM^J_{Y,g,A,r}(\Gamma^+;\Gamma^-)/\R$ of expected dimension zero, it makes sense at least in principle to assign a rational number $\# \calM^J_{Y,g,A,r}(\Gamma^+;\Gamma^-)/\R \in \Q$
by counting the number of elements, according to coherently chosen signs as in \cite[\S 1.8]{EGH2000} and 
weighted by $\frac{1}{|\op{Aut}|}$, where $\op{Aut}$ is the automorphism group of a given curve.
Typically, in order to get relations between these curve counts one needs not just compactness but also gluing results describing the local structure of compactified moduli spaces near their boundary strata. 
Ideally we would like to say that for example each moduli space $\calM^J_{Y,g,A,r}(\Gamma^+;\Gamma^-)/\R$ admits a compactification $\ovl{\calM}^J_{Y,g,A,r}(\Gamma^+;\Gamma^-)/\R$ as a smooth stratifed space, with combinatorially indexed boundary strata which correspond precisely to the limiting pseudoholomorphic buildings described in the SFT compactness theorem. 
Each additional node should increase the codimension by two and each additional level should increase the codimension by one.
In particular, we would like to argue that for generic $J$ all moduli spaces $\calM^J_{Y,g,A,r}(\Gamma^+;\Gamma^-)/\R$ of negative expected dimension are empty, while all moduli spaces of expected dimension $0$ or $1$ are regular, with the $1$-dimensional moduli spaces admitting compactifications $\ovl{\calM}^J_{Y,g,A,r}(\Gamma^+;\Gamma^-)/\R$ as one-dimensional manifolds, with boundary points consisting of two-level pseudoholomorphic buildings.
Although this can generally only be achieved in a virtual sense, we can still read off relations between curve counts by looking at the
boundary strata of expected codimension one in 
 compactified moduli spaces of expected dimension one.

In the sequel we will often omit $J$ from the notation when it is implicit from the context. From now on we will only consider genus zero curves and thus we omit $g$ from the notation, i.e. we put $\calM_{Y,A,r}(\Gamma^+;\Gamma^-)$ instead of $\calM_{Y,g=0,A,r}(\Gamma^+;\Gamma^-)$ and so on.

\subsection{Rational symplectic field theory formalism}\label{subsec:rsft_formalism}

In this subsection we describe the algebraic formalism underpinning rational symplectic field theory.
Firstly, to a strict contact manifold $(Y,\alpha)$ we associate the contact homology algebra $\cha(Y)$. 
This is a CDGA\footnote{Recall that by default we work at chain level, with passage to homology explicitly notated.}
 over $\Lamo$, freely generated as an algebra by a variable for each good\footnote{A Reeb orbit is ``bad'' if it is an even cover of another Reeb orbit whose Conley--Zehnder index has a different parity, otherwise it is ``good''.
Bad orbits are excluded for reasons having to do with signs and coherent orientations of moduli spaces (see \cite[\S 1.8]{EGH2000}).}
Reeb orbit of $(Y,\alpha)$, with
differential counting curves in the symplectization of $Y$ with one positive end and any nonnegative number of negative ends.
An exact symplectic cobordism $(X,\omega)$ between $(Y^+,\alpha^+)$ and $(Y^-,\alpha^-)$ induces a CDGA
homomorphism $\Phi_X: \cha(Y^+) \rightarrow \cha(Y^-)$. In the special case that $(X,\omega)$ is a filling of $(Y,\alpha)$, we get a CDGA augmentation $\aug: \cha(Y) \rightarrow \Lamo$. Given such an augmentation, we can ``linearize'' $\cha(Y)$ to obtain a much smaller chain complex $\chlin(X)$ called the linearized contact homology.
This can be interpreted geometrically as the chain complex over $\Lamo$ generated by the good Reeb orbits of $(Y,\alpha)$, with differential counting ``anchored'' cylinders in the symplectization of $(Y,\alpha)$, i.e. cylinders with some number of extra negative punctures bounding pseudoholomorphic planes in the filling $(X,\omega)$.

Next, we incorporate rational curves with more than one positive end.
In the original formulation of \cite{EGH2000}, one associates to $(Y,\alpha)$ a differential Poisson algebra with variables $q_{\gamma},p_{\gamma}$ for each good Reeb orbit, with
differential defined in terms of rational curves in the symplectization of $(Y,\alpha)$, and with symplectic cobordisms inducing correspondences. 
Instead, roughly following the ``$q$-variable only'' approach from \cite{hutchings_2013}, we will associate to $(Y,\alpha)$ an $\Li$ structure over $\Lamo$ whose underlying chain complex is $\cha(Y)$. 
This approach seems more natural for extracting quantitative invariants. We can also ``linearize'' this $\Li$ structure to arrive at a much smaller $\Li$ structure with underlying chain complex $\chlin(X)$. 

\begin{addendum}\label{add:janko}
After the first version of this paper appeared, the recent work \cite{latschev2022remarks} discusses how our quantitative RSFT formalism can be recast in terms of the original formulation of \cite{EGH2000}; we refer the reader to loc. cit. for more details.
\end{addendum}

\begin{remark}[on the relationship between SFT and Floer theory]\label{rmk:relationship_between_SFT_and_Floer}
Rational symplectic field theory is related to Floer theory as follows. Over $\mathbb{Q}$, linearized contact homology $\chlin(X)$ was shown in \cite{Bourgeois-Oancea_equivariant} to be chain homotopy equivalent to $\sc_{S^1,+}(X)$, the positive version of $S^1$-equivariant symplectic cochains.
The same authors showed that there is also a non-equivariant version of linearized contact homology, $\op{CH}_{\op{lin,non-eq}}(X)$, with two variables 
variables $\hat{\gamma},\check{\gamma}$ for each good Reeb orbit $\gamma$, 
 which is equivalent to positive symplectic cochains $\sc_+(X)$.
There is a natural chain map from $\op{CH}_{\op{lin,non-eq}}(X)$ to the Morse cochain complex $C(X)$, and the mapping cone 
gives ``filled non-equivariant linearized contact homology'', $\op{CH}_{\op{lin,non-eq,filled}}(X)$, which is equivalent to symplectic cochains $\sc(X)$.

Next, \cite{Ekholm-Oancea_DGAS} constructs an $\Li$ coalgebra structure on $\sc_{S^1}(X,\omega)$, defined by counting Floer curves with one positive end and several negative ends, with negative ends carrying certain weights which vary over a simplex. Via the cobar construction, this can also be packaged as a CDGA whose underlying algebra is the symmetric tensor algebra on $\sc_{S^1}(X)$. 
This latter CDGA is shown in \cite{Ekholm-Oancea_DGAS} to be equivalent to the contact homology algebra $\cha(Y)$,
and there is also a non-equivariant version which is equivalent to a non-equivariant version of the contact homology algebra.
Moreover, there are also $\Li$ structures extending the differentials on $\sc(X)$ and $\sc_{S^1}(X)$, which can be defined by using the mapping telescope framework of \cite{abouzaid2010open}.

More speculatively, it seems reasonable to expect that the $\Li$ structure on $\cha(X)$ (described in \S\ref{subsubsec:Li_str_on_CHA}), and perhaps also a non-equivariant analogue thereof, should have equivalent models in Floer theory. Roughly, one ought to count rational Floer curves with several positive and negative punctures, combining the mapping telescope approach of \cite{abouzaid2010open} with the simplex approach of \cite{Ekholm-Oancea_DGAS}. This would give a Floer-theoretic analogue of rational symplectic field theory.
However, to the author's knowledge such a construction has not been attempted in the literature.
\end{remark}

\subsubsection{The contact homology algebra and its linearization}\label{subsubsec:cha_and_its_linearization}

We first recall the definition of the contact homology algebra.
 Given a strict contact manifold $(Y,\alpha)$, 
let $\X$ denote the $T$-adic completion of the $\Lamo$-module generated by a formal variable $x_{\gamma}$ for each good Reeb orbit $\gamma$ of $(Y,\alpha)$.
We endow $\X$ with a grading such that $|x_\gamma| = n - 3 - \cz(\gamma)$. 
Along the lines of Remark \ref{rmk:on_gradings_in_symplectic_fillings}, this is only a $\Z/2$ grading, 
although in many cases it can be upgraded to say a $\Z$ grading (see Remark \ref{rmk:on_gradings} below for the relevant discussion).
Let $\A$ denote the $T$-adic completion of $S\X$.
That is, $\A$ is the free supercommutative unital $\Lamo$-algebra generated by the good Reeb orbits of $(Y,\alpha)$.\footnote{Recall from \S\ref{subsec:filtered_Linf_algebras} that our conventions are such that both $\X$ and $S\X$ are always completed with respect to the $T$-adic topology.}
Elements of $\A$ are $\Lamo$-linear combinations of ``words'' $x_{\gamma_1} \odot ... \odot x_{\gamma_k}$ for $k \geq 0$, which we will write more briefly as $x_{\gamma_1}...x_{\gamma_k}$.
 In particular, note that there is the ``empty word'' $e = 1 \in \Lamo = \odot^0\X$, which serves as the multiplicative unit.

The contact homology algebra $\cha(Y)$ is by definition the CDGA over $\Lamo$ with underlying graded algebra $\A$,
with degree one differential 
$\bdy$ defined as follows.
Pick an arbitrary total ordering $\leqq$ of the good Reeb orbits of $(Y,\alpha)$.
For $l \geq 0$, let $\orbset_l(Y,\alpha)$ 
denote the set of $l$-tuples $\Gamma = (\gamma_1,...,\gamma_l)$ of good Reeb orbits such that 
$\gamma_1 \leqq ... \leqq \gamma_l$.
Equivalently, a tuple $(\gamma_1,...,\gamma_l) \in \orbset_l(Y,\alpha)$ can be written as $(\underbrace{\eta_1,...,\eta_1}_{i_1},...,\underbrace{\eta_m,...,\eta_m}_{i_m})$
with $\eta_1 \prec ... \prec \eta_m$ for some $i_1,...,i_m \geq 1$ such that $i_1+...+i_m = l$.\footnote{We could alternatively forego the total ordering at the cost of later dividing out by $m!$ in \eqref{eq:CHA_diff} below -- c.f. \cite[\S2.1]{EGH2000}.} 
Set $\orbset(Y,\alpha) = \bigcup\limits_{l \geq 0}\orbset_l(Y,\alpha)$.
Note that there is an empty tuple $\nil \in \orbset_0(Y,\alpha) \subset \orbset(Y,\alpha)$.
For $\Gamma \in \orbset(Y,\alpha)$, put $\mu_{\Gamma} := i_1!...i_m!$,
$\kappa_{\Gamma} := \kappa_{\eta_1}^{i_1}...\kappa_{\eta_m}^{i_m}$, $\cz(\Gamma) := \sum_{j=1}^l \cz(\gamma_j)$,
and $x_{\Gamma} = x_{\gamma_1}...x_{\gamma_l}$.

For $l \geq 0$, we define a map $\bdy_{l}: \X \rightarrow \odot^l \X$ by
\begin{align}\label{eq:CHA_diff}
\bdy_{l}(x_{\gamma}) = \sum_{\substack{\Gamma^- \in \orbset_l(Y,\alpha)\\ A \in H_2(Y,(\gamma)\cup\Gamma^-;\Z)}} \# \ovl{\calM}_{Y,A}((\gamma);\Gamma^-)/\R  \frac{T^{[d\alpha]\cdot A}}{\mu_{\Gamma^-}\kappa_{\Gamma^-}}x_{\Gamma^-}.
\end{align}
Here we define $\# \ovl{\calM}_{Y,A}((\gamma);\Gamma^-)/\R \in \Q$ to be zero unless it is expected to be a count of points,
i.e. unless $$\ind_{Y,A}((\gamma);\Gamma^-) = (n-3)(1-l)+\cz(\gamma) - \cz(\Gamma^-) + 2c_1(Y,\alpha)\cdot A = 1.$$
Note that by nonnegativity of energy we have $[d\alpha] \cdot A \geq 0$, i.e. $T^{[d\alpha] \cdot A} \in \Lamo$.
The differential for the contact homology algebra is then defined on generators by 
\begin{align*}
\bdy(x_\gamma) = \sum_{l =0}^\infty \bdy_{l}(x_{\gamma}),
\end{align*}
extended to more general words by the Leibniz rule
$$\bdy(x_{\gamma_1}...x_{\gamma_l}) = \sum_{i=1}^l (-1)^{|x_{\gamma_1}|+...|x_{\gamma_{i-1}}|}x_{\gamma_1}...x_{\gamma_{i-1}}\bdy(x_{\gamma_i})x_{\gamma_{i+1}}...x_{\gamma_l}.$$
Informally, the sum is simply over all rigid (up to target translations) curves in the symplectization of $Y$ with one positive puncture $\gamma$ and several (possibly none) unordered negative punctures.
Each term $1/i_j!$ in $\mu_{\Gamma^-}$ accounts for the superfluous ordering of the punctures asymptotic to $\gamma_j$,
and each factor $1/\kappa_{\gamma_i}$ in $\kappa_{\Gamma^-}$ is related to the fact that there are there is a $\kappa_{\gamma}$-fold redundancy when gluing at a pair of positive and negative punctures both asymptotic to $\gamma$.
Note that by exactness the sums involved are actually finite and hence would be well-defined even without using the Novikov ring.

The relation $\bdy^2 = 0$ comes from the structure of the boundary strata of the compactified moduli space $\ovl{\calM}_{Y,A}((\gamma);\Gamma^-)/\R$ for $\ind_{Y,A}((\gamma);\Gamma^-) = 2$.
Namely, each boundary stratum corresponds to a two-level pseudoholomorphic building in the symplectization of $(Y,\alpha)$.
Each level must have index $1$ and hence consists of a single connected curve of index $1$ together with some number of trivial cylinders.
The structure coefficient $\langle \bdy^2x_{\gamma}, x_{\gamma_1}...x_{\gamma_l}\rangle$
corresponds to the count of points in $\bdy \ovl{\calM}_Y((\gamma);\Gamma^-)/\R$, and hence is zero.

\begin{remark}[on gradings]\label{rmk:on_gradings}

In many cases it is possible and desirable to upgrade the $\Z/2$ grading on $\cha(Y)$ to say a $\Z$ grading.
Here we focus on the case that $(Y,\alpha)$ has a filling $(X,\omega)$ with $2c_1(X,\omega) = 0$ and $H_1(X;\Z) = 0$
(e.g. star-shaped domains).
In this situation, every Reeb orbit $\gamma$ of $(Y,\alpha)$ has a canonical integral Conley--Zehnder index $\cz(\gamma) \in \Z$.
If we were to follow the conventions of \cite{EGH2000}, we would replace $\X$ by $\calQ$, the $\Lamo$-module generated by a formal variable $q_{\gamma}$ for each good Reeb orbit, with the somewhat different grading $|q_{\gamma}| = \cz(\gamma) + n - 3$.
Defining $\cha(Y)$ in otherwise the same way, the resulting CDGA is $\Z$-graded, with differential of degree $-1$ (homological convention), product of degree $0$, and unit in degree $0$.

However, if we use $\calQ$, the $\Li$ structures on $\cha(Y)$ and $\chlin(X)$ do not easily conform to the grading framework of $\S\ref{sec:filtered_L-inf_algebras}$.
Instead, we will implement a different grading convention which is more naturally geared for discussing $\Li$ structures and also matches up with standard grading conventions for symplectic cohomology.
Our conventions for symplectic cochains $\sc^*(X)$ are that orbits $\gamma$ are graded by $n - \cz(\gamma)$, so that
the differential has degree $+1$ (cohomological convention), the pair of pants product has degree $0$ with unit in degree $0$, and for $M$ a closed spin manifold the Viterbo isomorphism reads
$SH^*(T^*M) \cong H_{n-*}(\calL M)$.
With this grading, the $k$th $\Li$ operation on $\sc^*$ has degree $3-2k$, i.e. it corresponds to an $\Li$ structure on $s\sc^*$ within the framework of \S\ref{sec:filtered_L-inf_algebras}.
With the same conventions, the $k$th $\Li$ operation on $\sc^*_{S^1}$ has degree $4-3k$, corresponding to an $\Li$ structure on $s^2\sc_{S^1}^*$.

We can modify our definition of $\cha(Y)$ as follows.
We define $\X$ as before, except that we now work over $\Lamo[\fv,\fv^{-1}]$, where $\fv$ is a formal variable of degree $6 - 2n$ (whose only role is to correct gradings),
and the differential is given by 
$$\bdy(x_\gamma) = \sum_{l=0}^\infty \bdy_l(x_\gamma) \fv^{l-1}.$$
This makes $\cha(Y)$ into a $\Z$-graded chain complex with differential of degree $+1$. 
Similarly, when discussing the $\Li$ structure on $\cha(Y)$, we can modify $\ell^k_l$ by including an additional factor $\fv^{l-1}$.
The result is a $\Z$-graded $\Li$ structure on $s^{-1}\A$.
As explained in \S\ref{subsubsec:L-inf_structure_on_chlin}, this also ``linearizes'' to an $\Li$ structure on $s^{-1}\chlin(X)$,
i.e. with underlying $\Lamo$-module $s^{-1}\X$.
Since the generator $s^{-1}x_{\gamma}$ of $s^{-1}\X$ has degree $|s^{-1}x_{\gamma}| = n -2 - \cz(\gamma)$,
this matches the grading on $s^2\sc^*_{S^1,+}$. 

Having a $\Z$ grading can be very useful for computations. At any rate, from now on we will mostly suppress $\fv$ and other $\Z$ grading considerations from the notation for simplicity. We also note that the formal variable $\fv$ is not be needed when discussing the $\Li$ structure on $\chlin(X)$, which involves curves with only one negative end.
\end{remark}

Next, suppose that $(X,\omega)$ is an {\em exact} symplectic cobordism between $(Y^+,\alpha^+)$ and $(Y^-,\alpha^-)$.
We get an associated unital CDGA homomorphism $\Phi: \cha(Y^+) \rightarrow \cha(Y^-)$ defined as follows.
Define a map $\Phi_{l}: \X_{Y^+} \rightarrow \odot^l \X_{Y^-}$ by
\begin{align*}
\Phi_{l}(x_{\gamma}) = \sum_{\substack{\Gamma^- \in \orbset_l(Y^-,\alpha^-)\\ A \in H_2(X,(\gamma)\cup\Gamma^-;\Z)}} \# \ovl{\calM}_{X,A}((\gamma);\Gamma^-)  \frac{T^{[\omega]\cdot A}}{\mu_{\Gamma^-}\kappa_{\Gamma^-}}x_{\Gamma^-}.
\end{align*}
In this case, a summand can only be nonzero if $\ind_{X,A}((\gamma);\Gamma^-) = 0$, and again by nonnegativity of energy we have $[\omega] \cdot A \geq 0$.
We then define
\begin{align*}
\Phi(x_{\gamma}) = \sum_{l=0}^\infty \Phi_{l}(x_\gamma),
\end{align*}
extended to more general words multiplicatively by 
$$\Phi(x_{\gamma_1}...x_{\gamma_k}) = \Phi(x_{\gamma_1})...\Phi(x_{\gamma_k}).$$

The fact that $\Phi$ is a chain map follows by looking at the boundary strata of $\ovl{\calM}_{X,A}((\gamma);\Gamma^-)$ for $\ind_{X,A}((\gamma);\Gamma^-) = 1$. Namely, each point of $\bdy\ovl{\calM}_{X,A}((\gamma);\Gamma^-)$ is a two-level pseudoholomorphic building with an index $1$ level in the symplectization of $Y$ and an index $0$ level in $X$, in either order, and these translate to $\bdy_{Y^-} \circ \Phi$ and $\Phi \circ \bdy_{Y^+}$ respectively.
Note here that the Novikov coefficients also decompose consistently along such boundary degenerations (see the discussion of energy in \S\ref{subsec:geometric_setup}).
As a special case, if $(X,\omega)$ is a symplectic filling of $(Y,\alpha)$, we get an augmentation of $\cha(Y)$, i.e. a CDGA homomorphism $\aug: \cha(Y) \rightarrow \Lamo$ sending $e$ to $1$.
Here $\Lamo$ is the $\Li$ algebra equipped with the trivial differential.
Strictly speaking, we should also shift the degree of $\Lamo$ in order for $\aug$ to have degree $0$, but we will mostly suppress this from the notation to avoid clutter (and similarly for the more general augmentations appearing in \S\ref{sec:geometric_constraints}).

\sss

The algebra $\A = S\X$ has an increasing ``word length'' filtration given by 
$\calF_{\leq m}\A = \oplus_{k = 0}^m \odot^k \X$.
The differential $\bdy$ typically does not respect this filtration
due to the presence of the term $\bdy_{0}$.
However, given an augmentation $\aug: \cha(Y) \rightarrow \Lamo$ we can perform an algebraic operation called ``linearization'' which replaces $\cha(Y)$ with an isomorphic CDGA whose differential does preserve the word length filtration. Namely, let $F^\aug: \A \rightarrow \A$ be the CDGA homomorphism defined on generators by $F^{\aug}(x_{\gamma}) = x_\gamma + \aug(x_\gamma)$
and extended multiplicatively, i.e.
\begin{align*}
F^\aug(x_{1}...x_{k}) &= (x_{1} + \aug(x_{1}))...(x_{k} + \aug(x_{k}))\\
&= \sum_{i=0}^k\sum_{\sigma \in \Sh(i,k-i)} \sign(\sigma)\aug(x_{{\sigma(1)}})...\aug(x_{{\sigma(i)}})x_{{\sigma(i+1)}}...x_{{\sigma(k)}}.
\end{align*}
Note that $F^\aug$ is invertible, with inverse $F^{\aug^{-1} }= F^{-\aug}$.
We define a new differential on $\A$ by $\bdy^\eps := F^\eps\circ \bdy \circ F^{\eps^{-1}}$.
Then $F^\aug$ defines a CDGA isomorphism $(\A,\bdy) \xrightarrow{\cong} (\A,\bdy_\eps)$,
and it trivializes the augmentation in the sense that $\aug\circ F^{\eps^{-1}} = 0$.
Moreover, letting $\pi_i: \calA \rightarrow \odot^i\X$ denote the projection map, we have $\pi_0 \circ \bdy_\aug = 0$.
In particular, this means that $\bdy^{\lin} := \pi_1\circ \bdy^\aug: \X \rightarrow \X$ squares to zero.
Given a symplectic filling $(X,\omega)$, the linearized contact homology $\chlin(X)$ is by definition
the chain complex $(\X,\bdy_\lin)$, with the linearization taken with respect to the filling-induced augmentation $\aug: \cha(Y) \rightarrow \Lamo$.

In fact, $\chlin(X)$ has the following more geometric interpretation.
The structure coefficient $\langle \bdy_\lin(x_{\gamma}),x_{\eta}\rangle$
counts rational curves in the symplectization of $Y$ with one positive end $\gamma$, one negative end $\eta$,
and $k \geq 0$ extra negative ends at Reeb orbits $\gamma^-_1,...,\gamma^-_k$,
each weighted by $\aug(\gamma^-_i)$, i.e. by the count of holomorphic planes in $X$ with positive end $\gamma^-_i$.
We refer to such a configuration as a cylinder in the symplectization of $Y$ with ``anchors'' in the filling $X$.
More generally, one can speak of anchored curves of a given topological type, wherein one allows the component in $\R \times Y$ to be possibly disconnected and to have any number of extra negative ends, provided that each extra negative end is matched with the positive end of a component in $X$ and the total topological type is as specified.

\begin{remark}
In some cases, one can ignore the anchors altogether and simply interpret $\chlin(X)$ as a count of cylinders in the symplectization of $Y$.
For example, if $2c_1(X^{2n},\omega) = 0$ and $(Y,\alpha)$ has no Reeb orbits with $\cz(\gamma) + n - 3 = 0$,
then anchors can be ignored since there are no index zero holomorphic planes in $X$.
In this case linearized contact homology is sometimes also called ``cylindrical contact homology''.
A similar remark applies to the $\Li$ structure on $\chlin(X)$ described below, provided that there are no index zero $k$-punctured holomorphic spheres in $X$ for $k \in \Z_{\geq 1}$.
\end{remark}

\subsubsection{The $\Li$ structure on the contact homology algebra}\label{subsubsec:Li_str_on_CHA}
The differential for the contact homology algebra
is actually the first term $\ell^1=\bdy$
in a sequence of operations $\ell^k: \odot^k\calA \rightarrow \calA$, $k \geq 1$, which fit together to form an $\Li$ algebra structure on $s^{-1}\A$.
Here the map $\ell^k: \odot^k\calA \rightarrow \calA$ is defined by counting rational curves with $k$ positive ends in the symplectization of $(Y,\alpha)$.
Firstly, we define a map $\odot^k\X \rightarrow \odot^l \X$ by 
\begin{align*}
\ell_l^k(x_{\gamma_1},...,x_{\gamma_k}) := \sum_{\substack{\Gamma^- \in \orbset_l(Y,\alpha)\\ A \in H_2(Y,\Gamma^+\cup\Gamma^-;\Z)}}
\# \ovl{\calM}_{Y,A}(\Gamma^+;\Gamma^-)/\R  \frac{T^{[d\alpha]\cdot A}}{\mu_{\Gamma^-} \kappa_{\Gamma^-}}x_{\Gamma^-},
\end{align*}
where we have set $\Gamma^+ = (\gamma_1,...,\gamma_k)$.
We then define $$\ell^k(x_{\gamma_1},...,x_{\gamma_k}) := \sum_{l=0}^\infty \ell_{l}^k(x_{\gamma_1},...,x_{\gamma_k}).$$
We extend the input to more general words by requiring $\ell^k$ to satisfy the Leibniz rule in each input, i.e. 
given $\Gamma_1^+,...,\Gamma_k^+ \in \orbset(Y,\alpha)$ with $\Gamma_i = \gamma_1,...,\gamma_m$ for some $i$, we have
\begin{align*}
& \ell^k(x_{\Gamma_1^+},...,x_{\Gamma_{i-1}^+},x_{\gamma_1}...x_{\gamma_m},x_{\Gamma_{i+1}^+},...,x_{\Gamma_k^+})
\\&:= \sum_{j=1}^m  (-1)^{|x_{\gamma_1}|...|x_{\gamma_{i-1}}|}x_{\gamma_1}...x_{\gamma_{j-1}}\ell^k(x_{\Gamma_1^+},...,x_{\Gamma_{i-1}^+},x_{\gamma_j},x_{\Gamma_{i+1}^+},...,x_{\Gamma_k^+})x_{\gamma_{j+1}}...x_{\gamma_m}.
\end{align*}

To see that the maps $\ell^1,\ell^2,\ell^3,...$ satisfy the $\Li$ relations, consider the compactified moduli space
$\ovl{\calM}_{Y,A}(\Gamma^+;\Gamma^-)/\R$ for $\ind_{Y,A}(\Gamma^+;\Gamma^-) = 2$.
Each boundary stratum corresponds to a two-level pseudoholomorphic building. Each level has index $1$ and hence consists of an index $1$ component together with some number of trivial cylinders, such that the total glued curve is connected. 
Then the count of points in $\bdy \ovl{\calM}_Y(\Gamma^+;\Gamma^-)/\R$ precisely corresponds to the structure coefficient $\langle (\wh{\ell})^2(x_{\gamma_1}\odot ... \odot x_{\gamma_k}),x_{\Gamma^-}\rangle$, and hence is zero. 
Since the operations $\ell^1,\ell^2,\ell^3,...$ are derivations in each input, this suffices to establishes the $\Li$ relations for all inputs.

\begin{remark}
In fact, the $\Li$ algebra $(\ovl{S}\A,\wh{\ell})$ has a much richer algebraic structure, since $\calA$ additionally has a commutative product which is compatible with the $\Li$ operations, making it into a special case of a $\mathcal{G}_\infty$ (homotopy Gerstenhaber) algebra. However, in this paper we do not attempt to make this precise or utilize this finer structure in any systematic way.
\end{remark}

Next, suppose that $(X,\omega)$ is an exact symplectic cobordism between $(Y^+,\alpha^+)$ and $(Y^-,\alpha^-)$.
We define an $\Li$ homomorphism from $\cha(Y^+)$ to $\cha(Y^-)$, consisting of maps $\Phi^k: \odot^k \calA_{Y^+} \rightarrow \calA_{Y^-}$, $k \geq 1$, as follows.
In general, given any two partitions $$I_1 \sqcup ... \sqcup I_a = J_1 \sqcup ... \sqcup J_b = \{1,...,N\},$$
let $G(I_1,...,I_a;J_1,...,J_b)$ denote the graph with $a+b+N$ vertices, labeled $A_1,...,A_a,B_1,...,B_b,v_1,...,v_N$, with
\begin{itemize}
\item an edge joining $A_i$ and $v_j$ if $j \in I_i$
\item an edge joining $B_i$ and $v_j$ if $j \in J_i$.
\end{itemize}
Now consider words $w_1,...,w_k \in \calA_{Y^+}$ of the form 
\begin{align*}
(w_1,...,w_k) = (\underbrace{x_{1}...x_{{i_1}}}_{i_1},\underbrace{x_{i_1+1},...,x_{i_2}}_{i_2},...,\underbrace{x_{{i_1+...+i_{k-1}+1}},...,x_{{i_1+...i_k}}}_{i_k}),
\end{align*}
and set $N := i_1 + \cdots + i_k$.
Put $I_1 = \{1,...,i_1\},...,I_k = \{i_1+...+i_{k-1}+1,...,i_1+...+i_k\}$, so that
the subsets $I_1,...,I_k \subset \{1,...,N\}$ define a partition $I_1 \sqcup ... \sqcup I_k = \{1,...,N\}$.
We define
 \begin{align}\label{eq:Phi_k}
&\Phi^k(w_1,...,w_k) \\:= &\sum_{\substack{J_1\sqcup ... \sqcup J_m = \{1,...,N\} \\ \Gamma_1^-,...,\Gamma_m^- \in \orbset(Y^-,\alpha^-)\nonumber\\ A_1,...,A_m\;:\; A_i \in H_2(X,\Gamma_i^+\cup\Gamma_i^-;\Z) }} \#\ovl{\calM}_{X,A_1}(\Gamma_1^+;\Gamma_1^-)...\#\ovl{\calM}_{X,A_m}(\Gamma_m^+;\Gamma_m^-) \frac{T^{[\omega]\cdot(A_1+...+A_m)}}{\mu_{\Gamma_1^-}...\mu_{\Gamma_m^-}\kappa_{\Gamma_1^-}...\kappa_{\Gamma_m^-}}x_{\Gamma_1^-}...x_{\Gamma_m^-},
\end{align}
where the word $\Gamma^+_i \in \orbset_{|J_i|}(Y^+,\alpha^+)$ correspond to $J_i$ for $i=1,...,m$,
and the sum is over all partitions $J_1 \sqcup ... \sqcup J_m = \{1,..,N\}$ with the property that the graph $G(I_1,...,I_k;J_1,...,J_m)$ is simply connected.
Informally, we can view each word $w_i$ as a cluster of Reeb orbits in $Y^+$ with a punctured sphere abstractly bounding their union from above, and $\Phi^k(w_1,...,w_k)$
is the sum of all ways of attaching a collection of rational curves in $X$ such that the total curve is connected and genus zero. See Figure \ref{cha_Phi} for a cartoon.

\begin{figure}[!tbp]
  \centering
  \begin{minipage}[b]{.9\textwidth}
\includegraphics[width=\textwidth]{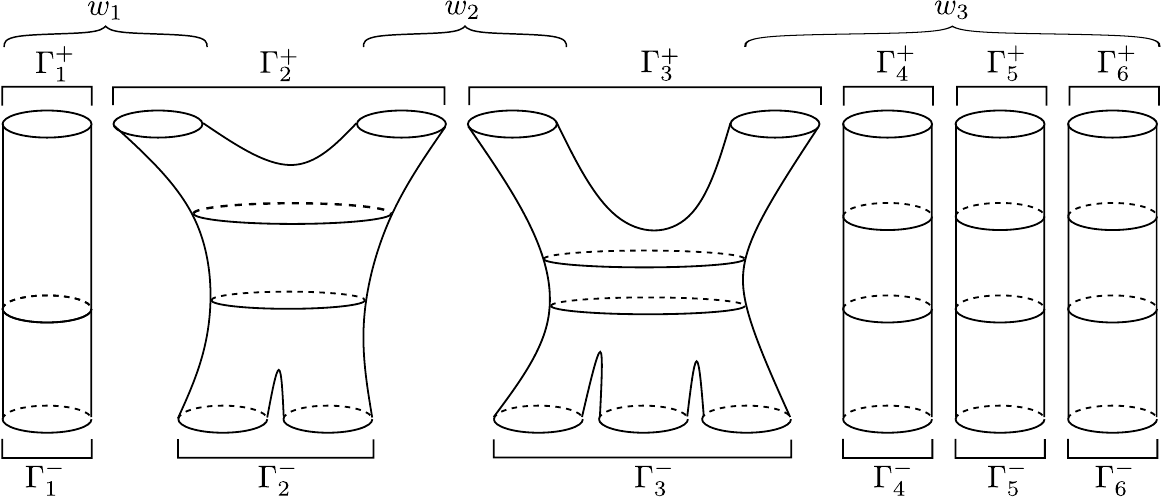}
  \end{minipage}
  \hfill
  \begin{minipage}[b]{.5\textwidth}
   \includegraphics[width=\textwidth]{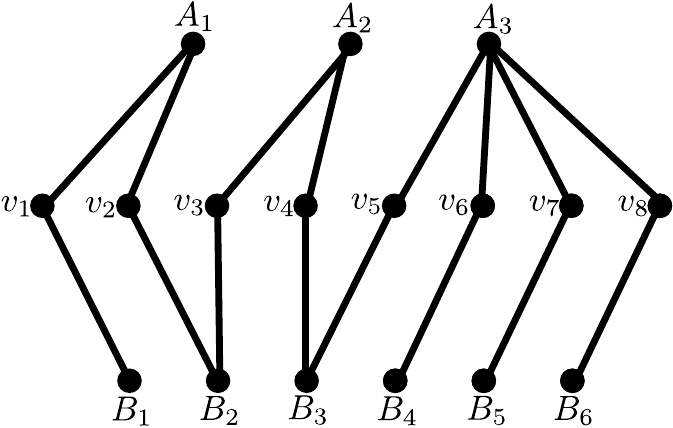}
  \end{minipage}
  \caption{Top: a collection of curves in $\wh{X}$ contributing to the coefficient $\langle \Phi^3(w_1,w_2,w_3),x_{\Gamma_1^-}...x_{\Gamma_6^-}\rangle$. Bottom: the corresponding graph $G(I_1,I_2,I_3;J_1,\dots,J_6)$, where $I_1 = \{1,2\}$, $I_2 = \{3,4\}$, $I_3 = \{5,6,7,8\}$, $J_1 = \{1\}$, $J_2 = \{2,3\}$, $J_3 = \{4,5\}$, $J_4 = \{6\}$, $J_5 = \{7\}$, $J_6 = \{8\}$ (note that in general the elements of $J_1,\dots,J_m$ need not appear in consecutive order).}
  \label{cha_Phi}
\end{figure}

To see that the maps $\Phi^1,\Phi^2,\Phi^3,...$ define an $\Li$ homomorphism, 
we look at the boundary of the compactified moduli space $\ovl{\calM}_{X}(\Gamma^+;\Gamma^-)$ for $\ind_{X,\omega}(\Gamma^+;\Gamma^-) = 1$, with $\Gamma^+ = (\gamma_1,...,\gamma_k)$.
It suffices to check 
the relation $\wh{\ell}_{Y^-}\circ \wh{\Phi} - \wh{\Phi}\circ \wh{\ell}_{Y^+} = 0$ on elements $x_{\gamma_1}\odot ... \odot x_{\gamma_k} \in \ovl{S}\calA_{Y^+}$,
with $\wh{\Phi}: \ovl{S}\calA_{Y^+} \rightarrow \ovl{S}\calA_{Y^-}$ the induced map on bar complexes.
This follows because the coefficient of $x_{\Gamma^-}$ in this relation precisely corresponds to the count of points in $\bdy\ovl{\calM}_X(\Gamma^+;\Gamma^-)$.

\sss

Given two different versions $\Phi,\Phi'$ of the cobordism map $\cha(Y^+) \rightarrow \cha(Y^-)$ defined by two different almost complex structures $J_0$ and $J_1$ (and other perturbation data) for the cobordism, there is a filtered $\Li$ homotopy relating $\Phi$ and $\Phi'$.
Indeed, we can pick a $1$-parameter family of almost complex structures $J_t$ joining $J_0$ and $J_1$, and consider a $t$-parametrized analogue of the map $\Phi$ defined by counting solutions to an appropriate $t$-parametrized moduli space of curves in the symplectic completion of $X$.
According to \cite[Thm 2.4.5]{EGH2000}, the resulting algebraic output is a homotopy between the rational SFT potentials associated to $J_0$ and $J_1$, formulated in terms of a certain Hamilton-Jacobi differential equation, and in our setup this gives rise to a filtered $\Li$ homotopy between filtered $\Li$ homomorphisms as in \S\ref{sec:filtered_L-inf_algebras}.

Suppose now that we have a symplectic cobordism $X = X_1 \circledcirc X_2$ which is the concatenation of exact symplectic cobordisms $X_1,X_2$ with common contact boundary $Y = \bdy^+X_1= \bdy^-X_2$. In this situation we have induced filtered $\Li$ homomorphisms
\begin{itemize}
  \item $\Phi_{X_2}: \cha(\bdy^+X_2) \ra \cha(Y)$
  \item $\Phi_{X_1}: \cha(Y) \ra \cha(\bdy^-X_1)$
  \item $\Phi_X: \cha(\bdy^+X_2) \ra \cha(\bdy^-X_1)$.
\end{itemize}
By a neck stretching analogue of the above homotopy (using the compactification described in \S\ref{subsec:SFT_comp_and_neck}), we get a filtered $\Li$ homotopy 
\begin{align}\label{eq:neck_stretch_comp}
\Phi_X \sim \Phi_{X_1} \circ \Phi_{X_2}
\end{align}
 as filtered $\Li$ homomorphisms $\cha(\bdy^+X_2) \ra \cha(\bdy^-X_1)$.

\sss

Observe that a priori $\cha(Y)$ depends on various choices, namely a contact one-form $\alpha$, an admissible almost complex structure $J$ on $\R \times Y$, and any additional perturbation data needed to achieve transversality.
We will sometimes write e.g. $\cha(Y,\alpha)$ if we wish to emphasize the role of the contact form.
However, it is possible to show that, up to filtered $\Li$ homotopy equivalence, $\cha(Y)$ is actually independent of the choice of almost complex structure $J$ on $\R \times Y$ and any other perturbation data, and also under replacing the contact one form $\alpha$ with another of the type $\alpha' = \alpha + df$ for a smooth function $f: Y \ra \R$.
Here it is instructive to observe that, for any small $\eps > 0$, we can find an exact compact symplectic cobordism $[0,\eps] \times Y$ with negative contact boundary $(Y,\alpha)$ and positive contact boundary $(Y,e^\eps \alpha')$ (if $\alpha' = \alpha$ we take the Liouville one-form to be simply $e^s \alpha$).
This induces an $\Li$ homomorphism $\cha(Y,\alpha') \ra \cha(Y,\alpha)$, where we may choose different almost complex structures and perturbation data for the domain and target, and which preserves the filtrations up small distortions (coming from identifying $\cha(Y,e^\eps \alpha')$ with $\cha(Y,\alpha')$).
We can similarly define a map in the other direction, $\cha(Y,\alpha) \ra \cha(Y,\alpha')$, and by the discussion in the previous paragraph the compositions in either order are, up to small distortions, filtered $\Li$ homotopic to the identity maps. 
By taking more care with energy considerations one can remove these small distortions to give genuine filtered $\Li$ homotopy equivalences.

Incidentally we will not need this in order to establish invariance properties of the capacities defined in the sequel, since as in Remark~\ref{rmk:caps_are_symp_invt} it suffices to show monotonicity under (exact) symplectic embeddings, which will follow directly from versions of the neck stretching composition ~\eqref{eq:neck_stretch_comp} (see \S\ref{subsec:properties}).
Indeed, suppose that $\mathfrak{c}$ is a real-valued invariant of symplectic fillings which is monotone under exact symplectic embeddings and scales like area, but which a priori depends on various auxiliary choices.
By the above we must have $\mathfrak{c}(X,\omega,\alpha,J) < e^\eps \mathfrak{c}(X,\omega,\alpha',J') < e^{2\eps}\mathfrak{c}(X,\omega,\alpha,J)$, and hence $\mathfrak{c}(X,\omega,\alpha,J) = \mathfrak{c}(X,\omega,\alpha',J')$ since $\eps$ is arbitrary.
If $\mathfrak{c}$ is moreover monotone under arbitrary symplectic embeddings, then $\mathfrak{c}(X,\omega,\alpha)$ is entirely independent of the choice of the contact one-form $\alpha$, i.e. it is a symplectomorphism invariant.

\subsubsection{The $\Li$ structure on linearized contact homology}\label{subsubsec:L-inf_structure_on_chlin}

Given a symplectic filling $(X,\omega)$, we can also linearize the $\Li$ structure on $\cha(X)$ to get a much smaller $\Li$ structure whose underlying chain complex is $\chlin(X)$.
More precisely, we define operations $\ell^k_\lin: \odot^k\X \rightarrow \X$
which together form an $\Li$ algebra structure on $s^{-1}\X$, and such that 
$\ell^1_\lin = \bdy_\lin$ is the linearized contact homology differential.

As a special case of the cobordism map $\Phi$ from \S\ref{subsubsec:Li_str_on_CHA}, the filling $X$ induces
an $\Li$ augmentation $\aug: \cha(Y) \rightarrow \Lamo$.
Let $w_1,...,w_k \in \calA_{Y}$ be words as above.
We define an $\Li$ homomorphism $F^\aug: \calA \rightarrow \calA$ by
\begin{align}\label{eq:Faug}
F^{\aug;k}(w_1,...,w_k) \\= &\sum_{\substack{J_1\sqcup ... \sqcup J_m \sqcup K = \{1,...,N\} \\ A_1,...,A_m\;:\; A_i \in H_2(X,\Gamma_i^+\cup \nil;\Z) }} \#\ovl{\calM}_{X,A_1}(\Gamma_1^+;\nil)...\#\ovl{\calM}_{X,A_m}(\Gamma_m^+;\nil) T^{[\omega]\cdot(A_1+...+A_m)}x_{\Gamma_K}, \nonumber
\end{align}
where $\Gamma_i^+ \in \orbset_{|J_i|}(Y,\alpha)$ corresponds to $J_i$ for $i = 1,...,m$, $ \Gamma_K \in \orbset_{|K|}(Y,\alpha)$ corresponds to $K$, and the sum is over all partitions $J_1 \sqcup ... \sqcup J_m \sqcup K = \{1,...,N\}$ with the property that the graph $G(I_1,...,I_k;J_1,...,J_m)$ is simply connected.
Note that the map $F^{\aug;1}$ agrees with the automorphism of $\calA$ defined in \S\ref{subsubsec:cha_and_its_linearization} (and denoted by $F^{\aug}$ there).
Moreover, the $\Li$ homomorphism $F^\aug$ is invertible,
with inverse $F^{-\aug}$.
We can therefore define a new $\Li$ structure on $s^{-1}\calA$ 
via 
\begin{align}\label{eq:whellaug}
\wh{\ell}^\aug := \wh{F^\aug} \circ \wh{\ell} \circ \wh{F^{-\aug}}.
\end{align}
This corresponds to defining new $\Li$ operations $\ell^{\aug;1},\ell^{\aug;2},\ell^{\aug;3},...$, which have the property that
$\pi_0 \circ \ell_\aug^k = 0$ for each $k$.
Therefore, we can restrict the operations $\ell^{\aug;k}$ to inputs from $\X$ and post-compose with the projection $\pi_1$ to get $\Li$ operations $\ell^{k}_{\lin}: \odot^k \X \rightarrow \X$ for $k=1,2,3,...$. 
In particular, the differential $\ell^{1}_{\lin}$ agrees with the linearized contact homology differential $\bdy_{\lin}$ from before.

As in the case of the linearized contact homology differential, 
we can interpret the $\Li$ operations on $\chlin$ more geometrically as follows.
The structure coefficient $\langle \ell^k_\lin(x_{\gamma_1},...,x_{\gamma_k}),x_\eta\rangle$ counts rational curves in the symplectization of $Y$ with $k$ positive ends $\gamma_1,...,\gamma_k$ and $1$ negative end $\eta$, plus some number of additional anchors in $X$. Unlike the case of $\bdy_\lin$, the symplectization level can actually be disconnected,
with one nontrivial component and some number of trivial cylinders, and the anchors in $X$ can have multiple positive ends. See Figure \ref{chlin_Linf} for a cartoon.
\begin{figure}
 \centering
  \includegraphics[scale=.6]{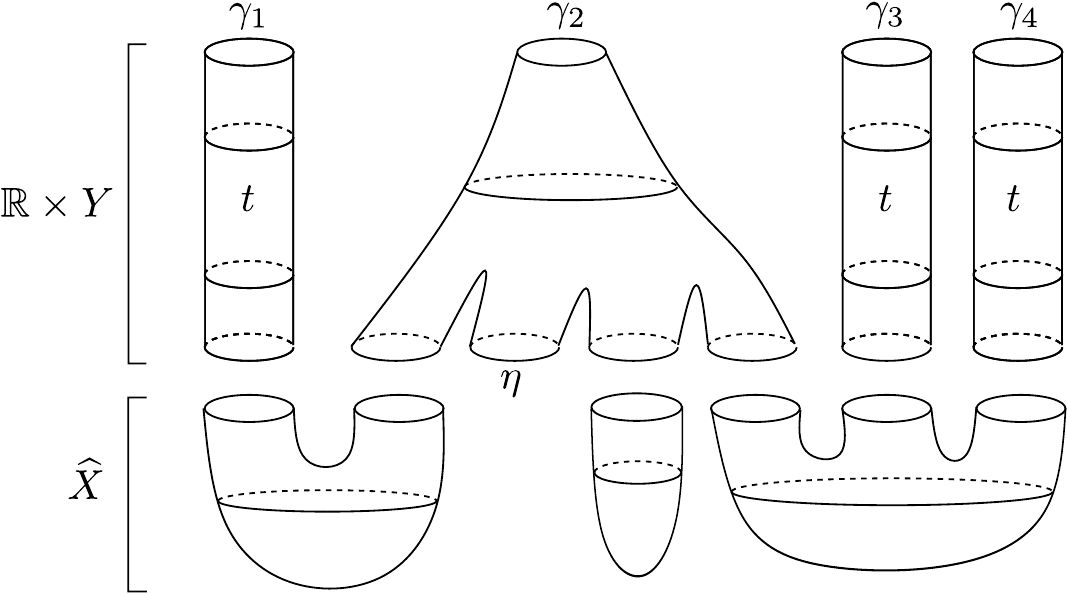}
 \caption{A typical configuration contributing to the coefficient $\langle \ell^4_{\lin}(x_{\gamma_1},x_{\gamma_2},x_{\gamma_3},x_{\gamma_4}),x_{\eta}\rangle$. Here the cylinders labeled by $t$ are trivial cylinders.}
 \label{chlin_Linf}
\end{figure}

Suppose that  $(X_0,\omega_0)$ and $(X_1,\omega_1)$ are symplectic fillings, with a symplectic embedding $X_0 \hookrightarrow X_1$ such that $X_1 \setminus X_0$ is an exact symplectic cobordism.
By linearizing the $\Li$ cobordism map $\Phi: \cha(\bdy^+ X_1) \rightarrow \cha(\bdy^+ X_0)$, we get an $\Li$ homomorphism
$$ \Phi_{\lin}: \chlin(X_1) \rightarrow \chlin(X_0),$$
well-defined up to filtered $\Li$ homotopy.
Also, analogous to the neck stretching discussion at the end of \S\ref{subsubsec:Li_str_on_CHA}, given a symplectic filling $W$ and exact symplectic cobordisms $X_1,X_2$ with $\bdy^+ W = \bdy^- X_1$ and $\bdy^+X_1 = \bdy^-X_2$, the fillings given by concatenating $W,X_1,X_2$ induce filtered $\Li$ homomorphisms 
\begin{itemize}
  \item $\Phi_{X_2,\lin}: \chlin(W \circledcirc X_1 \circledcirc X_2) \ra \chlin(W \circledcirc X_1)$
  \item $\Phi_{X_1,\lin}: \chlin(W \circledcirc X_1) \ra \chlin(W)$
  \item $\Phi_{X_1 \circledcirc X_2,\lin}: \chlin(W \circledcirc X_1 \circledcirc X_2) \ra \chlin(W)$, 
\end{itemize}
and there is a filtered $\Li$ homotopy
\begin{align*}
\Phi_{X_1 \circledcirc X_2,\lin} \sim \Phi_{X_1,\lin}\circ \Phi_{X_2,\lin}: \chlin(W \circledcirc X_1 \circledcirc X_2)
 \ra \chlin(W). 
\end{align*}
See also \cite[\S2.7]{latschev2022remarks} for another perspective on the structure of linearized RSFT.

\section{The Cieliebak--Latschev element}\label{sec:CL}

A {\em non-exact} symplectic cobordism $(X,\omega)$ between $(Y^+,\alpha^+)$ and $(Y^-,\alpha^-)$ does {\em not} generally induce chain maps 
\begin{align*}
\cha(Y^+) \rightarrow \cha(Y^-),\\
\chlin(Y^+) \rightarrow \chlin(Y^-),
\end{align*}
let alone $\Li$ homomorphisms.
The basic issue is that an index $1$ curve in $X$ can break into
a configuration involving ``reverse anchors'', i.e. rational curves in $X$ with only negative ends. 
These reverse anchors, which pose an obstruction to defining a chain map, are ruled out for exact symplectic cobordisms by Stokes' theorem.
Nevertheless, inspired by \cite{fukaya2006application}, Cieliebak--Latschev observed that we can still define transfer maps, provided we appropriately deform the target.
In fact, the reverse anchors assemble to define a Maurer--Cartan element $\cl$ in $\cha(Y^-)$, and there is an $\Li$ homomorphism
$$\Phi: \cha(Y^+) \rightarrow \cha_{\cl}(Y^-),$$ where the latter denotes the $\Li$ structure with deformed operations $\ell^1_{\cl},\ell^2_\cl,\ell^3_\cl,...$ as in \S\ref{subsec:MC_theory}.

More precisely, we define $\cl \in \cha(Y^-)$ by
\begin{align*}
\cl := \sum_{\substack{\Gamma^- \in \orbset(Y^-,\alpha^-)\\ A \in H_2(X,\nil \cup \Gamma^-;\Z)}} \# \ovl{\calM}_{X,A}(\nil;\Gamma^-) \frac{T^{[\omega]\cdot A}}{\mu_{\Gamma^-} \kappa_{\Gamma^-}}x_{\Gamma^-}.
\end{align*}
Recall that by definition $\cha(Y^-)$ is $T$-adically complete, and hence the above infinite sum is well-defined since by the SFT compactness theorem there can only be finitely many nonzero terms with bounded $T$ exponent.

To see that this is a Maurer--Cartan element, one analyzes the boundary strata of the spaces $\ovl{\calM}_{X,A}(\nil;\Gamma^-)$ in the case that $\ind_{X,A}(\nil;\Gamma^-) = 1$.
Namely, each boundary stratum is a two-level building, with an index $0$ level in $X$ and an index $1$ level in the symplectization of $Y$, and this contributes to a term of the form $\ell^k(\cl,...,\cl)$. 
A similar analysis, looking at moduli spaces of index $1$ curves in $X$, confirms that
$\Phi$ is an $\Li$ homomorphism from $\cha(Y^+)$ to $\cha_{\cl}(Y^-)$.
Arguing as in \S\ref{subsubsec:cha_and_its_linearization}, one can also show that $\cl$ is independent of any choices up to gauge equivalence.

\begin{remark}
We can view a symplectic filling $(X,\omega)$ of $(Y,\alpha)$ as a cobordism with negative end the empty set and trivial Cieliebak--Latschev element.
In particular, in this case the $\Li$ augmentation $\aug: \cha(Y) \rightarrow \Lamo$ is defined without any deformation.
\end{remark}

Now suppose that $(X_0,\omega_0)$ is a symplectic filling of $(Y^-,\alpha^-)$, and let $(X,\omega)$ be a (possibly non-exact) symplectic cobordism with positive end $(Y^+,\alpha^+)$ and negative end $(Y^-,\alpha^-)$.
In this situation we have a linearized version of $\cl \in \cha(Y^-)$, which is a Maurer--Cartan element $\cllin$ for the $\Li$ structure on $\chlin(X_0)$.
Namely, we push forward $\cl$ under the $\Li$ homomorphism $F^\aug: \cha(Y^-) \rightarrow \cha(Y^-)$,
and then set $$\cllin := \pi_1(F^\aug_*(\cl)) \in \chlin(X_0).$$
Let $\op{CH}_{\op{lin,\cllin}}(X_0)$ denote the $\Li$ structure on $\chlin(X_0)$ deformed by $\cllin$.
We can also linearize $\Phi$ to define an $\Li$ homomorphism
$\Phi_\lin: \chlin(X_0 \circledcirc X) \rightarrow \ch_{\lin,\cllin}(X_0)$.
See Figure \ref{Phi_lin_cl_and_cl_lin}.
\begin{figure}[!tbp]
  \centering
  \begin{minipage}[b]{0.48\textwidth}
\includegraphics[width=\textwidth]{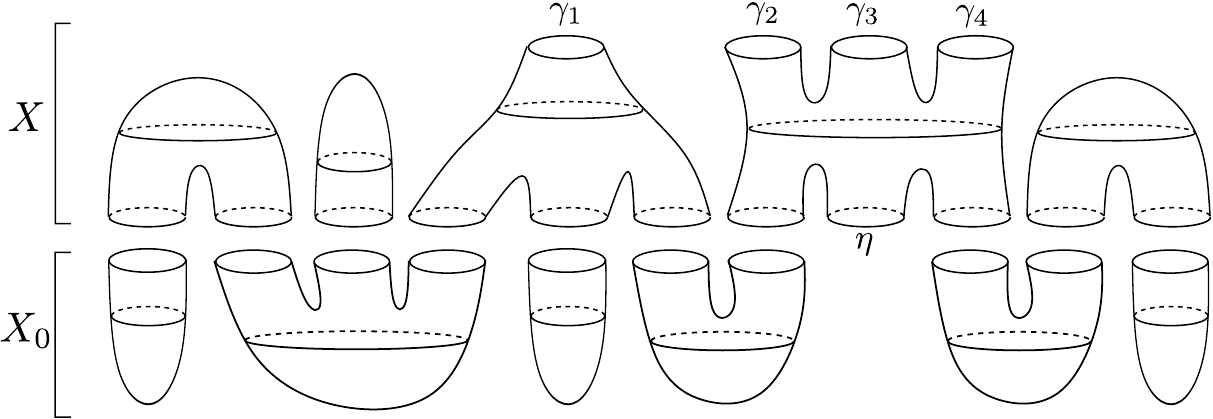}
  \end{minipage}
  \hfill
  \begin{minipage}[b]{0.48\textwidth}
   \includegraphics[width=\textwidth]{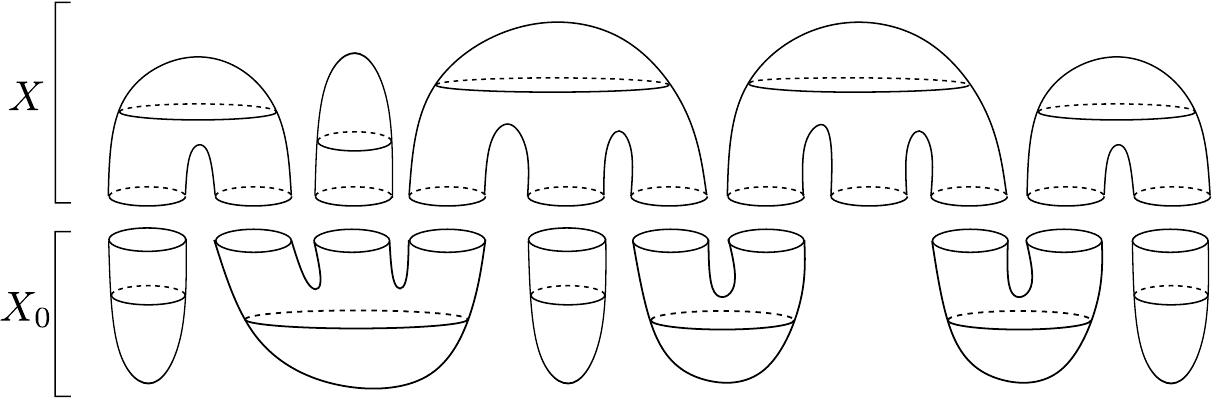}
  \end{minipage}
  \caption{Left: a typical configuration contributing to the coefficient $\langle \Phi^4_{\lin}(x_{\gamma_1},x_{\gamma_2},x_{\gamma_3},x_{\gamma_4}),x_{\eta}\rangle$, in the case of a non-exact symplectic cobordism $X$. Right: a typical configuration contributing to the Maurer--Cartan element $\cllin \in \chlin(X_0)$.}
 \label{Phi_lin_cl_and_cl_lin}
\end{figure}

\sss

Now consider the case that $(M,\beta) = (X^-,\omega^-) \circledcirc (X^+,\omega^+)$ 
is a closed symplectic manifold which decomposes into a symplectic filling $(X^-,\omega^-)$ and a symplectic cap $(X^+,\omega^+)$ with common contact boundary $(Y,\alpha)$. 
For simplicity, assume there are no rigid closed curves in $(X^-,\omega^-)$ or $(X^+,\omega^+)$ individually.
We have the Cieliebak--Latschev element $\cl \in \cha(Y)$,
and the augmentation $\aug: \cha(Y) \rightarrow \Lamo$ induced by the filling $X^-$.
We can also consider the genus zero Gromov--Witten invariant
$$\gw_M = \sum_{A \in H_2(M;\Z)}T^{[\omega] \cdot A} \gw_{M,A} \in \Lamo,$$
where for a homology class $A \in H_2(M;\Z)$ we put $\gw_{M,A} = \#\ovl{\calM}_{M,A}\in \Q$.
In this situation, by analyzing what happens when we neck stretch closed curves in $M$ along the contact-type hypersurface $Y$, we find the following relation:
$$\aug_*(\cl) = \gw_{M},$$
where $\aug_*(\cl) \in \Lamo$ is defined as in \S\ref{subsec:MC_theory} by 
$$\aug_*(\cl) = \sum_{k =1}^{\infty} \frac{1}{k!}\aug^k(\underbrace{\cl,...,\cl}_k) = \pi_1 \wh{\aug}(\exp(\cl)).$$
Similarly, we have 
\begin{align}\label{eq;auglin_cllin_gw}
(\auglin)_*(\cllin) = \gw_M.
\end{align}
In the next section we define various generalizations of the $\Li$ augmentation $\aug$ which count curves with additional geometric constraints, and the corresponding analogues of \eqref{eq;auglin_cllin_gw} will be the key ingredient used in \S\ref{subsubsec:gw_upper_bound} to give upper bounds on capacities.

\sss

In some cases we can make a more refined statement at the level of individual homology classes.
For example, suppose that $X^-$ is a symplectic filling of $Y$ such that
$H_1(Y;\Z) = H_2(Y;\Z) = H_2(X^-) = 0$ (e.g. $X^-$ is a star-shaped domain).
Let $X^+$ be a cap of $Y$, and set $M = X^- \circledcirc X^+$.
In this situation, each homology group $H_2(X^+;\nil\cup\Gamma^-;\Z)$
is naturally identified with $H_2(X^+;\Z)$, and the inclusion $H_2(X^+;\Z) \rightarrow H_2(M;\Z)$ is an isomorphism. 
Under these identifications, curves in $X^+$ define classes in $H_2(M;\Z)$,
and we have
$$\wh{\ell}(\exp(\cl)_A) = 0,$$
where $\exp(\cl)_A$ denotes the component of $\exp(\cl)$ in the class $A \in H_2(M;\Z)$.
Moreover, we have the relation 
$$\aug_*(\exp(\cl)_A) = T^{[\omega]\cdot A}\gw_{M,A},$$
and similarly for the other relations discussed above.

\section{Geometric constraints and augmentations}\label{sec:geometric_constraints}

Let $(Y,\alpha)$ be a strict contact manifold, with symplectic filling $(X,\omega)$. As discussed in \S\ref{subsec:rsft_formalism}, the filling induces an $\Li$ augmentation 
$\aug: \cha(Y) \rightarrow \Lamo$,
which is used to define the linearized $\Li$ algebra $\chlin(X)$.
Recall that the augmentation $\aug$ counts index $0$, (possibly disconnected) rational curves in $X$.
In this section, we define various other $\Li$ augmentations $\cha(Y) \rightarrow \Lamo$.
These can be interpreted as counting rational curves in $X$ which become rigid after imposing additional geometric constraints:
\begin{itemize}
\item
requiring curves to pass through $r$ generic points $p_1,...p_r \in X$
\item
requiring curves to pass through a chosen point in $p \in X$ with an order $m$ local tangency constraint at that point
\item requiring curves to have an order $b$ self-contact singularity at a chosen point $p \in X$ (e.g. $b=2$ corresponds to having a node at $p$),
\end{itemize}
as well as hybrids of these conditions.
By counting curves in $X$ which satisfy these extra constraints, we get corresponding $\Li$ augmentations $\cha(Y) \rightarrow \Lamo$,
denoted by $\aug\lll p_1,...,p_r\rrr$, $\aug\lll \T^{m}p\rrr$, and $\aug\lll \underbrace{p,...,p}_b\rrr$ respectively.
The first type of constraint is a standard part of symplectic field theory and is described (in slightly different language) in \cite{EGH2000}. The second type of constraint is based on the setup of Cieliebak--Mohnke \cite{CieliebakMohnkeInventiones}.
The local tangency constraints we consider are closely related to but subtly different from a gravitational descendant condition at a point (see Remark \ref{rmk:tangency_versus_descendant} for more details).
The third type of constraint is most naturally formulated in terms of curves in a multiple blowup of $X$ in a specified homology class.
In \S\ref{subsec:extra_negative_ends}, we combine all of these constraints into a unified framework by considering curves in a symplectic cobordism with extra negative ends.
The basic idea is to count curves with specified ends after removing 
small neighborhoods of points $p_1,...,p_r$ from $X$.
This allows us to combine constraints into general augmentations of the form $\aug\lll (\T^{m^1_1}p_1,...,\T^{m^1_{b_1}}p_1) ,..., (\T^{m^r_1}p_r,...,\T^{m^r_{b_r}}p_r) \rrr$, which can all be defined within a standard SFT framework.
By plugging these augmentations in the formalism of \S\ref{subsec:the_general_construction}, these give rise to symplectic capacities. 
In \S\ref{subsec:ppt}, we discuss a relation between the three types of constraints, based on the fact that constraints at different points in the target space can be ``pushed together'' into constraints at a single point.

\begin{remark}
The above constraints can also be posed for closed curves in say a closed symplectic manifold $(M,\beta)$, giving rise to Gromov--Witten-type invariants rather than $\Li$ augmentations. 
We denote the corresponding curve counts in homology class $A \in H_2(M;\Z)$ by $$\gw_{M,A}\lll (\T^{m^1_1}p_1,...,\T^{m^1_{b_1}}p_1) ,..., (\T^{m^r_1}p_r,...,\T^{m^r_{b_r}}p_r) \rrr \in \Q.$$ 
\end{remark}

\sss

Most of the above $\Li$ augmentations are {\em not} stable under products with $\C$. Indeed, each constraint $\T^{m^i_j}p_k$ prescribes a negative end of a curve, whereas we only expect curves with one negative end to stabilize (c.f. \eqref{eq:stab_ind_intro}).
The exception is the $\Li$ augmentation $\aug\lll \T^m p\rrr$, which is formulated in terms of a single negative end.
As such, it linearizes to an augmentation
$$ \auglin\lll \T^m p \rrr: \chlin(X) \rightarrow \Lamo$$
which is dimensionally stable.
The details are discussed in \S\ref{subsec:stabilization}.

\subsection{Point constraints}

Let $(X^{2n},\omega)$ be a symplectic cobordism between $(Y^+,\alpha^+)$ and $(Y^-,\alpha^-)$.
The induced $\Li$ cobordism map $\Phi: \cha(Y^+) \rightarrow \cha(Y^-)$
is defined in terms of index $0$ moduli spaces $\ovl{\calM}_{X,A}(\Gamma^+;\Gamma^-)$.
We can also consider the versions $\ovl{\calM}_{X,A,r}(\Gamma^+;\Gamma^-)$ with an additional $r$ (ordered) marked points freely varying in the domain.
These moduli spaces come with evaluation maps $\ev: \ovl{\calM}_{X,A,r}(\Gamma^+;\Gamma^-) \rightarrow X^{\times r}$.
Fixing pairwise distinct points $p_1,...,p_r \in X$, we define the moduli space of curves with $r$ point constraints by 
$$\ovl{\calM}_{X,A}(\Gamma^+;\Gamma^-)\lll p_1,...,p_r\rrr := \ev^{-1}(\{p_1\} \times ... \times \{p_r\}).$$ 
Note that each point constraint cuts down the expected dimension by $2n-2$,
so $\ovl{\calM}_{X,A}(\Gamma^+;\Gamma^-)\lll p_1,...,p_r\rrr$ having index $0$ corresponds to $\ovl{\calM}_{X,A,r}(\Gamma^+;\Gamma^-)$ having 
index $r(2n-2)$. 
The combinatorial structure of the boundary strata of $\ovl{\calM}_{X,A}(\Gamma^+;\Gamma^-)\lll p_1,...,p_r\rrr$
is such that a typical element of $\bdy\ovl{\calM}_{X,A}(\Gamma^+;\Gamma^-)\lll p_1,...,p_r\rrr$ consists of a pseudoholomorphic building with one ``main'' level in $X$ and some number of levels in the symplectization $\R \times Y^\pm$.
Note that the curves in the main level may be disconnected, and the point constraints $p_1,...,p_r$ are distributed arbitrarily between the different components in $X$.

We define the point-constrained $\Li$ cobordism map $\Phi\lll p_1,...,p_r\rrr: \cha(Y^+) \rightarrow \cha(Y^-)$ 
in the same way that we defined the $\Li$ cobordism map $\Phi: \cha(Y^+) \rightarrow \cha(Y^-)$, but now imposing the point constraints $p_1,...,p_r$.
This can also be seen as a special case of the blowup $\Li$ homomorphism discussed in \S\ref{subsec:blowup} below, where the extension of \eqref{eq:Phi_k} is given by \eqref{eq:Phi_k_blowup}.
In particular, if $(X,\omega)$ is a symplectic filling, we get an $\Li$ augmentation
$$\aug\lll p_1,...,p_r \rrr: \cha(Y) \rightarrow \Lamo.$$
As in \S\ref{subsec:rsft_formalism}, it follows from the general invariance properties of rational symplectic field theory that $\aug\lll p_1,...,p_r\rrr$ is independent of all choices, including the precise locations of the points $p_1,...,p_r\in X$ (assuming that $X$ is connected), up to filtered $\Li$ homotopy.

\begin{remark}
More generally, instead of requiring curves to pass through points $p_1,...,p_r$, we could require them to pass through $r$ chosen cycles in $X$. 
However, in this paper we will only make use of the point class, since this is the only homology class relevant to all symplectic fillings regardless of their topological type.
\end{remark}

\begin{remark}\label{rmk:another_defn_of_point_aug}
There is another $\Li$ augmentation $\cha(Y) \rightarrow \Lamo$ which is very similar to $\aug\lll p_1,...,p_r\rrr$ but does not reference the filling $(X,\omega)$.
It is defined in the same way except that we pick the points $p_1,...,p_r$ to lie in $\R \times Y$ rather than $X$, and we count curves in $\R \times Y$ rather than in $X$.
These curves still have no negative ends, and here we think of $\R \times Y$ as a symplectic cobordism rather than a symplectization (i.e. we do not quotient by translations in the target).
In some cases these two $\Li$ augmentations coincide (up to filtered $\Li$ homotopy). Namely, if for every $\Gamma^+ \in \orbset(Y,\alpha)$ we have
$\#\ovl{\calM}_X(\Gamma^+;\nil) = 0$, then by a neck stretching argument
the relevant curves in $\wh{X}$ are in correspondence with curves in $\R \times Y$.

For example, this holds for ellipsoids $E(a_1,...,a_n)$ for index reasons.
More generally, if we have $c_1(X,\omega) = 0$ and every Reeb orbit $\gamma$ of $Y$ satisfies $\cz(\gamma) > |n-3|$, then for all $\Gamma^+ = (\gamma_1^+,...,\gamma_k^+) \in \orbset(Y,\alpha)$ we have $\ind\; \ovl{\calM}_X(\Gamma^+;\nil) > 0$, so there are no rigid genus zero anchors in $X$.
Indeed, we have
$$\ind\; \ovl{\calM}_X(\Gamma^+;\nil) = (n-3)(2-k) + \sum_{i=1}^k \cz(\gamma_i^+) > (n-3)(2-k) + k|n-3| \geq 0.$$
\end{remark}

For $r=1$, the $\Li$ augmentation $\aug\lll p \rrr: \cha(Y) \rightarrow \Lamo$ induces
a linearized $\Li$ augmentation $\aug_\lin\lll p \rrr: \chlin(X) \rightarrow \Lamo$.
As mentioned in the introduction, its first component $\aug_\lin^1\lll p \rrr: \chlin(X) \rightarrow \Lamo$ is part of the SFT model for symplectic cohomology considered in \cite{bourgeois2009exact}.
For $r \geq 2$ the $\Li$ augmentation $\aug\lll p_1,...,p_r \rrr: \cha(Y) \rightarrow \Lamo$ does {\em not} induce an $\Li$ augmentation $\chlin(X) \rightarrow \Lamo$. Indeed, even the chain map $\aug^1\lll p_1,...,p_r \rrr: \cha(Y) \rightarrow \Lamo$ does not induce a chain map $\chlin(X) \rightarrow \Lamo$.
The essential reason is that an index $1$ curve in $X$ with two point constraints can degenerate into a configuration for which the marked points lie on different curve components in $X$.

\subsection{Local tangency constraints}\label{subsec:local_tan_con}

Following Cieliebak--Mohnke \cite{CieliebakMohnkeInventiones,CieliebakMohnkeHypersurfaces}, we can consider curves in a symplectic cobordism $(X,\omega)$ satisfying certain local tangency constraints. 
Let $J$ be an admissible almost complex structure on $\wh{X}$ which is integrable near a chosen point $p \in X$,
and let $D$ be a germ of complex divisor in $\Op(p)$.
Suppose we have a $J$-holomorphic map $u: \Sigma \rightarrow \wh{X}$
and a marked point $z \in \Sigma$ such that $u(z) = p$.
In this situation there is a well-defined notion of $u$ being tangent to $D$ of order $m$ at $z$, 
which one can formulate in local coordinates and then show to be coordinate-independent (see \cite{CieliebakMohnkeInventiones} for details).
Let $\calM_{X,A}(\Gamma^+;\Gamma^-)\lll \T^m p\rrr$ be defined in the same way as $\calM_{X,A}(\Gamma^+,\Gamma^-)$, but now with the constituent curves having one marked point $z$ which maps to $p$ and is tangent to $D$ of order $m$ at $z$.
It is shown in \cite[Prop 3.1]{CieliebakMohnkeInventiones} that for generic $J$ (which can be assumed to be fixed and integrable near $p$), the moduli space $\calM_{X,A}(\Gamma^+;\Gamma^-)\lll \T^m p\rrr$ is, at least near a somewhere injective curve $u$, a smooth manifold of dimension
$$(n-3)(2 - s^+ - s^-) + \sum_{i=1}^{s^+}\cz(\gamma_i^+) -  \sum_{j=1}^{s^-}\cz(\gamma_j^-) + 2c_1(X,\omega) \cdot A_u - (2n-2) - 2m.$$
Note that the $\T^m p$ constraint is a real codimension $2n-2 + 2m$ condition.
If we fix distinct points $p_1,...,p_r \in X$ and a divisor germ $D_i$ at each $p_i$,
for $m_1,...,m_r \in \Z_{\geq 0}$ we can similarly define the moduli space
$\calM_{X,A}(\Gamma^+,\Gamma^-)\lll \T^{m_1}p_1,...,\T^{m_r}p_r\rrr$
by considering curves with marked points $z_1,...,z_r$, with $z_i$ mapping to $p_i$ and tangent to $D_i$ of order $m_i$.

Now suppose that $(X,\omega)$ is a symplectic filling of $(Y,\alpha)$. By counting index $0$ curves in compactified moduli spaces $\ovl{\calM}_{X,A}(\Gamma^+,\nil)\lll \T^{m_1}p_1,...,\T^{m_r}p_r\rrr$, we get 
an $\Li$ augmentation
$\aug\lll \T^{m_1}p_1,..., \T^{m_r}p_r\rrr: \cha(Y) \rightarrow \Lamo$.
In the special case $r = 1$, we also have an induced linearized $\Li$ augmentation
$\aug_\lin\lll \T^m p\rrr: \chlin(X) \rightarrow \Lamo$.

\begin{remark}
In a closed symplectic manifold, the analogous counts of closed curve with tangency constraints are considered in \cite{enumerating}. It is shown in that setting that these give enumerative invariants which can be defined for all semipositive symplectic manifolds using only classical transversality techniques.
\end{remark}

 \begin{remark}[on the relationship between tangency constraints and gravitational descendants]\label{rmk:tangency_versus_descendant}
 Here we compare curve counts with the tangency constraint $\lll \T^m p\rrr$ to the gravitational descendant condition $\lll \psi^m p\rrr$.
For simplicity, assume that $(M^{2n},\beta)$ is a closed symplectic manifold, 
and let $A \in H_2(M;\Z)$ be a homology class such that $$
\ind_{M,A} = 2n - 6 + 2c_1(M,\beta) \cdot A - (2n-2) - 2m = 0.$$
In the context of algebraic Gromov--Witten theory, recall that one defines
$\gw_{M,A} \lll \psi^{m} p \rrr \in \Q$ as follows.
Consider the space $\ovl{\calM}_{M,A,1}$ of genus zero stable maps to $M$ in the homology class $A$ with one marked point.
This space comes with an evaluation map $\ev: \ovl{\calM}_{M,A,1} \rightarrow \CP^2$ and a complex line bundle $\mathbb{L} \rightarrow \ovl{\calM}_{M,A,1}$ whose fiber over a curve is the cotangent space to the curve at the marked point.
We then set $$\gw_{M,A}\lll \psi^m p\rrr := \int_{[\ovl{\calM}_{M,A}]^{\op{vir}}} \ev^*([p]^\vee)\cup\underbrace{c_1(\mathbb{L})\cup ... \cup c_1(\mathbb{L})}_{m},$$ where $[p]^\vee \in H^{2n}(M;\Q)$ denotes the point class and $[\ovl{\calM}_{M,A}]^{\op{vir}} \in H_{2n - 4 + 2c_1(M,\beta)\cdot A}(M;\Q)$ is the virtual fundamental class.
The case of open curves is more subtle, since then $\mathbb{L}$ is a line bundle over a space with codimension one boundary, but it does also appear in the literature (see for example \cite{eliashberg_sft_and_applications,pandharipande2014intersection,Fabert_descendants_2011}). Now suppose that $D$ is a smooth divisor passing through $p \in M$.
In \cite{Gathmann_absolute_and_relative} (see also \cite{tonkonog2018string}) it is observed that the line bundle $\mathbb{L}^{\otimes k}$ for $k \in \Z_{> 0}$ has a partially defined section $\sigma_k$, given by the $k$-jet of a curve at the marked point in the direction normal to $D$.
The intersection of the locus $\ev^{-1}(p) \cap \sigma_1^{-1}(0) \cap ... \cap \sigma_m^{-1}(0)$ with the main stratum $\calM_{M,A,1}$ consists of curves with smooth domain which are tangent to $D$ to order $m$ at the marked point.
The count of such curves would give precisely $m!\gw_{M,A}\lll\psi^m p\rrr$, 
provided we knew that the sections $\sigma_1,...,\sigma_m$ were sufficiently transverse and nonvanishing along the boundary strata $\ovl{\calM}_{M,A,1} \setminus \calM_{M,A,1}$.
However, it turns out that these sections always vanish on the boundary strata, and this typically gives additional contributions to the descendant invariant. 

As a simple example, we have $2!\gw_{\CP^1,\op{deg}=2} \lll\psi^2 p\rrr = 1/2$, which follows easily from the genus zero topological recursion formula (see for example \cite{kock2001notes}).
By contrast, we have $\gw_{\CP^1,\op{deg}=2}\lll \T^2 p\rrr = 0$.
In fact, the compactified moduli space $\ovl{\calM}_{\CP^1,\op{deg}=2}\lll \T^2 p\rrr$ is empty, essentially because there are no degree two Hurwitz maps $\CP^1 \rightarrow \CP^1$ with an order three branch point over $p$.
As a slightly more interesting example, we have $4!\gw_{\CP^2,2[L]}\lll \psi ^4 p\rrr = 3$,
whereas $\gw_{\CP^2,2[L]}\lll \T^4 p\rrr = 1$\footnote{As pointed out to the author by Chiu-Chu Melissa Liu, one can verify this computation by elementary projective geometry. However, if one wishes to use the standard integrable almost complex structure, one must take the local divisor $D$ to be sufficiently generic to avoid degenerate solutions.}
(see \S\ref{subsec:ppt} below).
One can show that discrepancies in both cases come from a configuration with two degree $1$ curves joined by a central ghost component containing one marked point.

Nevertheless, at least in the case of a global divisor, Gathmann \cite{Gathmann_absolute_and_relative} gives explicit descriptions of the extra boundary contributions occurring in the descendant invariants, described in terms of ``comb configurations'' related to the ghost configurations above.
This is related to the general fact that relative Gromov--Witten invariants can be reduced to absolute Gromov--Witten invariants, possibly with gravitational descendants (c.f. \cite{Maulik-Pandharipande_topological,Okounkov-Pandharipande_completed_cycles}). 
It might also be possible to extend this approach to the local divisor case in order to give
explicit descriptions of the descendant invariants $\gw_{M,A}\lll \psi^m p\rrr$ in terms of the tangency invariants $\gw_{M,A}\lll \T^m p\rrr$, 
plus extra discrepancy configurations.
\end{remark}

\subsection{Blowup constraints}\label{subsec:blowup}

We also count curves with nodal singularities or higher order self-contact singularities. These conditions are most naturally posed by considering curves in a suitable symplectic blowup at finitely many points.
First suppose that $(M^{2n},\beta)$ is a closed symplectic manifold with $2n \geq 4$. Fix distinct points $p_1,...,p_r \in X$ and let $\wt{X}$ denote a symplectic blowup of $X$ at the points $p_1,...,p_r$.
Recall that the symplectic blowup construction starts with a collection of disjoint Darboux balls $B_1,...,B_r$, with $B_i$ centered at $p_i$ and having area $a_i$ (i.e. radius $\sqrt{a_i/\pi}$).
Then $\wt{X}$ is given by removing each $B_i$ from $X$ and collapsing its boundary along the Hopf fibration.
Note that $\wt{X}$ depends on the points $p_1,...,p_r$ and balls $B_1,...,B_r$ up to symplectomorphism, but not up to symplectic deformation equivalence.
Each collapsed boundary $\bdy B_i$ gives rise to an exceptional divisor $\mathcal{E}_i \subset \wt{X}$, which is a symplectically embedded copy of $\CP^{n-1}$ whose line $L_i \subset \mathcal{E}_i$ has area $a_i$.
Note that the homology group $H_2(\wt{X};\Z)$ naturally splits as $H_2(X;\Z) \oplus \Z\langle [L_1],...,[L_r]\rangle$,
and we have $[\mathcal{E}_i] \cdot [L_i] = (-1)^{n-1}[pt]$,
where $[pt] \in H_0(\wt{X};\Z)$ is the point class.
Therefore we can uniquely write any class $\wt{A}$ in $H_2(\wt{X};\Z)$ as
$\wt{A} = A + C_1[L_1] + ... + C_r[L_r]$ for $C_1,...,C_r \in \Z$, where 
$\sum_{i=1}^r C_i[L_i]$ is the ``exceptional part'' of $\wt{A}$.

For $b_1,...,b_r \geq 1$, we introduce the following notation:
\begin{align*}
\gw_{X}\lll  \underbrace{(p_1,...,p_1)}_{b_1},...,\underbrace{(p_r,...,p_r)}_{b_r} \rrr := \sum_{A \in H_2(X;\Z)} T^{[\omega]\cdot A}\gw_{X,A}\lll\underbrace{(p_1,...,p_1)}_{b_1},...,\underbrace{(p_r,...,p_r)}_{b_r} \rrr,
\end{align*}
where $\gw_{X,A}\lll \underbrace{(p_1,...,p_1)}_{b_1},...,\underbrace{(p_r,...,p_r)}_{b_r} \rrr$ is defined to be $b_1!...b_r!\gw_{\wt{X},A-b_1[L_1]-...-b_r[L_r]}$.
Here the extra combinatorial factor $b_1!...b_r!$ has the effect of ordering the points mapping to each exceptional divisor $\mathcal{E}_i$, 
which is convenient since the marked points in the definition of $\gw_{X,A} \lll p_1,...,p_r \rrr$ are also ordered.
Note that the Novikov exponent $T^{[\omega]\cdot A}$ ignores the exceptional part $-b_1[L_1]-...-b_r[L_r]$,
and hence is independent of the sizes $a_i$ of the blowup, which we think of as being arbitrarily small.

Now suppose that $(X,\omega)$ is a symplectic cobordism between $(Y^+,\alpha^+)$ and $(Y^-,\alpha^-)$.
We can define an $\Li$ homomorphism 
$$\Phi\lll \underbrace{(p_1,...,p_1)}_{b_1},...,\underbrace{(p_r,...,p_r)}_{b_r} \rrr: \cha(Y^+) \rightarrow \cha(Y^-)$$ as follows.
The homology group $H_2(\wt{X},\Gamma^+\cup\Gamma^-;\Z)$ naturally splits as a direct sum of
$H_2(X,\Gamma^+\cup\Gamma^-;\Z)$ and $\Z\langle [L_1],...,[L_r]\rangle$.
The definition of $\Phi\lll \underbrace{(p_1,...,p_1)}_{b_1},...,\underbrace{(p_r,...,p_r)}_{b_r} \rrr $
is similar to that of the $\Li$ homomorphism $\Phi: \cha(Y^+) \rightarrow \cha(Y^-)$ from \S\ref{subsubsec:Li_str_on_CHA}, with $\Phi^k\lll \underbrace{(p_1,...,p_1)}_{b_1},...,\underbrace{(p_r,...,p_r)}_{b_r} \rrr (w_1,...,w_k)$ given by 
 \begin{align}\label{eq:Phi_k_blowup}
\sum_{\substack{J_1\sqcup ... \sqcup J_m = \{1,...,N\} \\ \Gamma_1^-,...,\Gamma_m^- \in \orbset(Y^-,\alpha^-) }} \#\ovl{\calM}_{\wt{X},\wt{A}_1}(\Gamma_1^+;\Gamma_1^-)...\#\ovl{\calM}_{\wt{X},\wt{A}_m}(\Gamma_m^+;\Gamma_m^-) \frac{T^{[\omega]\cdot(A_1+...+A_m)}b_1!...b_r!}{\mu_{\Gamma_1^-}...\mu_{\Gamma_m^-}\kappa_{\Gamma_1^-}...\kappa_{\Gamma_m^-}}x_{\Gamma_1^-}...x_{\Gamma_m^-},
\end{align}
but where the sum is now over all homology classes $\wt{A}_1,...,\wt{A}_m \in H_2(\wt{X},\Gamma_i^+ \cup \Gamma_i^-;\Z)$ such that the 
total exceptional part of $\wt{A}_1+...+\wt{A}_m$ is $-b_1[L_1] - ... - b_r[L_r]$.
In the case that $(X,\omega)$ is a filling of $(Y,\alpha)$, we get an $\Li$ augmentation
 
This does not generally induce an $\Li$ augmentation of $\chlin(X)$, unless $r = 1$ and $b_1 = 1$.
Also, the case $b_1 = ... = b_r = 1$ agrees with $\Phi\lll p_1 ,..., p_r \rrr$ from the previous subsection.

\subsection{Behavior under dimensional stabilization}\label{subsec:stabilization}
Here we discuss the extent to which curve counts persist to higher dimensions. 
We will show that the linearized contact homology $\Li$ algebra $\chlin(X)$
is invariant under taking a product with $\C$, in a suitable sense.
Similar statements hold for the $\Li$ augmentation $\aug\lll \T^k p \rrr: \chlin(X) \rightarrow \Lamo$ and the $\Li$ cobordism map $\chlin(X') \rightarrow \chlin(X)$ given an exact symplectic embedding $X \hookrightarrow X'$.
Together, these will directly translate into a stabilization result for symplectic capacities in \S\ref{subsubsec:behavior_under_stab}.

For simplicity we will assume that $X$ is a Liouville domain, with Liouville one-form $\la$.
Then Reeb orbits in $\bdy X$ have well-defined action, and for technical reasons in this subsection it will be more convenient to work over $\K = \Q$ rather than $\Lamo$, with $\chlin(X)$ endowed with its {\em increasing} action filtration $\calF_{<a}\chlin(X)$, $a \in \R_{> 0}$ (this is possible by our exactness assumption).
We note that since our capacities will be defined in terms of certain action minimizers subject to suitable constraints, we can safely ignore words of Reeb orbits of sufficiently high total action.

We will also make a simplifying assumption which allows us to ignore anchor curves in $\wh{X}$ for index reasons.
As in Remark~\ref{rmk:on_gradings_in_symplectic_fillings}, given Reeb orbits $\ga_1,\dots,\ga_k \in \bdy X$ with vanishing total homology class in $H_1(X)$, we can use a symplectic trivialization of $TX$ defined over a spanning surface to define the Conley--Zehnder indices $\cz(\ga_1),\dots,\cz(\ga_k)$, and the sum $\sum_{i=1}^k \cz(\ga_i)$ is independent of the choice of spanning surface provided that $2c_1(X) = 0$.
We will make the following assumption, which directly implies that any rational punctured curve without negative ends in $\R \times \bdy X$ or $\wh{X}$ has strictly positive index.
\begin{assumption}[``index positivity'']\label{assump:index_pos}
We have $2c_1(X) = 0$, and for any Reeb orbits $\ga_1,\dots,\ga_k$ in $\bdy X$ satisfying $\sum_{i=1}^k [\ga_i] = 0 \in H_1(X)$ we have $\sum_{i=1}^k \cz(\ga_i) > (n-3)(k-2)$.
\end{assumption}
\NI For instance, this holds if there is a global trivialization of $TX$ such that $\cz(\ga) > |n-3|$ for all Reeb orbits $\ga$ in $\bdy X$. Recall that a convex domain with smooth boundary in $\R^{2n}$ is a fortiori ``dynamically convex'' in the sense that all Reeb orbits $\ga$ satisfy $\cz(\ga) \geq n+1$ \cite{hofer1998dynamics}.

\begin{remark}
The papers \cite{Hind-Kerman_new_obstructions,CG-Hind_products,Ghost_stairs_stabilize,Mcduff_remark} also prove stabilization theorems in specific examples which are similar in spirit to and provide inspiration for the one in this paper.
\end{remark}

Naively, the main point for stabilization is that, for a split almost complex structure, curves in the product $X \times \C$ are given by a product of curves in each factor, and curves in $\C$ effectively must be constant.
However, there are a few subtleties to address before making this precise.
For one, the product $X \times \C$ is not strictly speaking a symplectic filling, due to its noncompactness. If we replace $\C$ with the two-ball $B^2(a)$ of area $a$, then $X \times B^2(a)$ is compact but has corners. We can smooth the corners to obtain an honest symplectic filling, but the smoothing process is not entirely canonical. In particular, different choices of smoothings lead to different Reeb dynamics.
Moreover, an admissible almost complex structure for the smoothed product is no longer split, so we cannot quite argue via the projection to the $\C$ factor.
Another issue is that the index of a punctured curve typically changes under stabilization and can potentially become negative, meaning that stabilized curves might be irregular and disappear after a small perturbation.
Since $\chlin$ effectively involves curves with one negative end, invariance under dimensional stabilization is at least plausible based on index considerations as in \S\ref{subsubsec:stab} (in contrast to $\cha$, which involves multiple negative ends). However, the $\Li$ structure maps of $\chlin$ technically count {\em anchored} curves with one negative end, hence the symplectization components could still have multiple negative ends and some care is needed.

To address these issues, we will take some care in constructing the smoothing of $X \times B^2(S)$, and we then use intersection theory for punctured pseudoholomorphic curves to rule out any ``unexpected curves'' after stabilizing.
In more detail, assume that $\bdy X$ has nondegenerate Reeb orbits. By \cite[Lem. 3.6.2]{enumerating2}, for any $S \in \R_{>0}$ we can construct a smoothing $X \smx B^2(S)$ of $X \times B^2(S)$ such that 
\begin{itemize}
  \item $X \times \{0\} \subset X \smx B^2(S) \subset X \times B^2(S)$
  \item each Reeb orbit $\ga$ in $\bdy X$ with action less than $S$ is also a Reeb orbit $\ga^\stab$ in $\bdy(X \smx B^2(S))$ lying in $\bdy X \times \{0\}$, and it is nondegenerate with normal Conley--Zehnder $\cz^\perp(\ga^\stab) = 1$
  \item any other Reeb orbit of $\bdy (X \smx B^2(S))$ has action greater than or equal to $S$.
\end{itemize}
Here the {\em normal Conley--Zehnder index} measures winding in the direction normal to the divisor $X \times \{0\} \subset X \smx B^2(S)$ and is taken with respect to the obvious choice trivialization of its symplectic normal bundle coming from the product splitting of $X \times B^2(S)$ (see e.g. \cite{MorS} or \cite[\S3.6]{enumerating2} for more details).

As a shorthand put $X^\stab := X \smx B^2(S)$.
Note that any (partial) trivialization of $TX$ naturally extends to a split trivialization of $TX^\stab$ along $X \times \{0\}$, 
and for a Reeb orbit $\ga$ in $\bdy X$ with $\cz^\perp(\ga^\stab) = 1$ we have $\cz(\ga^\stab) = \cz(\ga) + 1$.
 Using Assumption~\ref{assump:index_pos}, any rational punctured curve $C$ in $\R \times \bdy X^\stab$ or $\wh{X}^\stab$
with all positive ends lying in $\bdy X$ and no negative ends satisfies 
\begin{align}\label{eq:stab_curve_ind_at_least_2}
\ind(\wt{C}) > 2.
\end{align}

Note that on the level of $\K$-modules we have a natural inclusion $\jmath: \chlin(X) \hookrightarrow \chlin(X^\stab)$ sending $\ga$ to $\ga^\stab$.
We seek to show that for good Reeb orbits $\ga_1,\dots,\ga_k$ in $\bdy X$ with total action less than $S$ we have
\begin{align}
\jmath(\ell^k_\lin(\ga_1,\dots,\ga_k)) = \ell^k_\lin(\ga_1^\stab,\dots,\ga_k^\stab).
\end{align}
We take an admissible almost complex structure $J_{X^\stab}$ on $\wh{X}^\stab$ which restricts to an admissible almost complex structure $J_X$ on $\wh{X}$,
and similarly let $J_{\bdy X^\stab}$ be an admissible almost complex structure on $\R \times \bdy X^\stab$ which restricts to an admissible almost complex structure $J_{\bdy X}$ on $\R \times \bdy X$.
We assume that $J_{X^\stab}$ and $J_X$ restrict to $J_{\bdy X^\stab}$ and $J_{\bdy X}$ respectively on their cylindrical ends.

\begin{prop}\label{prop:Linf_operations_stabilize}
    For any $S > 0$ there is a natural filtration-preserving injective chain map 
\begin{align}\label{eq:Linf_ops_stab_filt_bar}
    \calF_{< S}\bar\chlin(X) \hookrightarrow \calF_{< S}\bar\chlin(X \smx B^2(S)).
\end{align}
\end{prop}
\begin{remark}
The appearance of bar complexes here reflects the fact that $\calF_{< S}\chlin(X)$ is not itself an $\Li$ subalgebra of $\chlin(X)$, since $\ga_1,\dots,\ga_k \in \calF_{<S}\chlin(X)$ only implies $\ell^k(\ga_1,\dots,\ga_k) \in \calF_{< kS}\chlin(X)$, whereas $\calF_{< S}\bar\chlin(X)$ effectively encodes all $\Li$ operations in $\chlin(X)$ such that the inputs have total action less than $S$.  
\end{remark}

\begin{proof}[Proof of Proposition~\ref{prop:Linf_operations_stabilize}]
For a Reeb orbit $\eta$ in $\bdy X^\stab$, the compactified moduli space $\ovl{\calM}^\stab$ computing the term $\langle \ell^k_\lin(\ga_1^\stab,\dots,\ga_k^\stab),\eta \rangle$ consists of anchored pseudoholomorphic buildings $C$ involving one or more symplectization levels in $\R \times \bdy X^\stab$ and possibly a level in $\wh{X}^\stab$ (these are the anchors).
By our action assumption, every asymptotic Reeb orbit which appears in a component of $C$ has action less than $S$ and hence must lie in $\bdy X \times \{0\}$. In particular, we have $\eta = \ga_0^\stab$ for some Reeb orbit $\ga_0$ in $\bdy X$.
Let $\ovl{\calM}$ denote the compactified moduli space computing the corresponding unstabilized term $\langle \ell^k_\lin(\ga_1\dots,\ga_k),\ga_0 \rangle$.
By our choice of almost complex structures there is a natural inclusion $\ovl{\calM} \hookrightarrow \ovl{\calM}^\stab$.
Given $C \in \ovl{\calM}^\stab$, we seek to show that every constituent curve component of $C$ lies in $\R \times \bdy X$ or $\wh{X}$, i.e. $C$ corresponds to an element of $\ovl{\calM}$.

Assume for the moment that the unstabilized moduli space $\ovl{\calM}$ is regular, and in particular for a configuration in $\ovl{\calM}$ every constituent curve component has nonnegative index. Then, since the total index is $1$, every configuration in $\ovl{\calM}$ consists of just a single curve component, which necessarily lies in $\R \times \bdy X$ (note that anchors in $\wh{X}$ are ruled out thanks to Assumption~\ref{assump:index_pos}).
By Lemma~\ref{lem:stab_symp_curves_are_slice} below, every component $C_0$ of $C$ with at least one negative end must lie in $\R \times (\bdy X \times \{0\})$, so it has nonnegative index as a curve in $\R \times \bdy X$, and hence also as a curve in $\R \times \bdy X^\stab$ (recall the index stabilization formula \eqref{eq:stab_ind_intro} from \S\ref{sec:intro}). One can also show that $C_0$ is still regular as a curve in the stabilization and counts with the same sign (see \cite[\S A]{enumerating2}).
Since the total index of $C$ is $1$ and every component of $C$ without any negative ends has index at least $3$ by \eqref{eq:stab_curve_ind_at_least_2}, we conclude that $C$ has just a single component and it is contained in $\R \times (\bdy X \times \{0\})$, and hence indeed corresponds to an element of $\ovl{\calM}$.

\begin{lemma}\label{lem:stab_symp_curves_are_slice}
Any $J_{\bdy X^\stab}$-holomorphic curve component in $\R \times \bdy X^\stab$, all of whose asymptotic Reeb orbits have action less than $S$ and having at least one negative end, is entirely contained in $\R \times (\bdy X \times \{0\})$. 
\end{lemma}
\begin{proof}

This is a minor variation on the proof of Lemma 3.6.3 in \cite{enumerating2}.
Using a version of Siefring's intersection theory \cite{siefring2011intersection}, for such a curve component $C_0$ which does {\em not} lie in $Q$, we have 
\begin{align}\label{eq:Siefring_type_eq}
\op{push}(C_0) \cdot Q = C_0 \cdot Q - \sum\limits_{z\;\text{pos. punc.}} \wind_z + \sum\limits_{z\;\text{neg. punc.}} \wind_z,
\end{align}
where $Q$ denotes the (noncompact) divisor $\R \times (\bdy X \times \{0\})$.
Here $C \cdot Q$ is the count of intersection points between $C_0$ and $Q$ before perturbing, which is necessarily nonnegative by positivity of intersections.
Meanwhile, $\op{push}(C) \cdot Q$ is the homological intersection after perturbing $C_0$ at infinity in the direction normal to $Q$ via the aforementioned trivialization, whence $\op{push}(C) \cdot Q = 0$.
We also have the winding number bounds $\wind_z \leq \lfloor \cz^\perp(\ga_z)/2 \rfloor = 0$ for $z$ a positive puncture and  $\wind_z \geq \lceil \cz^\perp(\ga_z)/2 \rceil = 1$ for $z$ a negative puncture (c.f. \cite[\S3.5]{wendl_SFT_notes}).
Since this violates \eqref{eq:Siefring_type_eq}, we conclude that $C_0$ must lie in $Q$.
\end{proof}

Lastly, recall that in the above we assumed regularity for $\ovl{\calM}$, which in general can only hold in a virtual sense.
More precisely, we compute $\#\ovl{\calM}$ using virtual perturbations in $\R \times \bdy X$ and $\wh{X}$, which we can then extend trivially in the normal direction to obtain virtual perturbations in $\R \times \bdy X^\stab$ and $\wh{X}^\stab$ for the computation of $\#\ovl{\calM}^\stab$. 
With these choices \eqref{eq:Siefring_type_eq} remains valid, and the rest of the above proof carries over.
\end{proof}

Similarly, for the local tangency augmentation we have:
\begin{prop}\label{prop:aug_map_stab}

    Under the inclusion of \eqref{eq:Linf_ops_stab_filt_bar}, we can further arrange that the chain map
$\calF_{< S}\bar\chlin(X \smx B^2(S)) \ra \bar\K$
induced by the $\Li$ augmentation 
  $\aug_{X \smx B^2(S)}\lll \T^m p\rrr: \chlin(X \smx B^2(S)) \ra \K$ 
  restricts to the corresponding chain map $\calF_{< S}\bar\chlin(X) \ra \bar\K$ induced by $\aug_X\lll \T^m p\rrr: \calF_{< S}\chlin(X) \ra \K$.
\end{prop}

The proof is closely analogous to that of Proposition~\ref{prop:Linf_operations_stabilize}.
Here we take the local divisor in $X \smx B^2(S)$ to be of the form $D \times B^2(\eps)$ for small $\eps > 0$, where $D \subset X$ is the local divisor at $p \in X$.
We then appeal to the following analogue of Lemma~\ref{lem:stab_symp_curves_are_slice}, noting that a curve in $\wh{X}^\stab$ satisfying the constraint $\lll \T^m p \rrr$ cannot be disjoint from $\wh{X}^\stab$.
\begin{lemma}\label{lem:stab_aug_curves_are_slice}
Any $J_{X^\stab}$-holomorphic curve curve in $\wh{X}^\stab$ all of whose asymptotic Reeb orbits have action less than $C$ is either disjoint from $\wh{X} \times \{0\}$ or entirely contained in it.  
\end{lemma}

\begin{remark}
The formulation and proof of the analogous statement for the $\Li$ cobordism map induced by the stabilization $X \smx B^2(S) \hookrightarrow X' \smx B^2(S)$ induced by an exact symplectic embedding $X \hookrightarrow X'$ follow along very similar lines, and we leave the details to the reader.
\end{remark}

\subsection{Constraints via extra negative ends}\label{subsec:extra_negative_ends}

We now give an alternative formulation of the local tangency
constraint $\lll \T^mp\rrr$, defined by excising a skinny ellipsoid and counting curves with an extra negative end.
This has the advantage of putting the $\Li$ augmentations $\aug \lll \T^{m}p\rrr: \chlin(X) \rightarrow \Lamo$ into a unified SFT framework and sheds light on the nature of the local tangency constraint.

Let $(X^{2n},\omega)$ be a symplectic filling of $(Y,\alpha)$, and suppose we are interested in understanding the $\Li$ augmentation $\aug_{X,\lin} \lll \T^m p \rrr: \chlin(X) \rightarrow \Lamo$.
To see how the constraint $\lll \T^m p\rrr$ can be reformulated in terms of extra negative ends, let $U \subset X$ be a small neighborhood of $p$ which is symplectomorphic to the ball $B^{2n}(\eps)$ of area $\eps$.
By stretching the neck along the boundary of $U$, we find that 
$\aug_{X,\lin}\lll \T^k p \rrr$ is $\Li$ homotopic to the composition $\aug_{U,\lin}\lll \T^k p \rrr \circ \Phi_\lin$,
where $\Phi_\lin$ denotes the $\Li$ cobordism map $\Phi_\lin: \chlin(X) \rightarrow \chlin(U)$.
Here $\aug_{U,\lin} \lll \T^mp \rrr: \chlin(U) \rightarrow \Lamo$ is identified with the $\Li$ augmentation 
$\aug_{B^{2n}(\eps),\lin}\lll \T^m p \rrr: \cha(B^{2n}(\eps)) \rightarrow \Lamo$, and can be defined
either using Morse--Bott techniques as in the previous section or by slightly perturbing $B^{2n}(\eps)$ to $E(\eps+\delta_1,...,\eps+\delta_n)$, where $0 < \delta_1 < ... < \delta_n$ are small compared to $\eps$ and rationally independent.
The data associated to $B^{2n}(\eps)$ is essentially the same as that of the unit ball $B^{2n}$, up to rescaling action.
Recall that, after listing the Reeb orbits of $E(\eps+\delta_1,...,\eps+\delta_n)$ in order of increasing action, the $k$th orbit has Conley--Zehnder index $n- 1 + 2k$, and in particular the $\Li$ operations on $\chlin(B^{2n})$ are all trivial for degree parity reasons.
We will not endeavor to compute the $\Li$ augmentation $\aug_{B^{2n},\lin}\lll \T^m p\rrr$ here, although the computations for $\CP^2$ in \cite{enumerating} can be viewed as the special case where all inputs are the long simple Reeb orbit of the (slightly perturbed) four-ball.
Presently, the main point is that we can use this augmentation to recover the augmentation $\aug_{X,\lin}\lll \T^m p \rrr$ from the cobordism map $\Phi_\lin$ in a ``universal'' way.

However, to obtain a cleaner reformulation of the constraint $\lll \T^mp \rrr$, it is more convenient to replace the round neighborhood $U$ by a neighborhood $V$ which is symplectomorphic to a small skinny ellipsoid $E_\sk := E(\eps,L,...,L)$, where $\eps$ and $L$ are both small and we have $L \gg \eps$. The reason is that the $\Li$ augmentation
$\aug_{E_\sk,\lin}: \chlin(E_\sk) \rightarrow \Lamo$ is particularly simple. 
Indeed, let $\alpha_1$ denote the simple Reeb orbit of $\bdy E_\sk$ of action $\eps$, and let $\alpha_k$ denote its $k$-fold iterate for $k \in \Z_{> 0}$.
Note that all other Reeb orbits of $\bdy E_\sk$ have action at least $L$,
and for $L$ sufficiently large compared to $\eps$ they also have arbitrarily large Conley--Zehnder indices.
\begin{prop}\label{prop:compute_T_aug}
For $L$ sufficiently large relative to $\eps$, $\auglin^k\lll \T^m p\rrr (\alpha_{m_1},...,\alpha_{m_k})$ is equal to $T^{(m+1)\eps}$ if $k = 1$ and $m_1 = m + 1$, and $0$ otherwise.
\end{prop}
\begin{proof}
We find as in the proof of Proposition \ref{prop:aug_map_stab} that the compactified moduli space $\ovl{\calM}_{E_\sk}((\alpha_{m_1},...,\alpha_{m_k});\nil)\lll \T^m p \rrr$ reduces to $\ovl{\calM}_{B^2(\eps)}((\alpha_{m_1},...,\alpha_{m_k});\nil)\lll \T^m p \rrr$.
We can easily compute $\#\ovl{\calM}_{B^2(\eps)}((\alpha_{m_1},...,\alpha_{m_k});\nil)\lll \T^m p \rrr$ as a count of Hurwitz covers of $\CP^1$ with a branch point of order $m+1$ over $p \in B^2$ and branch points of orders $m_1,...,m_k$ over $\infty$.
In particular we must have $m_1 + ... + m_k = m+1$.
As an element of $\ovl{\calM}_{E_\sk}((\alpha_{m_1},...,\alpha_{m_k});\nil)\lll \T^m p \rrr$, the index of such a curve is computed to be $2k-2$, which means that $\auglin^k\lll \T^m p\rrr (\alpha_{m_1},...,\alpha_{m_k})$ vanishes for degree reasons unless $k=1$.
\end{proof}

The upshot is that the $\Li$ augmentation $\aug_{X,\lin}\lll \T^m p \rrr$ can be described (up to $\Li$ homotopy) in terms of the $\Li$ cobordism map $\Phi_\lin: \chlin(X) \rightarrow \chlin(V)$:
\begin{cor}
In the above setting we have
 \begin{align}
 \aug_{X,\lin}^k\lll \T^m p \rrr(-,...,-)= T^{(m+1)\eps}\cdot \langle \Phi_\lin^k(-,...,-), \alpha_{m+1}\rangle.
  \end{align} 
\end{cor}
\NI Here  for $\gamma_1,...,\gamma_k \in \chlin(X)$, $\langle \Phi_\lin^k(\gamma_1,...,\gamma_k),\alpha_{m+1}\rangle$ is the coefficient of the generator $\alpha_{m+1}$ in $\Phi_\lin^k(\gamma_1,...,\gamma_k) \in \chlin(V)$. Note that the extra factor $T^{(m+1)\eps}$ is present to compensate for the fact that the orbits of $V$ have not-quite-zero action.

\begin{remark}
In principle we can similarly apply neck stretching to reformulate the blowup augmentation $\aug_{X}\lll \underbrace{p,...,p}_b \rrr: \cha(Y) \rightarrow \Lamo$ and its generalizations in terms of curves in $X \setminus E_\sk$. 
Indeed, we find that the $\Li$ augmentation $\aug_X\lll \underbrace{p,...,p}_b \rrr$ is $\Li$ homotopic to the composition
$\aug_{E_\sk}\lll \underbrace{p,...,p}_b \rrr \circ \Phi$, where $\Phi: \cha(Y) \rightarrow \cha(E_\sk)$ denotes the $\Li$ homomorphism associated to the cobordism $X \setminus E_\sk$.
However, we note that there is no simple analogue of Proposition \ref{prop:compute_T_aug} in this setting. For example, consider the term $\aug^1_{E_\sk}\lll p,p\rrr(\alpha_2)$, where $\alpha_2$ denotes the two-fold cover of simple Reeb orbit of $\bdy E_\sk$ of minimal action. One can show using the relative adjunction formula and ECH writhe bounds \cite{hutchings2016beyond} that there are no somewhere injective planes in the one-point blowup of $E_\sk$ with top end asymptotic to $\alpha_2$. 
However, there is a two-level building representing such a term, with top level in $\R \times \bdy E_\sk$ consisting of a double cover of a trivial cylinder with top end asymptotic to $\alpha_2$ and two bottom ends asymptotic to $\alpha_1$, and with bottom level in $E_\sk$ consisting of two planes each asymptotic to $\alpha_1$.
The count of such buildings and hence the determination of the coefficient $\aug_{E_\sk}\lll p,p\rrr(\alpha_1,\alpha_1)$ depends on the particularities of the chosen virtual perturbations used, i.e. it is not completely canonical.
\end{remark}

\subsection{Multibranched local tangency constraints}\label{subsec:multibr_loc_tan}

Now suppose we have a symplectic filling $(X,\omega)$ of $(Y,\alpha)$, and fix integers $b_1,...,b_r \geq 1$ and $m_j^i \geq 1$ for $1 \leq i \leq r$, $1 \leq j \leq b_i$.
We can also define an $\Li$ augmentation
$$\aug\lll (\T^{m^1_1}p_1,...,\T^{m^1_{b_1}}p_1) ,..., (\T^{m^r_1}p_r,...,\T^{m^r_{b_r}}p_r) \rrr: \cha(Y) \rightarrow \Lamo$$
which counts curves ``multibranched'' tangency constraints at the points $p_1,...,p_r$. For example, the constraint $\lll \T^{m_1}p,...,\T^{m_b}p\rrr$ corresponds to curves with $b$ local branches passing through $p$, such that the $i$th local branch has tangency order $m_i$ for the local divisor $D$ for $i = 1,...,b$.
These augmentations have not yet appeared in the literature, although analogous counts for closed curves in closed symplectic four-manifolds are discussed in detail in \cite[\S2.3]{enumerating}. 

We can also put these counts into an SFT framework as follows.
Let $X'= X \setminus (V_1 \cup ... \cup V_r)$ be the resulting symplectic cobordism after removing $r$ pairwise disjoint neighborhoods $V_1,...,V_r \subset X$ each symplectomorphic to a small skinny ellipsoids as in $E_\sk$ above. 
Let $\Phi_{X'}: \cha(Y) \rightarrow \cha(\bdy V_1 \cup ... \cup \bdy V_r)$ denote the $\Li$ cobordism map associated to $X'$. Note that we have an identifcation $\cha(\bdy V_1 \cup ... \cup \bdy V_r) = \cha(\bdy V_1) \otimes ... \otimes \cha(\bdy V_r)$ as $\Lamo$-modules, with both $\Li$ structures trivial.
We define $\aug\lll (\T^{m^1_1}p_1,...,\T^{m^1_{b_1}}p_1) ,..., (\T^{m^r_1}p_r,...,\T^{m^r_{b_r}}p_r) \rrr(\gamma_1,...,\gamma_k)$ by 
\begin{align*}
T^{\calA(w_1)+...+\calA(w_r)}\mu_{\Gamma_1}...\mu_{\Gamma_r}\cdot \langle \Phi(\gamma_1,...,\gamma_k), \Gamma_1\otimes ... \otimes \Gamma_r \rangle,
\end{align*}
where $w\Gamma_i = \alpha_{m^i_1+1}...\alpha_{m^i_{b_i}+1} \in \cha(\bdy V_i)$ for $i = 1,...,r$.
Here the combinatorial factor $\mu_{\Gamma_1}...\mu_{\Gamma_r}$ plays the role of $b_1!...b_r!$ from the \S\ref{subsec:blowup} and has the effect of ordered the constrained marked points.
That is, we count rational curves in $X'$ with positive ends $\gamma_1,...,\gamma_k$ and $b_i$ specified negative ends on each $\bdy V_i$, and each negative end is a $m^i_j$-fold iterate of the simple short orbit $\alpha_1$.
As explained in \cite{enumerating}, in the case of a closed symplectic manifold $(M,\beta)$, by similarly counting curves in $M \setminus (V_1 \cup ... \cup V_r)$ with specified negative ends (and no positive ends), we define invariants
\begin{align*}
\gw_{M,A}\lll(\T^{m^1_1}p_1,...,\T^{m^1_{b_1}}p_1) ,..., (\T^{m^r_1}p_r,...,\T^{m^r_{b_r}}p_r)\rrr \in \Q.
\end{align*}

\sss

Suppose that $M = X^- \circledcirc X^+$ is a closed symplectic manifold given by concatenating a symplectic filling $X^-$ and a symplectic cap $X^+$ with common contact boundary $Y$.
Extending the discussion in \S\ref{sec:CL}, using geometric constraints contained in $X^-$, neck stretching along $Y$ decomposes closed curve invariants of $M$ into punctured curves in $X^+$ (which are encoded by the Cieliebak--Latschev element) and punctured curves in $X^-$ carrying the geometric constraints.
Namely, we have e.g.
\begin{align}
(\auglin)_*\lll \T^m p\rrr(\cllin) = \gw_M \lll \T^m p\rrr,
\end{align}
and more generally
\begin{align}
\aug_* \lll (\T^{m^1_1}p_1,...,\T^{m^1_{b_1}}p_1) ,..., (\T^{m^r_1}p_r,...,\T^{m^r_{b_r}}p_r) \rrr(\cl) \\= \gw_M \lll (\T^{m^1_1}p_1,...,\T^{m^1_{b_1}}p_1) ,..., (\T^{m^r_1}p_r,...,\T^{m^r_{b_r}}p_r) \rrr,
\end{align}
provided that there are no rigid closed curves in $X^-$ satisfying the same constraints (so that closed curves are forced to split nontrivially under neck stretching).

\subsection{Pushing points together}\label{subsec:ppt}

As observed in \cite{enumerating}, we can combine multibranched tangency constraints at several different points to produce various constraints at a single point. 
For simplicity, we restrict to the case of dimension four.
The starting point is the following heuristic, which appears in the work \cite{Gathmann_GW_of_blowups} of Gathmann. 
Given an index zero curve satisfying point constraints at $q$ and $p$, consider what happens as we vary the locations of the point constraints such that $q$ approaches $p$ along some tangent vector $v$ to $p$.
If the marked points in the domain stay separated during the deformation, the limit should be a curve with two marked points which both map to $p_1$, and we can interpret this as a curve with a node at $p$. On the other hand, if the marked points collide in the domain, the result is a curve with a single marked point mapping to $p$, where it is tangent to the complexification of $v$.
Gathmann makes this precise with the following formula for curves in $\CP^2$:
\begin{align}\label{eq:gathmann}
\gw_{\CP^2,A}\lll q,p,- \rrr = \gw_{\CP^2,A}\lll \T p,- \rrr + \gw_{\CP^2,A}\lll p,p,- \rrr,
\end{align}
where the symbol $-$ is a shorthand for any additional constraints not involving $q$ or $p$.
Note that the last term can also be written as $2\gw_{\wt{\CP}^2,A-\mathcal{E}}$.

Using the formulation from the previous subsection, we can make the process of pushing points together precise by surrounding two small skinny ellipsoidal neighborhoods $V_1,V_2$ by another such neighborhood $V$, and then stretching the neck along the boundary of $V$.
In the example above with two point constraints $q \in V_1$ and $p \in V_2$, we start with a curve with an $\alpha_1$ negative end on each of $\bdy V_1,\bdy V_2$.
After neck stretching, in the cobordism $V \setminus (V_1 \cup V_2)$ there are two possibilities:
\begin{itemize}
\item
a union of two cylinders, one with an $\alpha_1$ positive end on $\bdy V$ and a negative end on $\bdy V_1$ and one with an $\alpha_1$ positive end on $\bdy V$ and a negative end on $\bdy V_2$.
\item 
a pair of pants with one $\alpha_2$ positive end on $\bdy V$ and an $\alpha_1$ negative end on each of $\bdy V_1,\bdy V_2$.
\end{itemize}
Looking at the remaining part of the curve in $\CP^2 \setminus V$, the first situation corresponds precisely to $\gw_{\CP^2,A}\lll \T p \rrr$, while the second situation corresponds precisely to $\gw_{\CP^2,A}\lll p,p\rrr$.\footnote{Strictly speaking this decomposition only holds virtually, as on the level of $J$-holomorphic moduli spaces the pair of pants is actually represented by a two level building. Nevertheless, one can show using obstruction bundle gluing techniques that it behaves as a unique pair of pants for the purposes of \eqref{eq:gathmann}, and similarly for \eqref{eq:ppt} below; see \cite[\S4.2]{enumerating} for details.} 

In general, given constraints at two distinct points $q,p \in M$, we can push them together as above by surrounding $V_1 \cup V_2$ by $V$ and stretching the neck.
By studying curves in the symplectic cobordism $V \setminus (V_1 \cup V_2)$, we get the following outcome 
(the details  appear in \cite{enumerating}):
\begin{thm}[``pushing points together'']
If $(M^4,\beta)$ is a closed four-dimensional symplectic manifold, we have the following equality:
\begin{align}\label{eq:ppt}
&\gw_{M,A}\lll (\T^mq),(\T^{m_1}p,...,\T^{m_b}p),-\rrr \\&= \gw_{M,A}\lll(\T^mp,\T^{m_1}p,...,\T^{m_b}p),-\rrr + \sum_{i=1}^b\gw_{M,A}\lll(\T^{m_1}p,...,\T^{m_i+m}p,...\T^{m_b}p),-\rrr.
\end{align}
Similarly, if $(X^4,\omega)$ is a four-dimensional symplectic filling, we have the following $\Li$ homotopy:
\begin{align*}
&\aug\lll (\T^mq),(\T^{m_1}p,...,\T^{m_b}p),-\rrr \\ &\sim \aug\lll(\T^mp,\T^{m_1}p,...,\T^{m_b}p),-\rrr + \sum_{i=1}^b\aug\lll(\T^{m_1}p,...,\T^{m_i+m}p,...\T^{m_b}p),-\rrr.
\end{align*}
\end{thm}

Note that in particular we can reduce all constraints at different points to a positive linear combination of constraints at a single point.
Moreover, by applying this formula recursively in a somewhat clever way, we can reduce invariants with tangency constraints to linear combinations of invariants without tangency constraints (but possibly with constraints at multiple points).
A recursive algorithm for doing this is given in \cite{enumerating}.
For example, this technique can be applied to recursively compute the invariants $\gw_{\CP^2,d[L]}\lll \T^{3d-2}p\rrr \in \Q$ in terms of Gromov--Witten invariants of blowups of $\CP^2$ at finitely many points,
which in turn are readily computable (c.f. \cite{Gottsche-Pandharipande}). 
Putting $T_d := \gw_{\CP^2,d[L]}\lll \T^{3d-2}p\rrr$, the first few values we are:
\begin{align*}
T_1 = 1, \;\;\;\;\; T_2 = 1, \;\;\;\;\; T_3 = 4,\;\;\;\;\; T_4 = 26,\;\;\;\;\; T_5 = 217,\;\;\;\;\; T_6= 2110,
\end{align*}
etc.
Strictly speaking, in order to apply the upper bound in \S\ref{subsubsec:gw_upper_bound} we only need to know that these numbers are nonzero,
which follows from \cite[Corollary 4.1.3]{enumerating}.

\section{Construction of capacities}\label{sec:construction_of_capacities}

\subsection{The general construction}\label{subsec:the_general_construction}

We now construct symplectic capacities using the $\Li$ augmentations from the previous section. 
For notational simplicity we will assume that all domains under discussion are connected.\footnote{Note that if $X$ is say  disconnected symplectic filling, then there are for example several different $\Li$ augmentations $\aug\lll \T^m p\rrr\chlin(X) \rightarrow \Lamo$, depending on which component of $X$ the point $p$ lies in. The freedom to choose points in different components plays a role in the disjoint union lower bound discussed in \S\ref{subsubsec:disjoint_union_lb} below.}
Firstly, we have the ones based on the $\Li$ structure on linearized contact homology.
Given $(X,\omega)$ a symplectic filling of $(Y,\alpha)$,
we have $\Li$ augmentations 
$$\auglin\lll \T^m p\rrr: \chlin(X) \rightarrow \Lamo$$ for $m \geq 0$.
We assemble these into a single $\Li$ homomorphism
$$ \auglin \lll \TT p \rrr: \chlin(X) \rightarrow \Lamo[t],$$
with the $t^m$ coefficient given by $\auglin\lll \T^m p \rrr$.
For $\bb \in \ovl{S}\K[t] \subset \ovl{S}\Lamo[t]$, put
$$ \gapac_\bb^{\leq l}(X) := \inf \{ a \;:\; T^a \bb = \wh{\auglin}\lll \TT p \rrr(x)\text{ for some } x \in \bar_{\leq l}\chlin(X)\text{ with } \wh{\ell}(x) = 0\}.$$
Recall here that $\wh{\ell}$ is the differential on the bar complex $\bar\chlin(X)$.
If $X$ is a Liouville domain, it is easy to check that this definition is equivalent to the one given in the introduction.
We also define the simplified capacities
\begin{align*}
\gapac^{\leq l} \lll \T^m p\rrr(X) &:= \inf \{a \;:\; T^a = \pi_1\circ\wh{\auglin}\lll \T^m p\rrr(x)\text{ for some } x \in \bar_{\leq l}\chlin(X)\text{ with } \wh{\ell}(x) = 0\}\\
&= \inf_{\bb \in \ovl{S}\K[t] \;:\; \pi_1(\bb) = t^m} \gapac_{\bb}^{\leq l}(X),
\end{align*}
where $\pi_1$ denotes the projection $\ovl{S}\K[t] \rightarrow \K[t]$ to the word length one subspace.

\begin{remark}
If we are interested in obstructing symplectic embeddings between two symplectic ellipsoids, it turns out that word length refined capacities $\gapac_{\bb}^{\leq l}$ provide no extra data beyond the capacities $\gapac_{\bb}$.
Indeed, for $X$ an ellipsoid we have $\gapac_{\bb}^{\leq l}(X) = \gapac_{\bb}(X)$ if $\bb \in \ovl{S}_{\leq l}\K[t]$, and $\gapac_{\bb}^{\leq l}(X) = 0$ otherwise. This  follows from the fact that (working over $\K$) the map $\bar \chlin(X) \ra \bar \K$ induced by $\auglin\lll \T^m p\rrr$ is an isomorphism at chain level, and hence there is a unique element in $\bar \chlin(X)$ mapping to any given $\bb \in \ovl{S}\K[t]$.

On the other hand, in some cases the capacities $\gapac_{\bb}^{\leq l}$ do provide stronger obstructions.
As a simple example, if we consider exact symplectic embeddings $D^*S^1 \hookrightarrow B^2$, the capacities $\gapac_{\bb}$ give the volume obstruction, whereas the capacities $\gapac_{\bb}^{\leq 1}$ show that no such embedding exists (note that there $\gapac_\bb^{\leq 1}(D^*S^1) = \infty$ since there are no contractible Reeb orbits).
\end{remark}

We also define the capacities based on the $\Li$ structure on the full contact homology algebra as follows.
We assemble the $\Li$ augmentations
$$\aug \lll \T^{m_1}p,...,\T^{m_b}p\rrr: \cha(Y) \rightarrow \Lamo$$
into a single $\Li$ homomorphism
$$ \aug\lll \TT p,...,\TT p\rrr: \cha(Y) \rightarrow S\Lamo[t],$$
where the coefficient of $t^{m_1} \odot ... \odot t^{m_b}$ is given by 
$\aug\lll \T^{m_1}p,...,\T^{m_b}p\rrr$.
The corresponding map on bar complexes is
$$\wh{\aug}\lll \TT p,...,\TT p\rrr: \bar\cha(Y) \rightarrow \ovl{S}S\Lamo[t].$$
For $\bb \in \ovl{S}S\K[t] \subset \ovl{S}S\Lamo[t]$, we define the capacity
$$\rapac_\bb^{\leq l}\lll \TT p,...,\TT p\rrr(X) \in \R_{> 0}$$ 
by
\begin{align*}
\inf \{ a \;:\; T^a \bb = \wh{\aug}\lll \TT p,...,\TT p \rrr(x)\text{ for some } x \in \bar_{\leq l}\cha(Y)\text{ with } \wh{\ell}(x) = 0\}.
\end{align*}
We also define the simplified capacities 
\begin{align}\label{eq:rapac_leql}
\rapac^{\leq l}\lll \T^{m_1}p,...,\T^{m_b}p\rrr :=  \inf_{\bb \in \ovl{S}S\K[t] \;:\; \pi_1(\bb) = t^{m_1}\odot ...\odot t^{m_b}} \rapac^{\leq l}_{\bb}(X),
\end{align}
where now $\pi_1$ denotes the projection $\ovl{S}S\K[t] \rightarrow S\K[t]$.

\begin{remark}
We can also define capacities using the $\Li$ augmentations involving several points in the target space
$$ \aug\lll (\T^{m^1_1}p_1,...,\T^{m^1_{b_1}}p_1) ,..., (\T^{m^r_1}p_r,...,\T^{m^r_{b_r}}p_r) \rrr: \cha(Y) \rightarrow S\Lamo[t].$$
By the discussion in \S\ref{subsec:ppt}, in dimension four these capacities with $r > 1$ can all be related to ones with $r = 1$.
\end{remark}

\subsection{Properties}\label{subsec:properties}

\subsubsection{Monotonicity}

It follows from the basic functoriality properties discussed in \S\ref{subsec:rsft_formalism} and \S\ref{sec:geometric_constraints} that all of the above capacities are monotone with respect to exact symplectic embeddings.
Indeed, by working on the level of bar complexes this follows from simple diagram chases as explained in the introduction.
Less obviously, these capacities are all monotone with respect to non-exact symplectic embeddings, provided we omit the word length refinement ``$\leq l$''.
Indeed, let $X,X'$ be symplectic fillings of $Y,Y'$ respectively, and suppose we have a symplectic embedding $X \hookrightarrow X'$.
According to \S\ref{sec:CL}, there is a Maurer--Cartan element
$\cl \in \cha(Y)$, and an $\Li$ homomorphism
$\Phi: \cha(Y') \rightarrow \cha_{\cl}(Y).$
Let $\aug_{Y}\lll-\rrr: \cha(Y) \rightarrow \Lamo$ and $\aug_{Y'}\lll -\rrr:\cha(Y') \rightarrow \Lamo$
denote any one of the augmentations from \S\ref{sec:geometric_constraints}, and let $\rapac\lll-\rrr$ denote the corresponding simplified capacity as in \eqref{eq:rapac_leql} (with $l = \infty$). 
We can deform $\aug_{Y}\lll-\rrr$ by $\cl$ to define an $\Li$ augmentation
$$ \aug_{Y,\cl}\lll-\rrr: \cha_{\cl}(Y) \rightarrow \Lamo$$
such that the composition of $\Phi$ and $\aug_{Y,\cl}\lll - \rrr$ is $\Li$ homotopic to $\aug_{Y'}\lll-\rrr$.
Suppose that we have $\rapac\lll-\rrr(X') = a$,
so that we have $x \in \bar\cha(Y')$ satisfying $\ell(x) = 0$ and $T^a = \pi_1\wh{\aug_{Y'}}\lll-\rrr(x)$.
We then have 
$$ 0 = \wh{\Phi}\wh{\ell}(x) = \wh{\ell_\cl}\wh{\Phi}(x) = \wh{\ell}(\wh{\Phi}(x)\odot \exp(\cl))$$
and 
$$T^a = \pi_1\wh{\aug_{Y'}}\lll-\rrr(x) = \pi_1\wh{\aug_{Y,\cl}}\lll - \rrr \wh{\Phi}(x) = \pi_1\wh{\aug_Y}\lll-\rrr(\wh{\Phi}(x) \odot \exp(\cl)).$$
In other words, the element $\wh{\Phi}(x) \odot \exp(\cl)$ exhibits the upper bound $\rapac\lll-\rrr(X) \leq a$. 
From this non-exact monotonicity for the corresponding capacity $\rapac\lll - \rrr$ follows, but observe that $\wh{\Phi}(x) \odot \exp(\cl) \in \bar\cha(Y)$ need not lie in the same word-filtration level as $x \in \bar\cha(Y')$.
Similar reasoning applies to the capacities $\rapac_{\bb}$ and the linearized versions.

\subsubsection{Behavior under stabilization}\label{subsubsec:behavior_under_stab}

\begin{prop}\label{prop:stabilization_for_gapac}
For any Liouville domain $(X,\omega)$ satisfying the index positivity Assumption~\ref{assump:index_pos},
we have
$$\gapac_{\bb}^{\leq l}(X \times B^2(S)) = \gapac_{\bb}^{\leq l}(X)$$
for any $S \geq \gapac_{\bb}^{\leq l}(X)$.
\end{prop}
\NI Here $\gapac_{\bb}^{\leq l}(X \times B^2(S))$ is by definition the supremum of $\gapac_{\bb}^{\leq l}(X')$ over all Liouville domains $X'$ admitting an exact symplectic embedding into $X \times B^2(S)$.
\begin{proof}
It suffices to establish 
\begin{align}\label{eq:stab_for_smx}
\gapac_{\bb}^{\leq l}(X \smx B^2(S)) = \gapac_\bb^{\leq l}(X)
\end{align}
 for $X \smx B^2(S)$ any smoothing of $X \times B^2(S)$ as in \S\ref{subsec:stabilization}, since there are smoothings of this form which are arbitrarily $C^0$ close to $X \times B^2(S)$.
But \eqref{eq:stab_for_smx} follows immediately from Proposition~\ref{prop:Linf_operations_stabilize} and Proposition~\ref{prop:aug_map_stab}, which together show that the data used to define $\gapac_{\bb}^{\leq l}(X)$ and $\gapac_{\bb}^{\leq l}(X \smx B^2(S))$ coincide for inputs of total action less than $S$.
By the definition of $\gapac_\bb$ as an infimum, it is easy to see that this discrepancy is irrelevant if we have $S \geq \gapac_\bb^{\leq l}(X)$.
\end{proof}

\begin{remark}
We do not know whether there is a natural analogue of Proposition \ref{prop:stabilization_for_gapac} for the capacities $\rapac \lll p_1,...,p_r \rrr$ constructed by Hutchings.
Notice that if $X$ has dimension $2n$, adding a point constraint is a codimension $2n-2$ condition, which is dimension dependent.
For example, for $n = 2$ the $\Li$ augmentations $\auglin\lll \T^4 p \rrr$ 
and $\auglin\lll p_1,...,p_5\rrr$ have the same degree, whereas for $n = 3$ the augmentations $\auglin \lll\T^4 p \rrr$ and $\auglin \lll p_1,p_2,p_3\rrr$ have the same degree.
\end{remark}

\subsubsection{Upper bounds from closed curves}\label{subsubsec:gw_upper_bound}

Let $(X,\omega)$ be a symplectic filling of $(Y,\alpha)$, and assume that $H_1(Y;\Z) = H_2(Y;\Z) = H_2(X) = 0$.
Let $(M,\beta)$ a closed symplectic manifold, and suppose there is a symplectic embedding $X \hookrightarrow M$.
By the discussion in \S\ref{sec:CL}, for each $A \in H_2(M;\Z)$ we have an element
$\exp(\cl)_A \in \bar\cha(Y)$ satisfying $\wh{\ell}(\exp(\cl)_A) = 0$ and its image
$$\pi_1\wh{\aug}\lll (\T^{m^1_1}p_1,...,\T^{m^1_{b_1}}p_1) ,..., (\T^{m^r_1}p_r,...,\T^{m^r_{b_r}}p_r) \rrr (\exp(\cl)_A)$$
agrees with the corresponding closed curve count
$$T^{[\omega]\cdot A}\gw_{M,A}\lll(\T^{m^1_1}p_1,...,\T^{m^1_{b_1}}p_1) ,..., (\T^{m^r_1}p_r,...,\T^{m^r_{b_r}}p_r)\rrr \in \Q.$$
Provided that this latter count is nonzero, this immediately implies the upper bound
$$\rapac \lll(\T^{m^1_1}p_1,...,\T^{m^1_{b_1}}p_1) ,..., (\T^{m^r_1}p_r,...,\T^{m^r_{b_r}}p_r) \rrr(X) \leq [\omega] \cdot A.$$
Similarly, we have $\exp(\cllin)_A \in \bar\chlin(X)$ satisfying $\wh{\ell}(\exp(\cllin)_A) = 0$
and 
$$ \pi_1 \wh{\auglin}\lll \T^m p\rrr(\exp(\cllin)_A) =T^{[\omega]\cdot A} \gw_{M,A},$$
which gives the upper bound
$$\gapac \lll\T^m p\rrr(X) \leq [\omega]\cdot A$$
provided that $\gw_{M,A}$ is nonzero.

\sss

For example, in the case that $M = \CP^2(a)$, i.e. $\CP^2$ with its standard Fubini-Study symplectic form normalized so that a line has area $a$,
we have 
\begin{align*}
&\rapac\lll p_1,...,p_{3d-1}\rrr(X) \leq da\\
 &\gapac\lll \T^{3d-2}p\rrr(X) \leq da
 \end{align*}
 for any symplectic filling $X$ admitting a symplectic embedding $X \hooksymp \CP^2(a)$.
 Similarly, in the case $M = \CP^1(a) \times \CP^1(a')$, i.e. $\CP^1 \times \CP^1$ with the symplectic form chosen such that the $\CP^1 \times \{pt\}$ has area $a$ and $\{pt\} \times \CP^1$ has area $a'$,
 we have
 \begin{align*}
&\rapac\lll p_1,...,p_{2d+1}\rrr(X) \leq a' + da\\
 &\gapac\lll \T^{2d}p\rrr(X) \leq a' + da
 \end{align*}
whenever there is a symplectic embedding $X \hooksymp \CP^1(a) \times \CP^1(a')$.
Here we are using the fact that the local tangency Gromov--Witten invariants $\gw_{\CP^2,d[L]}\lll \T^{3d-2}p\rrr$ and $\gw_{\CP^1 \times \CP^1,d[L_1]+[L_2]}\lll \T^{2d}p\rrr$ are nontrivial for all $d \in \Z_{\geq 1}$ (see Corollary 4.1.3 and Example 5.1.4 in \cite{enumerating}), while nonvanishing of the standard Gromov--Witten invariants $\gw_{\CP^2,d[L]}\lll p_1,\dots,p_{3d-1}\rrr$ and $\gw_{\CP^1 \times \CP^1,d[L_1]+[L_2]}\lll p_1,\dots,p_{2d+1}\rrr$ is well-known.

\subsubsection{Disjoint union lower bound}\label{subsubsec:disjoint_union_lb}

As observed by Hutchings, the capacities $\rapac \lll p_1,...,p_r\rrr$ 
satisfy a disjoint union lower bound, similar to the case of ECH capacities.
Consider points $p_1,...,p_r \in X$ which are pairwise distinct.
Suppose that $X^{2n}_1,...,X^{2n}_l$, $l \geq 1$, are symplectic fillings, with a symplectic embedding of the disjoint union 
$$X_1 \sqcup ... \sqcup X_l \hookrightarrow X.$$
Put $Y = \bdy X$ and $Y_i = \bdy X_i$ for $i = 1,...,l$.
Since the $\Li$ augmentation $\aug\lll p_1,...,p_r\rrr: \cha(Y) \rightarrow \Lamo$ 
does not depend on the locations of the points $p_1,...,p_r \in X$ up to $\Li$ homotopy, we can distribute them as we like amongst the subdomains $X_1,...,X_l$.
Given a partition $r = j_1 + ... + j_l$ with $j_1,...,j_l \geq 1$, let us assume that exactly $j_i$ of the points lie in $X_i$. 
More specifically, let us suppose that we have a partition
$$\{p_1,...,p_r\} = \{p_1^1,...,p_{j_1}^1\} \sqcup ... \sqcup \{p_1^l,...,p^l_{j_l}\},$$
where $p^i_1,...,p^i_{j_i}$ denote the points lying in $X_i$.
By stretching the neck along $Y_1 \sqcup ... \sqcup Y_l$, we find that the following relation holds up to $\Li$ homotopy:
$$ \aug_{X}\lll p_1,...,p_r \rrr = \otimes_{i=1}^l \left(\aug_{X_i}\lll p^i_1,...,p^i_{j_i}\rrr\right) \circ \Phi,$$
where $\Phi: \cha(Y) \rightarrow \otimes_{i=1}^l \cha(Y_j)$ denotes the $\Li$ cobordism map associated to $X \setminus (X_1 \sqcup ... \sqcup X_l)$.
This relation implies the inequality
$$ \rapac\lll p_1,...,p_r \rrr(X) \geq \sum_{i=1}^l \rapac \lll p^i_1,...,p^i_{j_i}\rrr(X).$$
Since the points $p_1,...,p_r$ are all distinct (assuming that $X$ is connected), their precise locations are irrelevant, and we can introduce the shorthand
$$\rapac_r(X) := \rapac \lll p_1,...,p_r\rrr.$$
Maximizing the lower bound over all partitions, we have
$$\rapac_r (X) \geq \max_{j_1+...+j_l = r}\sum_{i=1}^l \rapac_{j_i}(X_i).$$

More generally, consider the capacity 
$$\rapac\lll (\T^{m^1_1}p_1,...,\T^{m^1_{b_1}}p_1) ,..., (\T^{m^r_1}p_r,...,\T^{m^r_{b_r}}p_r) \rrr(X).$$
By distributing the points $p_1,...,p_r$ amongst the subdomains $X_1,...,X_l$ as above, we get a lower bound of the form
$$ \sum_{i=1}^l  \rapac\lll (\T^{-}p^i_1,...,\T^{-}p^i_1),...,(\T^{-}p^i_{j_i},...,\T^{-}p^i_{j_i})\rrr(X_i),$$
where $(\T^{-}p^i_j,...,\T^{-}p^i_j)$ corresponds to $(\T^{m^c_1}p_c,...,\T^{m^c_{b_c}}p_c)$ if $p^i_j$ is the $c$th point.

\sss

Observe that in particular any ball packing of $X^{2n}$ induces a lower bound for the capacities of $X$.
For example, in the case $2n = 4$, this lower bound can be made explicit using the computations for $B^4$ given in \S\ref{subsec:first_computations} below.
Namely, suppose that $X^4$ admits a packing by balls $B^4(a_1),...,B^4(a_l)$ with areas $a_1 \geq ... \geq a_l$. Then we get the following lower bounds:
\begin{align*}
&\gapac \lll \T^{k-1}p \rrr(X) \geq  a_1\lceil \tfrac{k+1}{3}\rceil\\
&\rapac_k(X) \geq \max_{j_1 + ... + j_a = k}\sum_{i=1}^l a_i \lceil \tfrac{j_i+1}{3}\rceil\\
&\rapac \lll \underbrace{p,...,p}_k \rrr(X) \geq k a_1.
\end{align*}
Notice that these three capacities all count curves satisfying constraints of the same codimension, but their lower bounds get progressively stronger.

As an important special case, the ellipsoid $E(1,x)$ for $x$ rational admits a canonical ball packing $\sqcup_{i=1}^l B^4(a_i) \hookrightarrow E(1,x)$ with $a_1 \geq ... \geq a_l$ related to the continued fraction expansion of $x$ (see \cite{Schlenk_old_and_new} for details). 
This decomposition plays a central role in the solution of the four-dimensional ellipsoid embedding problem.
For example, for $x = 55/8$ we have a packing of $E(1,55/8)$ by $6$ balls of area $1$, $1$ ball of area $7/8$, and $7$ balls of area $1/8$.

\subsection{First computations and applications}\label{subsec:first_computations}

We end this paper by giving the computations stated in the introduction. 

\begin{prop}\label{prop:rapac_of_ball}
For $r \geq 1$, we have $\rapac\lll p_1,...,p_r \rrr(B^4)= \lceil \tfrac{r+1}{3}\rceil$.
\end{prop}
Note that $B^4$ actually has Morse--Bott Reeb dynamics, with simple Reeb orbits in $\bdy B^4$ corresponding to fibers of the Hopf fibration $S^3 \rightarrow \CP^1$. Alternatively, we make the Reeb orbits nondegenerate by slightly perturbing to an irrational ellipsoid $E(1,1+\delta)$ for $\delta > 0$ arbitrarily small.
After such a perturbation, there are precisely two simple Reeb orbits in $\bdy E(1,1+\delta)$, with actions $1$ and $1+\delta$ respectively.
Let $\alpha_1,\alpha_2,\alpha_3,...$ denote the iterates of the shorter simple Reeb orbit, and let $\beta_1,\beta_2,\beta_3,...$ denote the iterates of the longer simple Reeb orbit.
The orbits $\alpha_1,\alpha_2,\alpha_3,...$ have actions $1,2,3,...$ and Conley--Zehnder indices $3,7,11,...$,
while the orbits $\beta_1,\beta_2,\beta_3,...$ have actions $1+\delta,2+2\delta,3+3\delta,...$ and Conley--Zehnder indices $5,9,13,...$.
Recall that we are computing Conley--Zehnder indices with respect to a global trivialization $\tau_{\op{ex}}$ of the contact vector bundle (c.f. \cite[\S3.2]{enumerating}).

\begin{proof}
The lower bound $\rapac \lll p_1,...,p_r \rrr(B^4) \geq \lceil \tfrac{r+1}{3}\rceil$ follows easily by considering the actions and indices of Reeb orbits.
For any Reeb orbit $\gamma$ we have $\cz(\gamma) \leq 1 + 4\calA(\gamma)$.
The capacity $\rapac \lll p_1,...,p_r\rrr (B^4)$ is represented by a (possibly disconnected) rigid rational curve in $B^4$ with $r$ point constraints.
Consider a connected such curve with positive ends $\gamma_1,...,\gamma_k$. 
Note that $k$ is at most $\calA = \sum_{i=1}^k \calA(\gamma_i)$.
Index considerations then give
\begin{align*}
 0 &= (2-3)(2-k) + \sum_{i=1}^k \cz(\gamma_i) - 2r\\
 &\leq 6\calA -2 - 2r,
 \end{align*}
i.e. $\calA \geq \tfrac{r+1}{3}$.
Since $\calA$ is an integer (up to a factor of $\delta$), this gives 
$\calA \geq \lceil \tfrac{r+1}{3}\rceil$. 
The analysis for a disconnected curve is similar.

As for the upper bound $\rapac \lll p_1,...,p_r \rrr(B^4) \leq \lceil \tfrac{r+1}{3}\rceil$, since $\gw_{\CP^2,d[L]} \lll p_1,...,p_{3d-1} \rrr \neq 0$ for all $d \in \Z_{> 0}$, the upper bound from \S\ref{subsubsec:gw_upper_bound} reads 
$$ \rapac \lll p_1,...,p_{3d-1}\rrr(B^4) \leq d.$$
The upper bound for all other $r$ follows from the observation that $\rapac\lll p_1,...,p_r\rrr$ is nondecreasing with $r$
(c.f. Remark \ref{rmk:another_defn_of_point_aug}).
\end{proof}

\begin{prop}\label{prop:gapac_for_ball}
We have $\gapac\lll \T^{m-1} p\rrr(B^4) = \lceil \tfrac{m+1}{3}\rceil$
for all $m \geq 1$ congruent to $2$ mod $3$.
\end{prop}
\begin{proof}
The lower bound $\gapac \lll \T^{m-1}p\rrr(B^4) \geq \lceil \tfrac{m+1}{3}\rceil$ follows exactly as in the proof of Prop \ref{prop:rapac_of_ball}.
Similarly, if $m = 3d-1$ for some $d \in \Z_{> 0}$, $\gapac \lll \T^{m-1}p\rrr(B^4)$ follows from the upper bound from \S\ref{subsubsec:gw_upper_bound}, using the fact that $B^4$ symplectically embeds in $\CP^2$ and we have $\gw_{\CP^2,d[L]}\lll \T^{3d-2}p\rrr \neq 0$ for all $d \geq 1$ by \cite[Corollary 4.1.3]{enumerating}
\end{proof}

\begin{remark}
We expect the above computation to be valid for all $m \in \Z_{> 0}$, but this does not follow directly from the proof given, since it is not clear from our definition whether 
$\gapac \lll \T^m p\rrr$ is nondecreasing with $m$. 
\end{remark}

\begin{prop}\label{prop:ellipsoid_computation}
For $1 \leq m \leq x$, we have $\gapac \lll \T^{m-1}p\rrr(E(1,x)) = m$.
\end{prop}

\begin{proof}
We can assume without loss of generality that $x$ is irrational.
As before let $\alpha_1,\alpha_2,\alpha_3,...$ denote the iterates of the short simple Reeb orbit and $\beta_1,\beta_2,\beta_3,...$ denote the iterates of the long simple Reeb orbit.
The Reeb orbit $\alpha_1$ in $\bdy E(1,x)$ of minimal action has $\cz (\alpha_1) = 3$, and it bounds a unique holomorphic plane in $E(1,x)$ passing through $p$.
There is a unique $m$-fold cover of this plane with a branch point at the pre-image of $p$ and one positive end $\alpha_m$.
It then satisfies a $\T^{m-1}p$ condition, and it gives the upper bound 
$\gapac\lll \T^{m-1}\rrr (E(1,x)) \leq m$.

As for the lower bound $\gapac\lll \T^{m-1}\rrr(E(1,x)) \geq m$, note that any representative curve with energy less than $m$ would necessarily have at least two positive ends by index considerations. 
After replacing the tangency constraint with the skinny ellipsoidal constraint as in \S\ref{subsec:extra_negative_ends}, it is easy to rule out 
such curves using the relative adjunction formula and ECH writhe bounds \cite{Hlect}. See \cite[Lemma 4.2.3]{enumerating} for details.
\end{proof}

\begin{prop}\label{prop:gapac_poly_comp}
We have $\gapac \lll \T^{m-1}p\rrr(P(1,x)) = \min(m,x+\lceil \tfrac{m-1}{2}\rceil)$ for all odd $m \geq 1$.
\end{prop}
\begin{proof}
The upper bound $\gapac \lll \T^{m-1}p\rrr(P(1,x)) \leq m$ follows from Proposition \ref{prop:ellipsoid_computation}, using the fact that $P(1,x)$ symplectically embeds into $E(1+\delta,x')$ for all $\delta > 0$ and $x'$ sufficiently large.
For $m = 2d+1$ with $ed \geq 0$, 
the upper bound $\gapac \lll \T^{2d}p\rrr(P(1,x)) \leq x + d$ follows
from the symplectic embedding of $P(1,x)$ into $\CP^1(1) \times \CP^1(x)$,
together with the fact that the count of curves in $\CP^1 \times \CP^1$ in bidegree $(d,1)$ and full codimension tangency constraint at a point is nontrivial, i.e. $\gw_{\CP^1 \times \CP^1,A = d[D_1] + [D_2]}\lll \T^{2d}p\rrr \neq 0$ (see \cite[Example 5.1.4]{enumerating}).

As for the necessary lower bounds, consider a curve or building $u$ as in the definition of $\gapac\lll \T^{m-1}p\rrr(P(1,x))\rrr$.
As in \cite[\S 5.3]{hutchings2016beyond}, we can find a $C^0$-small perturbation $\wt{P}(1,x)$ of $P(1,x)$ such that in an arbitrarily large action window the Reeb orbits are of the following form:
\begin{enumerate}
\item $\alpha_{i,j}$ for $(i,j) \in \Z_{\geq 1}^2$ with $\cz(\alpha_{i,j}) = 2i + 2j$ and $\calA(\alpha_{i,j}) \approx i + jx$
\item $\beta_{i,j}$ for $(i,j) \in \Z_{\geq 0}^2 \setminus \{(0,0)\}$ with $\cz(\beta_{i,j}) = 1 + 2i + 2j$ and $\calA(\beta_{i,j}) \approx i + jx$.
\end{enumerate}
Here the pair of orbits $\alpha_{i,j},\beta_{i,j}$ is roughly related to the fact that the strict product of an orbit in $\bdy B^2(1)$ and an orbit in $\bdy B^2(x)$ gives rise to a Morse--Bott circle of Reeb orbits, which then perturbs to two Reeb orbits with Conley--Zehnder indices differing by $1$.
Meanwhile, $\beta_{i,0}$ and $\beta_{0,i}$ roughly correspond to orbits of the form $\bdy B^2(1) \times \{0\}$ and $\{0\} \times \bdy B^2(x)$  respectively in $\bdy \wt{P}(1,x)$.
See also \cite[\S5.3]{chscI} for more details.

Suppose first that at least one of the positive ends of $u$ is not of the form $\beta_{i,0}$ for some $i$.
It then follows from straightforward action and index considerations that the energy of $u$ is at least $x + \lceil\tfrac{m-1}{2}\rceil$.
For example, suppose that $u$ has positive ends $\beta_{0,s},\beta_{s_1,0},...,\beta_{s_k,0}$ for $k,s_1,...,s_k \geq 1$.
We then have
$$0 = \ind(u) = 2s + 2k + 2\sum_{i=1}^k s_i -  2m,$$
i.e. $\sum_{i=1}^k s_i = m -s - k$.
Since $s_i \geq 1$ for all $i$, we have
$$\sum_{i=1}^k s_i \geq m - s - \sum_{i=1}^k s_i,$$
and hence $\sum_{i=1}^k s_i \geq \tfrac{m-s}{2}$.
This gives
\begin{align*}
\calA(u) &= sx + \sum_{i=1}^ks_i\\
&\geq x + s-1 + \tfrac{m-s}{2}\\
&\geq x + \lceil \tfrac{m-1}{2}\rceil
\end{align*}
as desired.
The other cases can be handled similarly.

Otherwise, we can assume that the positive ends of $u$ are of the form 
$\Gamma^+ = (\beta_{s_1,0},...,\beta_{s_k,0})$ for some $k,s_1,...,s_k \geq 1$, and it suffices to show that the energy $\sum_{i=1}^k s_i$ of $u$ is at least $m$.
Arguing as in Addendum~\ref{addendum:chscI_comps}, we can arrange that any such curve must be contained in the two-dimensional slice corresponding to the symplectic completion of $B^2(1) \times \{0\}$.
But then $u$ must be a Hurwitz cover of this slice, and in particular we must have $m = \sum_{i=1}^k s_i$.
\end{proof}

\sss

As mentioned in the introduction, the above computations of $\gapac\lll \T^{m-1}p\rrr$ give various new obstructions for stabilized embedding problems. 
However, these capacities do not recover many of the optimal obstructions underlying Theorem \ref{thm:stab_ell_known_results}, and hence they do not generally give optimal obstructions to Question
\ref{question:stab_ell}.
On the other hand, we now explain why the more general capacities $\gapac_{\bb}(E(1,x))$ for $\bb \in \ovl{S}\K[t]$ do recover the obstructions of Theorem \ref{thm:stab_ell_known_results}, and why we expect Conjecture \ref{conj:complete_obstructions} to hold (see also Addendum \ref{addendum:chscI_comps}).
To keep the exposition concrete, we consider the symplectic embedding problem
$$ E(1,13/2) \times \C^{n-2} \hookrightarrow B^4(c) \times \C^{n-2},$$
and for convenience we work with the increasing action filtrations over $\K$ as in the introduction.

According to \cite{CG-Hind_products}, the optimal obstruction is $c \geq \tfrac{13}{5}$.
By constrast, the Ekeland--Hofer capacities give the weaker obstruction $c \geq 2$,
and it can be shown that the capacities $\gapac\lll \T^{m-1}p\rrr$ also only give $c \geq 2$ (see \cite{enumerating2}).
After slightly perturbing $E(1,13/2)$ to $E(1,13/2+\delta)$ for $\delta > 0$ sufficiently small, let $\alpha_1$ and $\beta_1$ denote the simple Reeb orbits of actions $1$ and $13/2+\delta$ respectively.
Including iterates, in order of increasing action the orbits are 
$\alpha_1,\alpha_2,\alpha_3,\alpha_4,\alpha_5,\alpha_6,\beta_1,\alpha_7,...$, with Conley--Zehnder indices $3,5,7,9,11,13,15,17,...$, etc.
Similarly, after perturbing $B^4(c)$ to $E(c,c+\delta')$,
let $\alpha_1'$ and $\beta_1'$ denote the primitive Reeb orbits of actions $1$ and $1+\delta'$ respectively, and including iterates we have Reeb orbits $\alpha_1',\beta_1',\alpha_2',\beta_2',\alpha_3',\beta_3',...$, etc.

Recall that the differentials on $\bar\chlin(E(c,c+\delta'))$ and $\bar\chlin(E(1,13/2+\delta))$ are both trivial for degree parity reasons, and hence passing to homology has no effect. We also have that the cobordism map 
$\bar\chlin(E(c,c+\delta')) \rightarrow \bar\chlin(E(1,13/2+\delta))$ is an isomorphism (with inverse given by the cobordism in the reverse direction up to scaling).
The maps $\wh{\aug}\lll \TT p\rrr: \bar\chlin(E(c,c+\delta'))\rightarrow \ovl{S}\K[t]$ and $\wh{\aug}\lll \TT p \rrr: \bar\chlin(E(1,13/2+\delta)) \rightarrow \ovl{S}\K[t]$ are also isomorphisms, using the identification of $\ovl{S}\K[t]$ with $\bar\chlin(E_\sk)$.
In particular, for any given $\bb \in \ovl{S}\K[t]$, there is a unique element $w \in \bar\chlin(E(1,13/2+\delta))$ such that $\wh{\aug}\lll \TT p\rrr(w) = \bb$, and the capacity $\gapac_{\bb}(E(1,13/2+\delta))$ is simply the action of $w$.

The Cristofaro-Gardiner--Hind obstruction comes from the existence
of a curve $C$ in the complementary cobordism $E(c,c+\delta') \setminus E(1,13/2+\delta)$ with five positive ends each asymptotic to $\beta_1'$, and one negative end asymptotic to $\alpha_{13}$ (see \cite{CG-Hind_products}).
Such a curve has Fredholm index zero, and Stokes' theorem gives $5c \geq 13$. 
The word $\beta_l \odot \beta_l \odot \beta_l \odot \beta_l \odot \beta_l$ defines an element in $\bar\chlin(E(c,c+\delta'))$,
and its image $y \in \ovl{S}\chlin(E(1,13/2+\delta))$ is $\alpha_s^{13}$ 
plus a sum of elements of word length at least two.
Let $\bb$ be the image of $\beta_l \odot \beta_l \odot \beta_l \odot \beta_l \odot \beta_l$ under the map
$\wh{\aug}\lll \TT p\rrr: \bar\chlin(E(c,c+\delta')) \rightarrow \ovl{S}\K[t]$.
Note that, by the above observation about vanishing differentials, $y \in \ovl{S}\chlin(E(1,13/2+\delta))$ is the unique element mapping to $\bb \in \ovl{S}\K[t]$, so we have $\gapac_{\bb}(E(1,13/2)) = \calA(y) \geq 13$ (this relies on the existence of $C$).
We then have $\gapac_{\bb}(B^4(c)) = 5c$ and 
$\gapac_{\bb}(E(1,13/2)) \geq 13$, and these together recover the optimal obstruction $c \geq \tfrac{13}{5}$.
A similar argument shows that any obstruction coming from the existence of a somewhere injective rational curve with one negative end should be visible to $\gapac_{\bb}$ for some $\bb \in \ovl{S}\K[t]$.

\bibliographystyle{plain}
\bibliography{HSC_biblio}

\end{document}